\documentclass{yorkThesis}
\usepackage[colorlinks=true, linkcolor=blue, citecolor=blue, pagebackref=true]{hyperref}

\usepackage[backrefs, alphabetic]{amsrefs} \usepackage{color}

\usepackage[pagebackref=true]{hyperref}

\usepackage{mathtools} \usepackage{etoolbox}
\patchcmd{\section}{\scshape}{\bfseries}{}{} \makeatletter

\usepackage{setspace}
\spacing{1.5}
\usepackage{fancyhdr}
\renewcommand{\chaptermark}[1]{\markboth{#1}{}}
\usepackage{amsthm}
\usepackage{amssymb}

\theoremstyle{plain} \newtheorem{theorem}{Theorem}[chapter]
\theoremstyle{plain} \newtheorem{lemma}[theorem]{Lemma}
\theoremstyle{plain} 
\theoremstyle{plain} \newtheorem{corollary}[theorem]{Corollary}
\theoremstyle{plain} \newtheorem{Conjecture}[theorem]{Conjecture}
\theoremstyle{definition} 
\theoremstyle{definition} \newtheorem*{definition*}{Definition}
\theoremstyle{plain} \newtheorem*{theorem*}{Theorem}
\theoremstyle{definition} \newtheorem{remark}[theorem]{Remark}
\theoremstyle{definition} \newtheorem*{remark*}{Remark}

\usepackage{color}
\usepackage[usenames,dvipsnames]{xcolor}

\newcommand{\NN}{\mathbb{N}} \newcommand{\RR}{\mathbb{R}}
\newcommand{\QQ}{\mathbb{Q}} \newcommand{\ZZ}{\mathbb{Z}}
\newcommand{\TT}{\mathbb{T}} \newcommand{\cQ}{\mathcal{Q}}
\newcommand{\cD}{\mathcal{D}} \newcommand{\cS}{\mathcal{S}}
\newcommand{\cH}{\mathcal{H}} \newcommand{\cR}{\mathcal{R}}
 
\newcommand{\bp}{\boldsymbol{p}} \newcommand{\bx}{\boldsymbol{x}}
 
\newcommand{\by}{\boldsymbol{y}} \newcommand{\bz}{\boldsymbol{z}}
\newcommand{\bq}{\boldsymbol{q}} \newcommand{\0}{{\boldsymbol{0}}}
\newcommand{\br}{\boldsymbol{r}} \newcommand{\bs}{\boldsymbol{s}}
\newcommand{\bi}{\boldsymbol{i}} \newcommand{\bv}{\boldsymbol{v}}
\newcommand{\balpha}{\boldsymbol{\alpha}} \newcommand{\bbeta}{\boldsymbol{\beta}}
\newcommand{\bgamma}{\boldsymbol{\gamma}} 
\newcommand{\bad}{\mathbf{Bad}}
\newcommand{\I}{{\rm I}}
\newcommand{\F}{{\rm F}}
\newcommand{\sgn}{\operatorname{sgn}}
\DeclareMathOperator{\SL}{SL}

\def\eps{\varepsilon}

\providecommand{\ceil}[1]{\lceil#1\rceil}
\providecommand{\abs}[1]{\lvert#1\rvert}
\providecommand{\Abs}[1]{\left|#1\right|}
\providecommand{\norm}[1]{\lVert#1\rVert}

\providecommand{\inner}[2]{\langle#1,#2\rangle}

\title{Simultaneous Diophantine approximation on affine subspaces and Dirichlet improvability} \author{Fabian~S{\"u}ess}
\dept{Mathematics}
\supervisor{Sanju Velani, Evgeniy Zorin}
\submitdate{July 2017}

\numberwithin{equation}{chapter}
\numberwithin{section}{chapter}
\numberwithin{subsection}{section}

\begin{document}

\begin{titlePage}

\end{titlePage}


\clearpage
\thispagestyle{plain}
\phantom{a}
\vfill
\newpage
\vfill

\abstract{
\addcontentsline{toc}{section}{Abstract}
We show that affine coordinate subspaces of dimension at least two in Euclidean
space are of Khintchine type for divergence. For affine coordinate subspaces of dimension
one, we prove a result which depends on the dual Diophantine type of the base point of
the subspace. These results provide evidence for the conjecture that all affine subspaces of
Euclidean space are of Khintchine type for divergence. We also prove a partial analogue regarding the Hausdorff measure theory.

Furthermore, we obtain various results relating weighted Diophantine approximation and Dirichlet improvability. In particular, we show that weighted badly approximable vectors are weighted Dirichlet improvable, thus generalising a result by Davenport and Schmidt. We also provide a relation between non-singularity and twisted inhomogeneous approximation. This extends a result of Shapira to the weighted case.}

\newpage
\addcontentsline{toc}{section}{Contents}
\tableofcontents

\newpage
\addcontentsline{toc}{section}{Acknowledgements}
\afterPreface{
Now to the sentimental bits. Mathematics is an incredible success story within the history of humankind, being built on the ideas of (almost) uncountably many curious minds and spanning millennia as well as continents. I have been lucky to get a chance to add some of the most negligible bits to this story. However, even this smallest of contributions would not have been possible without the help of many people.

First I would like to thank Sanju and Evgeniy, for always keeping an open ear and an open mind, and for letting me work at my own speed. I have never been fast, but I have always tried to be thorough. A big thanks also to my co-authors Felipe and David, without whom this thesis would not look the same. Many others have led me to the path of mathematics or helped keeping me on it. I will never know if I could have made it through my undergraduate degree without the company of Lisa and Salome.

I will always be grateful to my parents and my sister for letting me grow up in an environment where everything was possible, for supporting me in any way they could and for being welcoming and surprisingly excited each time I came home to visit.

York has become a second home to me and this is mainly due to the people I have met here, be they colleagues, housemates or just other fellow lost souls. These include, but are not limited to: Spiros, James, Oliver \& Ellie, Ben, all the Daves, Demi (a special thanks for reading through this thesis), Mirjam, Vicky, the Italians, the pool \& snooker lot and Derek.

Whenever I went back to Switzerland, I tried (and failed) to visit all my dear old friends. Out of all those I would like to specially mention Paddy, Pagi and Janine. Thanks for being there through all the highs and lows.

Last, but by no means least, I would like to thank Henna.
}
\newpage
\thispagestyle{plain}
\vspace*{10ex}
\noindent
\textit{``And the mercy seat is waiting\\ And I think my head is burning\\ And in a way I'm yearning\\ To be done with all this measuring of truth\\ An eye for an eye\\ And a tooth for a tooth\\ And anyway there was no proof\\ Nor a motive why"\\[2ex]}
- Nick Cave, 1988

\chapter{Introduction}\label{Ch:Introduction}

The most fundamental problem in Diophantine approximation is the characterisation of points in Euclidean space $\RR^n$ as to how well they can be approximated by rational points. Obviously, the set $\QQ^n$ is dense in $\RR^n$ and so for a given $\balpha=(\alpha_1,\dots,\alpha_n)$ in $\RR^n$ we can find infinitely many rational points contained in an arbitrarily small ball around $\balpha$. On the other hand, this requires considering rationals with arbitrarily large denominators, which gives rise to the idea of relating the quality of approximation to the size of the denominator. More formally, given $\balpha\in\RR^n$, we are looking for solutions $q\in \NN$ to the inequality
\begin{equation}\label{Eqn:Psi}
	\parallel q\balpha\parallel < \psi(q),
\end{equation}
where
\begin{equation*}
	\parallel \bx \parallel = \min\limits_{\bz\in\ZZ^n}\left\{|\bx-\bz|\right\} = \min\limits_{\bz\in\ZZ^n}\left\{\max\limits_{1\leq j\leq n}\{|x_j-z_j|\}\right\}
\end{equation*}
denotes the distance from a point $\bx\in\RR^n$ to the nearest integer $\bz=(z_1,\dots,z_n)\in\ZZ^n$ and $\psi:\RR\rightarrow\RR$ is a positive real-valued function. It is readily seen that for any $\balpha\in\RR^n\backslash\QQ^n$ and $q\in \NN$, we can find $\bp=(p_1,\dots,p_n)\in \ZZ^n$ with 
\begin{equation*}
	\left|\alpha_j-\frac{p_j}{q}\right|<\frac{1}{2q},\quad (1\leq j\leq n)
\end{equation*}
and so any $q\in \NN$ is a solution to \eqref{Eqn:Psi} if $\psi$ is taken to be constant to $1/2$. Hence, we will want to consider functions $\psi$ which tend to zero with growing $q$. It is also worth noting that the case of $\balpha=\boldsymbol{a}/b\in\QQ^n$ with $\boldsymbol{a}\in\ZZ^n$ and $b\in\NN$ is of little interest as $\parallel q\balpha\parallel$ is equal to zero when $q$ is a multiple of $b$ and bounded from below by $1/b$, otherwise. It follows that we will restrict our attention to irrational points.

\section[Basic metric properties of simultaneous Diophantine approximation]{Basic metric properties of simultaneous\\ Diophantine approximation}\label{Sec:Basic}

The first fundamental result in Diophantine approximation was proved by Dirichlet and is a consequence of the pigeonhole principle, a rather simple, but powerful concept.

\theoremstyle{plain} \newtheorem*{pigeon}{Pigeonhole Principle}

\begin{pigeon}
	If m objects are placed into n boxes, where $m>n$, then at least one of the boxes contains at least two objects.
\end{pigeon}

\begin{theorem}[Dirichlet, 1842]
	For any $\alpha\in\RR$ and $Q\in\NN$, there exist integers $p$ and $q$, such that
	\begin{equation*}
		\left|\alpha-\frac{p}{q}\right|<\frac{1}{qQ} \quad \text{ and } \quad 1\leq q\leq Q.
	\end{equation*}
\end{theorem} 
\vspace{2ex}
\begin{proof}
	Let $\lfloor x\rfloor =\max\{z\in\ZZ:z\leq x\}$ and $\{x\}=x-\lfloor x\rfloor$ denote the integer and fractional part of a real number $x$, respectively. For any $x\in\RR$, we have $0\leq\{x\}<1$ and so the $Q+1$ numbers $\{0\alpha\},\{\alpha\},\{2\alpha\},\dots,\{Q\alpha\}$ are all contained in the half-open unit interval $[0,1)$. We can divide this interval into $Q$ disjoint subintervals of the form 
	\begin{equation*}
		[u/Q,(u+1)/Q),\quad u\in\{0,1,\dots,Q-1\}.
	\end{equation*}		
	Hence, by the Pigeonhole Principle, one of these intervals contains two points $\{q_1\alpha\},\ \{q_2\alpha\}$ with $q_1< q_2\in \{0,1,\dots,Q\}$ and it follows that
	\begin{equation*}
		|\{q_2\alpha\}-\{q_1\alpha\}|<\frac{1}{Q}.
	\end{equation*}
	As $\{q_k\alpha\}=q_k\alpha-p_k$ with $p_k=\lfloor q_k\alpha\rfloor$ for $k=1,2$, this implies
	\begin{equation*}
		|\{q_2\alpha\}-\{q_1\alpha\}|=|(q_2\alpha-p_2)-(q_1\alpha-p_1)|=|(q_2-q_1)\alpha-(p_2-p_1)|
	\end{equation*}
	and so letting $q=q_2-q_1$ and $p=p_2-p_1$ we have found integers $p$ and $q$, with $1\leq q\leq Q$, satisfying
	\begin{equation*}
		|q\alpha-p|<\frac{1}{Q},
	\end{equation*}
	or, in other words,
	\begin{equation*}
		\left|\alpha-\frac{p}{q}\right|<\frac{1}{qQ}.\\ \qedhere
	\end{equation*}
\end{proof}
\vspace{1ex}
There is a higher-dimensional analogue to Dirichlet's Theorem concerning the approximation of points $\balpha\in\RR^n$ by rational points $(\bp/q)=(p_1/q,\dots,p_n/q)\in\ZZ^n\times\NN$. 

\begin{theorem}[Dirichlet]\label{Dir}
	For any $\balpha=(\alpha_1,\dots,\alpha_n)\in\RR^n$ and $Q\in\NN$, there exist $\bp=(p_1,\dots,p_n)\in\ZZ^n$ and $q\in\NN$, such that
	\begin{equation}\label{eqDir}
		\max\limits_{1\leq j\leq n}\left|\alpha_j-\frac{p_j}{q}\right|<\frac{1}{qQ^{1/n}} \quad \text{ and } \quad 1\leq q\leq Q.
	\end{equation}
\end{theorem} 

We will skip the proof for now, but the statement can be easily obtained as a consequence of Minkowski's Theorem for systems of linear forms, which will be a later topic, see Section \ref{sec:weighted}. Of course, whenever $q\leq Q$ is a solution to \eqref{eqDir}, it will also satisfy the same equation with $Q=q$. This, and the fact that that $\parallel q\balpha\parallel$ is bounded away from zero for any given $q\in\NN$ and $\balpha\in\RR^n\backslash\QQ^n$, gives rise to the following corollary.

\begin{corollary}\label{DirCor}
	Let $\balpha\in\RR^n$. There exist infinitely many $q\in\NN$ satisfying\vspace{-1ex}
	\begin{equation}\label{Eqn:DirCor}
		\parallel q\balpha\parallel <\frac{1}{q^{1/n}}.
	\end{equation}
\end{corollary} 

\begin{remark*}
Corollary \ref{DirCor} is trivially true for rational points. If $\balpha=(\boldsymbol{a}/b)\in\QQ^n$, then $\parallel q\balpha\parallel=0$ for any $q\in\NN$ of the form $q=kb$ with $k\in\NN$.
\end{remark*}

We note that inequality \eqref{Eqn:DirCor} has the form  $\parallel q\balpha\parallel <\psi(q)$ as introduced above with $\psi(q)=q^{-1/n}$. We would like to extend this concept to a suitable class of functions. A positive-valued decreasing function $\psi:\mathbb{R}^+\rightarrow\mathbb{R}^+$ is called an \textit{approximating function}. Given such a function $\psi$, we look at the set of points in $\I^n=[0,1]^n$ which are \textit{simultaneously $\psi$-approximable}.

Namely, this is the set
\begin{equation}\label{approx}
		W_n(\psi)=\left\{\balpha\in \I^n: \parallel q\balpha\parallel<\psi(q)\text{ infinitely often}\right\},
\end{equation}
where infinitely often means that the inequality holds for infinitely many $q\in\NN$. If $n=1$, we simply write $W(\psi)$. An often occurring case is when the approximating function $\psi$ has the form $\psi(q)=q^{-\tau}$ for $\tau>0$. In this case we speak of \textit{simultaneously $\tau$-approximable} points and denote the corresponding set by $W_n(\tau)$. 

\begin{remark*}
	The restriction to the unit cube $\I^n$ is purely for simplicity and does not affect the generality of results. This is due to the fact that the fractional part of the product $q\balpha$ does not depend on the integer part of $\balpha$ and thus we have that
\begin{equation*}
	\parallel q\balpha\parallel=\parallel q(\balpha+\boldsymbol{k})\parallel
\end{equation*}
for any integer vector $\boldsymbol{k}\in\ZZ^n$. In other words, $\balpha$ is $\psi$-approximable if and only if any element of the set $\balpha+\ZZ^n$ is $\psi$-approximable.
\end{remark*}

\begin{remark*}
	Simultaneous approximation is one of the two main types of Diophantine approximation, the other one being \textit{dual approximation}. In the dual case, points in $\RR^m$ are approximated by rational hyperplanes of the form
	\begin{equation*}
		\{\bq\cdot \bx=p:(p,\bq)\in\ZZ\times\ZZ^m\setminus\{\0\}\},
	\end{equation*}
	where $\bq\cdot\bx=q_1 x_1+\dots +q_m x_m$ is the standard scalar product in $\RR^m$. Given an approximating function $\psi$, one can define the set of \textit{dually $\psi$-approximable} points as
	\begin{equation*}
		W_m^D(\psi)=\{\balpha\in\RR^m:\norm{\bq\cdot\balpha}<\psi(|\bq|)\text{ infinitely often}\}.
	\end{equation*}
	Clearly, when $n=m=1$, the sets $W(\psi)$ and $W^D(\psi)$ coincide. The two approximation problems are also closely connected in higher dimension. Of course we can combine the two forms and this leads to a system of linear forms. Famous results like the Khintchine--Groshev Theorem or Minkowski's Theorem (see Theorem \ref{Minkowski}) are formulated in a generality which allows us to deduce both simultaneous and dual statements. Theorem \ref{thm:lines}, one of the main results in Chapter \ref{fibres}, is dependent on dual approximation properties, and in Section \ref{Sec:KTP} we will make use of a transference principle relating the two types of approximation. Other than this, we will not be concerned with dual approximation and refer the reader to \cite{Cassels} or \cite{BBDV} for a more extensive theory.
	
\end{remark*}

Dirichlet's Theorem tells us that $W_n(1/n)=\I^n$. Since \eqref{approx} requires infinitely many solutions, we can conclude that $W_n(\psi)=\I^n$ for any function $\psi$ which eventually dominates $q^{-1/n}$. However, Theorem \ref{Dir} by itself cannot reveal anything more about functions not falling in this category. To make further progress, we start by noting that $W_n(\psi)$ can be a written as a so-called $\limsup$ set. Given a countable collection of sets $(A_k)_{k\in\NN}$, we denote by
\begin{equation*}
	\limsup\limits_{k\rightarrow\infty}A_k:=\bigcap\limits_{l=1}^{\infty}\bigcup\limits_{k=l}^{\infty}A_k
\end{equation*}
the set of points contained in infinitely many of the $A_k$. In our case, let
\begin{equation*}
	A_n(\psi,q)=\bigcup\limits_{|\bp|\leq q} B\left(\frac{\bp}{q},\frac{\psi(q)}{q}\right)\cap \I^n,
\end{equation*}
where $(\bp/q)=(p_1/q,\dots,p_n/q)\in\QQ^n$ and $B(\bx,r)$ is the ball of radius $r$ with respect to $\max$-norm centred at $\bx\in\RR^n$. As always, $|\bp|=\max\{|p_1|,\dots,|p_n|\}$ and the index $|\bp|\leq q$ means that the union runs over all integer vectors $\bp$ satisfying this condition. It follows directly from definition \eqref{approx} that
\begin{equation*}
	W_n(\psi)=\limsup\limits_{q\rightarrow\infty}A_n(\psi,q).
\end{equation*}
The structure of $\limsup$ sets proves to be very useful when trying to investigate the measure theoretic properties of $W_n(\psi)$ with respect to $n$-dimensional Lebesgue measure $\lambda_n$. In fact, it directly fits the requirements for the convergence part of the Borel--Cantelli Lemma, a fundamental result in probability theory.

\begin{lemma}[Borel--Cantelli]\label{BoCa}
	Let $(\Omega,m)$ be a finite measure space and let $(A_k)_{k\in\NN}$ be a collection of $m$-measurable subsets of $\Omega$. Then
	\begin{equation*}
		\sum\limits_{k=1}^{\infty}m(A_k)<\infty\quad \Rightarrow\quad m\left(\limsup\limits_{k\rightarrow\infty}A_k\right)=0.
	\end{equation*}
\end{lemma}
Clearly, $\I^n$ with measure $\lambda_n$ is a finite measure space and we get
\begin{align*}
	\lambda_n(A_n(\psi,q))&=\lambda_n\left(\bigcup\limits_{|\bp|\leq q} B\left(\frac{\bp}{q},\frac{\psi(q)}{q}\right)\cap \I^n\right)\\[1ex]
	&\leq\sum\limits_{|\bp|\leq q}\lambda_n\left(B\left(\frac{\bp}{q},\frac{\psi(q)}{q}\right)\cap \I^n\right)\\[1ex]
	&=\sum\limits_{|\bp|\leq q}\lambda_n\left(B\left(\frac{\bp}{q},\frac{\psi(q)}{q}\right)\right)\\[1ex]
	&=2^n q^n\frac{\psi(q)^n}{q^n}=2^n\psi(q)^n,
\end{align*}
where we can do the shift in summation to compensate for the portions of balls outside $\I^n$. Equality holds whenever the balls contained in $W_n(\psi,q)$ are not overlapping, i.e. for $\psi(q)<1/2$, which will always be satisfied in our considerations. Lemma \ref{BoCa} implies that
\begin{equation}\label{Eqn:KhinCon}
	\lambda_n(W_n(\psi))=0 \quad \text{ if } \quad \sum\limits_{q=1}^{\infty}\psi(q)^n<\infty.
\end{equation}
Observe that in the above argument we have not made use of the fact that $\psi$ is monotonic. 

It is much more work to obtain a converse statement to \eqref{Eqn:KhinCon}. Divergence is not enough to satisfy the converse version of the Borel--Cantelli Lemma, the sets in question are also required to be \textit{pairwise independent}. For example, consider the sets of the form $A_k=[0,\frac{1}{k}]$ with $k\in\NN$. It is easily seen that
\begin{equation*}
	\sum\limits_{k=1}^{\infty}\lambda(A_k)=\sum\limits_{k=1}^{\infty}\frac{1}{k}=\infty,
\end{equation*}
where in the one-dimensional case we just write $\lambda$ for the Lebesgue measure $\lambda_1$. However, as we are dealing with a collection of nested intervals, it follows that\vspace{-2ex}
\begin{equation*}
	\limsup\limits_{k\rightarrow\infty}A_k=\lim\limits_{k\rightarrow\infty}A_k=\{0\},
\end{equation*}
which is a null-set with respect to Lebesgue measure. Similar overlaps occur between the sets $A_n(\psi,q)$, from which $W_n(\psi)$ is constructed and so pairwise independence is not satisfied in our case. There is a proof using the notion of \textit{quasi-independence on average} (see \cite{DAaspects} for an outline of the proof), but originally the following statement was proved by Khintchine using methods of classical measure and integration theory \cite{Khintchine}.

\begin{theorem}[Khintchine, 1924]\label{Khintchine}
	Let $\psi$ be an approximating function. Then
	\begin{equation*}
		\lambda_n(W_n(\psi))=
		\begin{dcases}
			1\quad &\text{ if }\quad \sum\limits_{q=1}^{\infty} \psi(q)^n=\infty,\\[2ex]
			0\quad &\text{ if }\quad \sum\limits_{q=1}^{\infty} \psi(q)^n<\infty.
		\end{dcases}
	\end{equation*}	
\end{theorem}

We will not present a proof here, but later we will show how Khintchine's Theorem can be obtained as a consequence of \textit{ubiquity theory}. See Section \ref{Sec:ClassicalResults}.

\begin{remark*}
	As mentioned above, the convergence part of Theorem \ref{Khintchine} does not need the function $\psi$ to be decreasing. In fact, it has been shown by Gallagher that even for the divergence part this assumption can be removed if $n\geq 2$ \cite{Gallagherkt}. This is an important improvement of Khintchine's Theorem and will be vital to our proof of Theorem \ref{thm:subspaces}.
\end{remark*}
\vspace{-1ex}
However, monotonicity of $\psi$ is essential when $n=1$. In 1941, Duffin and Schaeffer proved the existence of a non-monotonic function $\vartheta:\RR^+\rightarrow\RR^+$, for which the sum $\sum_{q\in\NN}\vartheta(q)$ diverges, but $\lambda(W(\vartheta))=0$ \cite{Duffin}. The construction of $\vartheta$ uses the following facts. For any square-free positive integer $N$, and $s>0$,\vspace{-1ex}
\begin{equation}\label{Eqn:Fact1}
		\sum\limits_{q\in\NN,\ q|N}q=\prod\limits_{p\in\mathbb{P},\ p|N}(1+p)
		\vspace{-2ex}
\end{equation}
and
\begin{equation}\label{Eqn:Fact2}
		\prod\limits_{p\in\mathbb{P}}(1+p^{-s})=\frac{\zeta(s)}{\zeta(2s)},	
\end{equation}
where $\mathbb{P}\subset\NN$ is the set of prime numbers. Here, $\zeta$ denotes the Riemann zeta function, which on our domain of interest is defined by the infinite series
\begin{equation*}
	\zeta(s):=\sum\limits_{k=1}^{\infty}k^{-s}=\frac{1}{1^s}+\frac{1}{2^s}+\frac{1}{3^s}+\cdots .
\end{equation*}
The value $\zeta(1)$ is the limit of the harmonic series, which diverges, whereas $\zeta(2)$ takes the finite value $\pi^2/6$. Therefore, \eqref{Eqn:Fact2} implies that
\begin{equation*}
	\prod\limits_{p\in\mathbb{P}}(1+p^{-1})=\frac{\zeta(1)}{\zeta(2)}=\infty.
\end{equation*}
Thus, we can find a sequence of square-free positive integers $(N_i)_{i\in\NN}$ such that $N_i$ and $N_j$ are coprime whenever $i\neq j$ and which satisfy
\begin{equation}\label{Eqn:Fact3}
	\prod\limits_{p\in\mathbb{P},\ p|N_i}(1+p^{-1})>2^i+1.
\end{equation}
We define the function $\vartheta$ on the positive integers by
\begin{equation*}
	\vartheta(q):=
	\begin{cases}
		2^{-i-1}\frac{q}{N_i}\ &\text{ if }q>1\text{ and }q|N_i\text{ for some }i,\\
		0\ &\text{ otherwise.}
	\end{cases}
\end{equation*}
As above, let
\begin{equation*}
	A(\vartheta,q)=\bigcup\limits_{p=0}^q B\left(\frac{p}{q},\frac{\vartheta(q)}{q}\right)\cap\I.
	\vspace{2ex}
\end{equation*}
If $q$ divides $N_i$, then $A(\vartheta,q)\subseteq A(\vartheta,N_i)$, since
\begin{equation*}
	\frac{\vartheta(q)}{q}=2^{-i-1}\frac{q}{qN_i}=\frac{2^{-i-1}}{N_i}=\frac{\vartheta(N_i)}{N_i}
\end{equation*}
for any divisor $q$ of $N_i$. Hence,\vspace{-1ex}
\begin{equation*}
	\bigcup\limits_{q\in\NN,\ q|N_i}A(\vartheta,q)=A(\vartheta,N_i)
	\vspace{-1ex}
\end{equation*}
and so
\begin{equation*}
	\lambda\left(\bigcup\limits_{q\in\NN,\ q|N_i}A(\vartheta,q)\right)=\lambda\left(A(\vartheta,N_i)\right)=2\vartheta(N_i)=2^{-i}.
	\vspace{1ex}
\end{equation*}
By definition
\begin{equation*}
	W(\vartheta)=\limsup\limits_{q\rightarrow\infty}A(\vartheta,q)=\limsup\limits_{i\rightarrow\infty}A(\vartheta,N_i)
	\vspace{1ex}
\end{equation*}
and, moreover, we have that
\begin{equation*}
	\sum\limits_{i=1}^{\infty}m(A(\vartheta,N_i))=\sum\limits_{i=1}^{\infty}2^{-i}=1<\infty.
\end{equation*}
Thus, the Borel--Cantelli Lemma implies that \vspace{-1ex}
\begin{equation*}
	\lambda(W(\vartheta))=0.
	\vspace{-1ex}
\end{equation*}
On the other hand, using the equations \eqref{Eqn:Fact1} and \eqref{Eqn:Fact3}, we can show that
\begin{align*}
	\sum\limits_{q=1}^{\infty}\vartheta(q)&=\sum\limits_{i=1}^{\infty}2^{-i-1}\frac{1}{N_i}\sum\limits_{q>1,\ q|N_i}q\\[1ex]
	&=\sum\limits_{i=1}^{\infty}2^{-i-1}\frac{1}{N_i}\left(\prod\limits_{p\in\mathbb{P},\ p|N_i}(1+p)-1\right)\\[1ex]
	&=\sum\limits_{i=1}^{\infty}2^{-i-1}\frac{1}{N_i}\left(\prod\limits_{p\in\mathbb{P},\ p|N_i}(1+p^{-1})p-1\right)\\[1ex]
	&>\sum\limits_{i=1}^{\infty}2^{-i-1}\frac{1}{N_i}\left((2^i+1)N_i-1\right)\\[1ex]
	&>\sum\limits_{i=1}^{\infty}2^{-i-1}\frac{1}{N_i}2^i N_i\\[1ex]
	&=\sum\limits_{i=1}^{\infty}2^{-1}=\infty.
\end{align*}
In the same paper, Duffin and Schaeffer also discuss a variation of Khintchine's Theorem for arbitrary positive functions $\psi$. The inequality $\norm{q\alpha}<\psi(q)$ implies the existence of an integer $p$ satisfying
\begin{equation}\label{Eqn:DS}
	\left|\alpha-\frac{p}{q}\right|<\frac{\psi(q)}{q}.
\end{equation}
We can uniquely fix the rational point $p/q$ to the approximation error $\psi(q)/q$ by requiring $\gcd(p,q)=1$. Let $W'(\psi)$ denote the set of points $\alpha$ in $\I$ for which inequality~\eqref{Eqn:DS} holds for infinitely many coprime pairs $(p,q)\in\ZZ\times\NN$. Clearly, $W'(\psi)\subseteq W(\psi)$. The Borel--Cantelli Lemma directly implies that
\begin{equation*}
	\lambda(W'(\psi))=0\quad \text{ if }\quad \sum\limits_{q=1}^{\infty}\varphi(q)\frac{\psi(q)}{q}<\infty,
\end{equation*}
where $\varphi$ is the Euler totient function, i.e.
\begin{equation*}
	\varphi(q):=\left|\left\{k\in\NN: 1\leq k\leq q \text{ and } (k,q)=1\right\}\right|.
\end{equation*}

\begin{Conjecture}[Duffin--Schaeffer, 1941]\label{Con:DS}
	For any positive real-valued function $\psi$,
	\begin{equation*}
	\lambda(W'(\psi))=1\quad \text{ if }\quad \sum\limits_{q=1}^{\infty}\varphi(q)\frac{\psi(q)}{q}=\infty.
\end{equation*}
\end{Conjecture}

Duffin and Schaeffer proved the following weaker version of their conjecture.

\begin{theorem}\label{Thm:DS}
	Conjecture \ref{Con:DS} holds under the additional assumption that
	\begin{equation}\label{Eqn:AssumDS}
		\limsup\limits_{k\rightarrow\infty}\left(\sum\limits_{q=1}^{k}\varphi(q)\frac{\psi(q)}{q}\right)\left(\sum\limits_{q=1}^{k}\psi(q)\right)^{-1}>0.
	\end{equation}
\end{theorem}

Theorem \ref{Thm:DS} will be referred to as the Duffin--Schaeffer Theorem.

\begin{remark*}
	The Duffin--Schaeffer conjecture is one of the most important and difficult open problems in Diophantine approximation. Various partial results have been established (see \cite{Harman} or \cite{Sprindzuk2}) and Gallagher's ``$0-1$ law'' shows that $W'(\psi)$ has either zero or full measure \cite{Gallagher01}. Conjecture \ref{Con:DS} and Khintchine's Theorem are equivalent in the case when $\psi$ is monotonic. 
	
	It is also worth noting that while $\vartheta$ as defined above is a counterexample to Khintchine's Theorem without monotonicity, it is not a counterexample to the Duffin--Schaeffer conjecture. Indeed, using the fact that $\sum_{q|N}\varphi(q)=N$, we see that
	\begin{align*}
		\sum\limits_{q=1}^{\infty}\varphi(q)\frac{\vartheta(q)}{q}&=\sum\limits_{i=1}^{\infty}2^{-i-1}\frac{1}{N_i}\sum\limits_{q>1,\ q|N_i}\varphi(q)\\[1ex]
		&<\sum\limits_{i=1}^{\infty}2^{-i-1}\frac{1}{N_i}N_i\\[1ex]
		&=\sum\limits_{i=1}^{\infty}2^{-i-1}=\frac{1}{2}<\infty.
	\end{align*}
\end{remark*}

Turning back to Khintchine's Theorem, as a direct consequence we get that Corollary \ref{DirCor} is optimal in the sense that almost no $\balpha\in\I^n$ is $\tau$-approximable for $\tau>1/n$. On the other hand, if we define a collection of approximating functions $(\psi_k)_{k\in\NN}$ by
\begin{equation}\label{Eqn:DefPsiK}
	\psi_k:q\mapsto\psi_k(q):=\frac{1}{k}q^{-\frac{1}{n}},
\end{equation}
then almost all $\balpha\in\I^n$ are $\psi_k$-approximable for any $k\in\NN$. We call a point $\balpha\in\I^n$ \textit{badly approximable} if there exists a $k\in\NN$ such that $\balpha\notin W_n(\psi_k)$ and we denote the set of badly approximable points in $\I^n$ by $\bad_n$. For simplicity, we will write $\bad$ for $\bad_1$. In other words, $\bad_n$ is the set of points $\balpha\in\I^n$, for which
\begin{equation*}
	\liminf\limits_{q\rightarrow\infty}q^{1/n}\parallel q\balpha\parallel>0.
\end{equation*}
Khintchine's Theorem immediately implies the following.

\begin{theorem}
	$\lambda_n(\bad_n)=0$.
\end{theorem}

\begin{proof}
	We define an approximating function $\psi$ by
	\begin{equation*}
		\psi(q):=\frac{1}{(q\log q)^{1/n}}.
	\end{equation*}
For any $k\in\NN$, there exists a $Q(k)\in\NN$ such that
\begin{equation*}
	\psi_k(q)=\frac{1}{k}q^{-1/n}>\frac{1}{(\log q)^{1/n}}q^{-1/n}=\psi(q)
\end{equation*}
for all $q\geq Q(k)$ and thus $W_n(\psi)\subseteq W_n(\psi_k)$ for all $k\in\NN$. This implies that 
	\begin{equation*}
		\bad_n=\bigcup\limits_{k=1}^{\infty}\left(\I^n\setminus W_n(\psi_k)\right)\subseteq \I^n\setminus W_n(\psi).
	\end{equation*}
	Now, observe that
	\begin{equation*}
		\sum\limits_{q=1}^{\infty}\psi(q)^n=\sum\limits_{q=1}^{\infty}\frac{1}{q\log q}=\infty,
	\end{equation*}
	and hence, by Khintchine's Theorem, $\lambda_n(\psi)=1$ and so, in turn, $\lambda_n(\bad_n)=0$.
\end{proof}
A priori, the set $\bad_n$ could be empty. However, while being a null set with respect to Lesbesgue measure $\lambda_n$, $\bad_n$ is maximal with respect to Hausdorff dimension (see Section \ref{Sec:Hausdorff} for the definition).

\begin{theorem}\label{Thm:DimBad}
	$\dim(\bad_n)=n$.
\end{theorem}

In dimension one, Theorem \ref{Thm:DimBad} was proved by Jarn\'ik using a Cantor set construction \cite{JarnikBad}. The proof for arbitrary dimensions was done by Schmidt within the more general context of \textit{Schmidt games} and \textit{winning} sets \cite{Schmidtgames}. A concise account of both methods can be found in \cite{DAaspects}. The sets of badly approximable numbers are of great importance in Diophantine approximation and have been studied thoroughly. In particular, in the one-dimensional case they can be completely characterised using the theory of \textit{continued fractions}.

Any number $\alpha\in\RR$ can be written as an iterated fraction of the form
\begin{equation*}
	\alpha=a_0+\cfrac{1}{a_1+\cfrac{1}{a_2+\cfrac{1}{a_3+\cfrac{1}{\ddots}}}}
	\vspace{2ex}
\end{equation*}
with $a_0\in\ZZ$ and $a_k\in\NN$ for $k\geq 1$. For simplicity, we usually prefer the notation
\begin{equation*}
	\alpha=[a_o;a_1,a_2,a_3,\dots].
\end{equation*}
The numbers $a_i$, $i\in\NN$, are called the \textit{partial quotients} of the continued fraction. The first entry is the integral part of $\alpha$ and so our usual restriction to $\alpha\in\I$ corresponds to only considering continued fractions with $a_0=0$. Clearly, rational numbers are represented by finite length continued fractions. On the other hand, every irrational number $\alpha\in\I$ has a unique infinite continued fraction expansion. Given a number $\alpha\notin\QQ$ and $k\in\NN$, we define the $k$-th \emph{convergent} of $\alpha$ by
\begin{equation*}
	\frac{p_k}{q_k}:=[a_o;a_1,a_2,a_3,\dots,a_k].
\end{equation*}
The convergents have very useful properties. They provide explicit solutions to the inequality in Corollary \ref{DirCor}. That is,
\begin{equation*}
	\left| \alpha-\frac{p_n}{q_n}\right| < \frac{1}{q_n^2}\quad \text{ for all }n\in\NN.
\end{equation*}
Furthermore, they are also so-called \textit{best approximates}. This means, given $1\leq q<q_n$, any rational $\frac{p}{q}$ satisfies
\begin{equation*}
	\left| \alpha-\frac{p_n}{q_n}\right| < \left| \alpha-\frac{p}{q}\right|.
	\vspace{1ex}
\end{equation*}
For the proofs of these facts and an extensive account of the theory on continued fractions, see \cite{HardyWright} or \cite{Khintchine-book}. Coming back to badly approximable numbers, we know the following to be true.
\begin{theorem}
	Let $\alpha\in\I^n\setminus\QQ$ with continued fraction expansion $[0;a_0,a_1,a_2,\dots]$. Then
	\begin{equation*}
		\alpha \in \bad\ \Longleftrightarrow\ \exists\ M=M(\alpha)\geq 1\ \text{ s.t. } a_i\leq M\  \text{ for all } i\in\NN. 
	\end{equation*}
\end{theorem}
In other words, $\bad$ consists exactly of the numbers whose continued fractions have bounded partial quotients. In particular, this implies that quadratic irrationals are in $\bad$, since they are precisely the numbers with a periodic continued fraction expansion. It is widely believed that no higher degree algebraic irrationals are badly approximable, but this has not been verified for any single number. In Section \ref{sec:weighted} we will introduce a more general class of badly approximable numbers and we will show later how being badly approximable implies good approximation behaviour with regard to twisted Diophantine approximation. See Theorem  \ref{Thm:Kurzweil} and Theorem \ref{Con1}.

Corollary \ref{DirCor} shows that all numbers are $\psi_1$-approximable, with $\psi_k$ as defined in \eqref{Eqn:DefPsiK}. In other words, this means that the complement of $W_n(\psi_1)$ is empty. Letting $k$ grow, Khintchine's Theorem tells us that $W_n(\psi_k)$ still has full measure for any $k\in\NN$, but the complement might be non-empty. In dimension one, Hurwitz first showed that
\begin{equation*}
	\I\setminus W(\psi_3)\neq\emptyset,
\end{equation*}
implying the existence of badly approximable numbers \cite{Hurwitz}. Asymptotically, as $k$ tends to infinity, this complement will contain all the badly approximable numbers. This example shows the limitations of Theorem \ref{Khintchine}. Being purely a zero-one law it cannot illustrate the difference between those exceptional sets. Even more importantly, let $1/n<\tau_1<\tau_2$. Then, clearly
\begin{equation*}
W_n(\tau_2)\subseteq W_n(\tau_1),
\vspace{-2ex}
\end{equation*}
but
\begin{equation*}
\lambda_n(W_n(\tau_1))=\lambda_n(W_n(\tau_2))=0.
\end{equation*}
From a heuristic point of view, one would expect $W_n(\tau_2)$ to be strictly smaller than $W_n(\tau_1)$, but Khintchine's Theorem is not powerful enough to verify this claim. Hence, we need some finer means to distinguish the sizes of Lebesgue null-sets. This leads us to the concepts of \textit{Hausdorff measure} and \textit{dimension}.

\section[Hausdorff measures and dimension and Jarn\'ik's Theorem]{Hausdorff measures and dimension\\and Jarn\'ik's Theorem}\label{Sec:Hausdorff}

The definition of Hausdorff measures involves several steps. First, a \textit{dimension function} $f:\RR^+\rightarrow\RR^+$ is a left-continuous monotonic function such that 
\begin{equation}\label{Eqn:DimFct}
	\lim\limits_{t\rightarrow 0}f(t)=0.
\end{equation}
Let $A$ be a subset of $\RR^n$. Then, given a real number $\rho>0$, a $\rho$\textit{-cover} of $A$ is a countable collection $\{A_k\}_{k\in\NN}$ of subsets of $\RR^n$ such that
\begin{equation*}
	A\subseteq\bigcup\limits_{k=1}^{\infty}A_k
	\vspace{-2ex}
\end{equation*}
and
\begin{equation*} 
	d_k=d(A_k)<\rho\quad \text{ for all } k\in\NN,
\end{equation*}
where the \textit{diameter} of a set $B\subset\RR^n$ is given by 
\begin{equation*}
	d(B):=\sup\{|\bx-\by|:\bx,\by\in B\}.
\end{equation*}
Let
\begin{equation*}
	\mathcal{H}^f_{\rho}(A)=\inf\sum\limits_{k=1}^{\infty}f(d(A_k)),
	\vspace{2ex}
\end{equation*}
where the infimum is taken over all $\rho$-covers of $A$.

\newpage

As $\rho$ decreases, the class of possible $\rho$-covers for $A$ is reduced and thus $\mathcal{H}^f_{\rho}$ increases. Hence, the limit
\begin{equation*}
	\mathcal{H}^f(A)=\lim\limits_{\rho\rightarrow 0}\mathcal{H}^f_{\rho}(A)
\end{equation*}
exists and is either finite or equal to $+\infty$. $\mathcal{H}^f(A)$ is called the \textit{Hausdorff} $f$\textit{-measure} of $A$. In the case that $f(t)=t^s$ with $s> 0$, the measure $\mathcal{H}^f$ is called the $s$\textit{-dimensional Hausdorff measure} and denoted by $\mathcal{H}^s$. This definition can be extended to $s=0$, even though the function
\begin{equation*}
	f:t\mapsto t^0=1\quad \text{for all }t\in\RR^+
\end{equation*}
does not satisfy condition \eqref{Eqn:DimFct}. It is easily verified that $\mathcal{H}^0(A)$ is the cardinality of $A$ and for $s\in\NN$, $\mathcal{H}^s$ is a constant multiple of the $s$-dimensional Lebesgue measure on $\RR^s$. Importantly, this means their notions of null sets and full measure coincide. It is also worth noting that $\mathcal{H}^g(A)=0$ if
\begin{equation*}
	\cH^f(A)<\infty\quad \text{ and }\quad \lim\limits_{t\rightarrow 0}\frac{g(t)}{f(t)}=0.
\end{equation*}
In particular, unless $A$ is finite, this shows there exists a unique $s_0$ where $\cH^s(A)$ drops from infinity to zero, i.e.
\begin{equation*}
	\cH^s(A)=
	\begin{cases}
		\infty\ &\text{ if }\ s<s_0,\\
		0\ &\text{ if }\ s>s_0.
	\end{cases}
\end{equation*}
This critical point is called the \textit{Hausdorff dimension} of $A$. Formally,
\begin{equation*}
	\dim A=\inf\{s>0:\cH^s(A)=0\}.
\end{equation*}
If $\dim A=s$, then $\cH^s(A)$ may be zero or infinite or satisfy $0<\cH^s(A)<\infty$. For example, a ball in $\RR^n$ has finite measure $\cH^n$. Often it is easier to find the dimension of a set than to obtain the actual measure at the critical value. Showing $\dim A=s$ is usually done by proving $\dim A\leq s$ and $\dim A\geq s$ separately. In most cases the upper bound is more easily attained since it is enough to provide specific covers as $\rho\rightarrow 0$, whereas for the lower bound we need to show no other sequence of covers could lead to a smaller limit.

For more details on Hausdorff measure and dimension as well as related notions and constructions, see \cite{Falconer} and \cite{Mattila}. A standard example for a non-integral dimensional subset of $\mathbb{R}$ is the middle third Cantor set $C$. The set $C$ is constructed by removing the middle third from the unit interval $\I=[0,1]\subset\mathbb{R}$ and then successively removing the middle third from all the resulting intervals. This means that after the first iteration the resulting set consists of the intervals $\left[0,\frac{1}{3}\right]$ and $\left[\frac{2}{3},1\right]$, after removing the middle thirds of those intervals the next set comprises the intervals $\left[0,\frac{1}{9}\right]$, $\left[\frac{2}{9},\frac{1}{3}\right]$, $\left[\frac{2}{3},\frac{7}{9}\right]$ and $\left[\frac{8}{9},1\right]$. Continuing in the same fashion the $k$-th step leaves $2^k$ disjoint intervals of length $3^{-k}$. We will refer to them as level $k$ intervals and the Cantor set is the resulting infinite limit or infinite intersection of this process. It is possible to describe $C$ in explicit ways, for example as all the numbers in the unit interval with a $3$-adic expansion which does not contain the digit 1; i.e. if we write any $\alpha\in \I$ as
\begin{equation*}
	\alpha=\sum\limits_{k=0}^{\infty}a_k 3^{-k},\quad a_k\in\{0,1,2\}\ \text{ for all }k\in\NN,
\end{equation*}
then $C$ comprises exactly the numbers which have a representation with $a_k\neq 1$ for all $k$. This is true since the $k$-th step of the process above removes precisely the numbers where $a_k$ is the first coefficient equal to one.

\begin{lemma} [Example 2.7. of \cite{Falconer}]
	Let $s=\frac{\log 2}{\log 3}$. Then the Cantor set satisfies $\dim C=s$. Furthermore, $\cH^s(C)=1$.
\end{lemma}

\begin{proof}
	As the level $k$ intervals are a collection of $2^k$ intervals of length $3^{-k}$ covering $C$, it is natural to use those sets as a $3^{-k}$-cover of $C$. We directly see that 
	\begin{equation*}
		\mathcal{H}_{3^{-k}}^s(C)\leq 2^k 3^{-ks}=1,
	\end{equation*}
	which implies $\mathcal{H}^s\leq 1$ by letting $k\rightarrow\infty$ and thus it follows that $\dim C\leq s$. 
	
	To prove that $\mathcal{H}^s(C)\geq\frac{1}{2},$ we will show that
	\begin{equation} \label{cantor1}
		\sum\limits_{i\in J}d_i^s\geq\frac{1}{2}=3^{-s}
	\end{equation}
	for any cover $\mathcal{U}=\{U_j\}_{j\in J}$ of $C$ with $d_j=d(U_j)$. Clearly, we can assume all the $U_j$ to be intervals and the compactness of $C$ implies the existence of a finite subcover of $\mathcal{U}$. Hence, by choosing the closure of those intervals, it is enough to prove \eqref{cantor1} when $\mathcal{U}$ is a finite collection of closed subintervals of $[0,1]$. For each $j\in J$, let $k$ be the unique integer such that
	\begin{equation} \label{cantor2}
		3^{-(k+1)}\leq d_j<3^{-k}.
	\end{equation}
	Then $U_j$ can intersect at most one of the level $k$ intervals since all those intervals lie at least $3^{-k}$ apart. Let $l\geq k$, then by the $l$-th iteration any level $k$ interval has been replaced by $2^{l-k}$ level $l$ intervals and hence $U_j$ can intersect at most
	\begin{equation*}
		2^{l-k}=2^l 3^{-sk}\leq 2^l 3^s d_j^s
	\end{equation*}
	intervals of level $l$ where the last inequality is due to \eqref{cantor2}. We can choose $l$ large enough such that $3^{-(l+1)}\leq d_k$ holds for all $k\in J$ and then, since $\mathcal{U}$ is a cover of $C$ and hence the $U_k$ intersect all $2^l$ level $l$ intervals, it follows that
	\begin{equation*}
		2^l\leq\sum\limits_{j\in J}2^l 3^s d_j^s
		\vspace{-1ex}
	\end{equation*}
	which reduces to \eqref{cantor1}. 
\end{proof}

\begin{remark*}
	This proof actually shows that $\frac{1}{2}\leq\cH^s(C)\leq 1$. Proving that the upper bound is sharp requires a little more effort and will be omitted here.
\end{remark*}

Turning back to the $\limsup$ set of $\psi$-approximable points, we can make use of its structure similarly to the above example to obtain an upper bound for $\dim W_n(\psi)$. For simplicity we will stick to the case $n=1$. Recall that
\begin{equation*}
	W(\psi)=\limsup\limits_{q\rightarrow\infty}A(\psi,q),
	\vspace{-2ex}
\end{equation*}
where
\begin{equation*}
	A(\psi,q)=\bigcup\limits_{0\leq p\leq q} B\left(\frac{p}{q},\frac{\psi(q)}{q}\right)\cap \I.
	\vspace{2ex}
\end{equation*}
For each $t\in\NN$, the balls contained in the sets $A(\psi,q)$ with $q\geq t$ form a cover of $W(\psi)$. Let $\rho>0$. Assuming that $\psi$ is monotonic and $\psi(q)<1$ for $q$ large enough, $\rho>2\psi(t)/t$ holds for $t$ large enough. Fixing such a $t$, the balls contained in the collection $\{A(\psi,q)\}_{q\geq t}$ form a $\rho$-cover of $W(\psi)$. It follows that
\begin{align*}
	\cH^s(W(\psi))&\leq 2^s \sum\limits_{q=t}^{\infty}q\left(\frac{\psi(q)}{q}\right)^s\\[1ex]
	&=2^s\sum\limits_{q=t}^{\infty}q^{1-s}\psi(q)^s\rightarrow 0
\end{align*}
as $t\rightarrow\infty$ (i.e. as $\rho\rightarrow 0$) if the sum $\sum_{q\in\NN}q^{1-s}\psi(q)^s$ converges. It can be shown with some extra effort that one can remove the monotonicity condition, giving us a Hausdorff measure analogue of the convergence part of Khintchine's Theorem. Formally:
\begin{lemma}\label{BCHd}
	Let $\psi:\RR^+\rightarrow\RR^+$ be a function and $s>0$. Then
	\begin{equation*}
		\cH^s(W(\psi))=0\quad \text{ if }\quad \sum\limits_{q=1}^{\infty}q^{1-s}\psi(q)^s<\infty.
	\end{equation*}
\end{lemma}
If we let $\psi(q)=q^{-\tau}$ with $\tau>1$ and $s>\frac{2}{\tau+1}$, we see that
\begin{equation*}
	\sum\limits_{q=1}^{\infty}q^{1-s}\psi(q)^s=\sum\limits_{q=1}^{\infty}q^{1-s(1+\tau)}\leq\sum\limits_{q=1}^{\infty}q^{-1}<\infty,
\end{equation*}
and so $\cH^s(W(\tau))=0$ for $s>\frac{2}{\tau+1}$. For $\tau>1$, this shows that $\dim W(\tau)\leq\frac{2}{\tau+1}$.
This result relies on the $\limsup$ nature of $W(\psi)$ and proves the easier half of the Jarn\'ik--Besicovitch Theorem. 

\begin{theorem}[Jarn\'ik--Besicovitch]\label{JB}
	Let $\tau>\frac{1}{n}$. Then
	\begin{equation*}
		\dim(W_n(\tau))=\frac{n+1}{\tau+1}.
	\end{equation*}
\end{theorem}

Theorem \ref{JB} was proved independently and with different methods by Jarn\'ik in 1929 and Besicovitch in 1934 \cite{Jarnikold}, \cite{Besicovitch}. This confirms the intuition that given ${\tau_1<\tau_2}$, $W_n(\tau_2)$ is strictly smaller than $W_n(\tau_1)$. However, it is not able to provide us with the measure $\cH^s(W_n(\tau))$ at the critical exponent and it can only deal with functions of the form $\psi(q)=q^{-\tau}$. These gaps were subsequently filled by Jarn\'ik~\cite{Jarnik}.

\begin{theorem}[Jarn\'ik, 1931]\label{Jarnik}
	Let $\psi$ be an approximating function and $s\in(0,n)$. Then
	\begin{equation*}
		\cH^s(W_n(\psi))=
		\begin{dcases}
			\infty\quad &\text{ if }\quad \sum\limits_{q=1}^{\infty}q^{n-s}\psi(q)^s=\infty,\\[2ex]
			0\quad &\text{ if }\quad \sum\limits_{q=1}^{\infty}q^{n-s}\psi(q)^s<\infty.
		\end{dcases}
		\vspace{2ex}
	\end{equation*}
\end{theorem}

\begin{remark*}
	Here we have to exclude the case when $s=n$ since Khintchine's Theorem shows that
	\begin{equation*}
		\cH^n(W_n(\psi))=\cH^n(\I^n)<\infty\quad \text{ if }\quad \sum\limits_{q=1}^{\infty}\psi(q)^n=\infty.
		\vspace{1ex}
	\end{equation*}
	It is worth noting that the original statement required a stronger monotonicity condition and other additional conditions were imposed on $\psi$. Those restrictions were removed in \cite{limsup}. Again, $\psi$ being monotonic is only needed for the divergence case.
\end{remark*}

As an immediate consequence of Theorem \ref{Jarnik}, we can generalise Theorem \ref{JB} to arbitrary approximation functions. 
\begin{corollary}
	Let $\psi$ be an approximating function. Then
	\begin{equation*}
		\dim(W_n(\psi))=\inf\left\{s>0:\sum\limits_{q=1}^{\infty}q^{n-s}\psi(q)^s<\infty\right\}.
	\end{equation*}
	Moreover, $\cH^s(W_n(\tau))=\infty$, where $s=\frac{n+1}{\tau+1}$. 
\end{corollary}

As with Khintchine's Theorem, we will see how Jarn\'ik's Theorem can be derived from ubiquity theory. However, it can also be deduced as a consequence of Khintchine's Theorem using the \textit{Mass Transference Principle}, which in turn implies we can completely remove the monotonicity condition if $n\geq 2$. Before we go to that proof, we notice that we can combine the statements of Khintchine and Jarn\'ik by making use of the fact that
\begin{equation*}
	\cH^s(\I^n)=
	\begin{dcases}
		\infty\quad &\text{ if }\quad s<n,\\
		C(n)\quad &\text{ if }\quad s=n,
	\end{dcases}	
	\vspace{2ex}
\end{equation*}
where $C(n)$ is a constant only depending on $n$ since $\cH^n$ is a constant multiple of $\lambda_n$. This gives rise to the following concise theorem, which describes the measure theoretic behaviour of $W_n(\psi)$.

\begin{theorem}[Khintchine--Jarn\'ik]\label{KhiJa}
	Let $\psi$ be an approximating function and $s\in(0,n]$. Then
	\begin{equation*}
		\cH^s(W_n(\psi))=
		\begin{dcases}
			\cH^s(\I^n)\ &\text{ if }\quad \sum\limits_{q=1}^{\infty}q^{n-s}\psi(q)^s=\infty,\\[2ex]
			0\quad &\text{ if }\quad \sum\limits_{q=1}^{\infty}q^{n-s}\psi(q)^s<\infty.
		\end{dcases}
		\vspace{2ex}
	\end{equation*}
\end{theorem}
This notation allows us to get rid of the problem when $s=n$ and illustrates how closely the Theorems of Khintchine and Jarn\'ik are related. While one could think of the Hausdorff theory as a subtle refinement of the Lebesgue measure theory, the following section shows how the Hausdorff theory can actually be derived as a consequence of the Lebesgue theory.

\section{The Mass Transference Principle}\label{secMTP}

In this section, we describe a general principle that allows us to deduce Jarn\'ik's Theorem from Khintchine's Theorem. Let $(\Omega,d)$ be a locally compact metric space and suppose there exist constants 
\begin{equation*}
	\delta>0,\quad 0<c_1<1<c_2<\infty,\quad \text{ and }\ r_0>0,
\end{equation*}
such that the inequalities
\begin{equation}\label{MTdim}
	c_1r^{\delta}<\cH^{\delta}(B(\bx,r))<c_2r^{\delta}
\end{equation}
are satisfied for any ball $B(\bx,r)\subset\Omega$ with $\bx\in\Omega$ and $r<r_0$. We have only defined Hausdorff measures and dimension for $\RR^n$, but the theory can be extended to arbitrary metric spaces (see \cite{Falconer}, \cite{Mattila}). It follows from \eqref{MTdim} that
\begin{equation*}
	0<\cH^{\delta}(\Omega)\leq\infty\quad \text{ and }\quad \dim\Omega=\delta.
\end{equation*}
Now, given a dimension function $f$ and a ball $B=B(\bx,r)\subset\Omega$, we define the scaled balls
\begin{equation*}
	B^f=B\left(\bx,f(r)^\frac{1}{\delta}\right)\quad \text{ and }\quad B^s=B\left(\bx,r^\frac{s}{\delta}\right)
	\vspace{2ex}
\end{equation*}
in the special case where $f(r)=r^s$ for some $s>0$, respectively. Clearly, $B=B^\delta$. When, dealing with $\limsup$ sets in $\Omega$, the Mass Transference Principle allows us to derive $\cH^f$-measure theoretic results from statements concerning $\cH^{\delta}$-measure. In the `typical' case where $\delta=n\in\NN$ and $\Omega=\RR^n$ this means we can transform Lebesgue measure theoretic statements to results on Hausdorff measures. The following has been established by Beresnevich and Velani in 2006. The complete theory and proofs can be found in \cite{MassTrans}.

\begin{samepage}
\begin{theorem}[Mass Transference Principle]\label{MTP}
	Let $\{B_k\}_{k\in\NN}$ be a sequence of balls in $\Omega$ with $r(B_k)\rightarrow 0$ as $k\rightarrow\infty$. Let $f$ be a dimension function such that $t^{-\delta}f(t)$ is monotonic. For any ball $B\subset\Omega$ with $\cH^{\delta}(B)>0$,
	\begin{equation}\label{Eqn:MTPcond}
		\text{ if}\quad \cH^{\delta}\left(B\cap\limsup\limits_{k\rightarrow\infty}B_k^f\right)=\cH^{\delta}(B),
	\end{equation}
	\begin{equation*}
		\text{ then }\quad \cH^f\left(B\cap\limsup\limits_{k\rightarrow\infty}B_k^{\delta}\right)=\cH^f(B).
	\end{equation*}
\end{theorem}
\end{samepage}

It is worth noting that Theorem \ref{MTP} has no monotonicity requirements on the radii of balls and even the condition that $r(B_k)\rightarrow 0$ as $k\rightarrow\infty$ is simply of cosmetic nature. Before showing how the Mass Transference Principle can be used to deduce Jarn\'ik's Theorem from Khintchine's Theorem, we will prove the Jarn\'ik--Besicovitch Theorem as a consequence of Dirichlet's Theorem.

\begin{proof}[Proof of Theorem \ref{JB}]
	Let $\psi(q)=q^{-\tau}$ for $\tau>1/n$. This means, given $q\in\NN$, we are dealing with balls of radius $r=q^{-(\tau+1)}$ centred at rational points $\bp/q$. Corollary \ref{DirCor} of Dirichlet's Theorem states that for any $\balpha\in\RR^n$ there are infinitely many rational points $\bp/q$, $q\in\NN$, satisfying
	\begin{equation*}
		\left|\balpha-\frac{\bp}{q}\right|<q^{-\left(1+\frac{1}{n}\right)}.
	\end{equation*}
	Letting $s=\frac{n+1}{\tau+1}$, we see that
	\begin{equation*}
		r^{\frac{s}{n}}=\left(q^{-(\tau+1)}\right)^{\frac{n+1}{n(\tau+1)}}=\left(q^{-(\tau+1)}\right)^{\frac{1+1/n}{\tau+1}}=q^{-(1+\frac{1}{n})}
	\end{equation*}
	and thus \eqref{Eqn:MTPcond} is satisfied with $f(r)=r^s$. This imples that $\dim W_n(\tau)\geq s$. The upper bound follows easily using the Borel--Cantelli Lemma.
\end{proof}

\begin{remark*}
	This actually proves more than just Theorem \ref{JB}. We have also showed that $\cH^s(W_n(\tau))=\infty$, a fact we were previously only able to deduce once Jarn\'ik's Theorem was established. We note that we will also apply the Mass Transference Principle to a similar setting to complete the proof of Theorem \ref{HDfibres} in Section \ref{Sec:HDfibres}.
\end{remark*}

\begin{remark*}
In the following proof and later throughout the text, we will be using the Vinogradov symbol $\ll$. Given two real-valued functions $f$ and $g$, $f(q)\ll g(q)$ means there exist positive constants $c$ and $Q$ such that $f(q)\leq c g(q)$ for all $q\geq Q$. When applied to infinite sums,
\begin{equation*}
	\sum\limits_{q=1}^{\infty}f(q)\ll\sum\limits_{q=1}^{\infty}g(q)\quad :\Longleftrightarrow\quad \sum\limits_{q=1}^{Q}f(q)\leq c\sum\limits_{q=1}^{Q}g(q)
	\vspace{2ex}
\end{equation*}
for a constant $c$ and all $Q$ large enough.
\end{remark*}

\newpage

\begin{proof}[Proof of Jarn\'ik's Theorem modulo Khintchine's Theorem]
	Without loss of generality we can assume that $\psi(q)/q\rightarrow 0$ as $q\rightarrow\infty$. Otherwise, $W_n(\psi)=\I^n$ and obviously $\cH^s(W_n(\psi))=\infty$ for any $s<n$. Hence, the decay condition on the radii in Theorem \ref{MTP} is satisfied. With respect to the above setup, $(\Omega,d)$ is the unit cube $\I^n$ equipped with the supremum norm, $\delta=n$ and $f(r)=r^s$ with $s\in(0,n)$. We assume that $\sum_{q\in\NN} q^{n-s}\psi(q)^s=\infty$. Letting 
	\begin{equation*}
		\vartheta(q)=\left(q^{n-s}\psi(q)^s\right)^{\frac{1}{n}}=q^{1-\frac{s}{n}}\psi(q)^{\frac{s}{n}},
	\end{equation*}
	we see that $\sum_{q\in\NN}\vartheta(q)^n=\infty$. For now, we suppose that either $\vartheta$ is decreasing or that $n\geq 2$. Hence, Khintchine's Theorem implies that 
	\begin{equation*}
		\cH^n\left(B\cap W_n(\vartheta)\right)=\cH^n(B)
	\end{equation*}
	for any ball $B$ in $\RR^n$. Here we are dealing with a $\limsup$ set of balls with radii $r(B_k^s)=\vartheta(q)/q=q^{-s/n}\psi(q)^{s/n}$ for some $q\in\NN$ and thus the Mass Transference Principle tell us that\vspace{2ex}
	\begin{equation*}
		\cH^s\left(B\cap\limsup\limits_{k\rightarrow\infty}B_k^{n}\right)=\cH^s(B),
		\vspace{2ex}
	\end{equation*}
	where $r(B_k^n)=q^{-1}\psi(q)$ for $q\in\NN$. Those are exactly the balls that contribute to the $\limsup$ set $W_n(\psi)$ and thus $\cH^s(W_n(\psi))=\infty$.
	
	In the case where $n=1$ and $\vartheta$ is non-monotonic, we will make use of the Duffin--Schaeffer Theorem (see Theorem \ref{Thm:DS}). For any $q\in\NN$, there is an integer $k\geq 0$ such that $q\in(2^{k-1},2^k]$. As both $q^{1-s}$ and $\psi(q)$ are monotonic functions, we see that
	\begin{align*}
		\vartheta(q)&=q^{1-s}\psi(q)^s\\[1ex] 
		&\leq \left(2^k\right)^{1-s}\psi\left(2^{k-1}\right)^s\\[1ex]
		&=2^{1-s}\left(2^{k-1}\right)^{1-s}\psi\left(2^{k-1}\right)^s.
	\end{align*}
	There are $2^{k-1}$ integers in the interval $q\in(2^{k-1},2^k]$ and, using a shift in summation, this shows that\vspace{2ex}
	\begin{equation*}
		\sum\limits_{q=1}^{\infty}\vartheta(q)\ll\sum\limits_{k=1}^{\infty}\left(2^k\right)^{2-s}\psi\left(2^{k}\right)^s.
		\vspace{2ex}
	\end{equation*}
	On the other hand, the monotonicity of the functions $q^{1-s}$ and $\psi(q)$ also gives us the lower bound
	\begin{align*}
		\vartheta(q)\frac{\varphi(q)}{q}&=q^{1-s}\psi(q)^s\frac{\varphi(q)}{q}\\[1ex]
		&\geq \left(2^{k-1}\right)^{1-s}\psi\left(2^k\right)^s\frac{\varphi(q)}{q}\\[1ex]
		&=2^{s-1}\left(2^{k}\right)^{1-s}\psi\left(2^k\right)^s\frac{\varphi(q)}{q}.
	\end{align*}
	This implies that
	\begin{align*}
		\sum\limits_{q=1}^{\infty}\vartheta(q)\frac{\varphi(q)}{q}&\gg\sum\limits_{k=1}^{\infty}\left(2^{k}\right)^{1-s}\psi\left(2^k\right)^s\sum\limits_{2^{k-1}<q\leq 2^k}\frac{\varphi(q)}{q}\\[1ex]
		&\gg\sum\limits_{k=1}^{\infty}\left(2^k\right)^{2-s}\psi\left(2^{k}\right)^s\\[1ex]
		&\gg\sum\limits_{q=1}^{\infty}\vartheta(q),
	\end{align*}
	where we use the fact that the Euler totient function $\varphi$ satisfies
	\begin{equation*}
		\sum\limits_{q=1}^Q\frac{\varphi(q)}{q}=\frac{6}{\pi^2}Q+\mathcal{O}(\log Q).
	\end{equation*}
	For a proof of this property, see \cite[Theorem 111]{HardyWright}. Thus, the function $\varphi$ satisfies the condition \eqref{Eqn:AssumDS} and, as $\sum_{q\in\NN}\vartheta(q)=\infty$, the Duffin--Schaeffer Theorem implies that $\lambda(W'(\vartheta))=1$. The rest of the argument works as in the previous case, which completes the proof of the divergence part of Jarn\'ik's Theorem. The convergence part follows directly from the $n$-dimensional analogue of Lemma~\ref{BCHd}.
\end{proof}

\section[Inhomogeneous approximation: standard and twisted case]{Inhomogeneous approximation:\\ standard and twisted case}

We introduced simultaneous Diophantine approximation as the study of rational points $\bp/q\in\QQ^n$ lying close to our point of consideration $\balpha\in\RR^n$. However, already from the identity $\parallel q \balpha \parallel<\psi(q)$ we can obtain another point of view to characterise this problem. We take a point $\balpha\in\RR^n$ and want to investigate how close its natural multiples get to integers in $\ZZ^n$, or in other words, how closely their fractional parts $\{q\balpha\}=(\{q\alpha_1\},\dots,\{q\alpha_n\})$ approach points in the set $\{0,1\}^n\subset\RR^n$. Essentially, this corresponds to rotations of the torus $\mathbb{T}^n=\RR^n/\ZZ^n$ by the angle $\balpha\in\RR^n$ and to the question of how often the trajectory of the point $\boldsymbol{0}$ under this action returns to a small neighbourhood of the origin. We will mostly stay clear of this dynamical point of view, but a formal introduction as well as many far reaching consequences can be found in \cite{Einsiedler}.

All of what we have done so far focusses on approximation of the origin. However, we might as well fix any point $\bgamma=(\gamma_1,\dots,\gamma_1)\in \I^n$ and investigate how close we can get with terms of the form $\{q\balpha\}$, where $q$ is in $\NN$. Formally speaking, given an approximating function $\psi$ and a point $\bgamma\in \I^n$, let
\begin{equation*}
	W_n(\psi,\bgamma)=\left\{\balpha\in \I^n:\parallel q\balpha-\bgamma\parallel<\psi(q)\text{ infinitely often}\right\}
\end{equation*}
denote the \emph{inhomogeneous} set of \emph{$\psi$-approximable} points in $\I^n$. Hence, a point $\balpha$ is in $W_n(\psi,\bgamma)$ if there exist infinitely many `shifted' rational points
\begin{equation*}
	\frac{\bp-\bgamma}{q}=\left(\frac{p_1-\gamma_1}{q},\dots,\frac{p_n-\gamma_n}{q}\right),\quad q\in\NN,\ p_j\in\ZZ \text{ for } j\in\{1,\dots,n\} 
\end{equation*}
such that the inequalities
\begin{equation*}
	\left| \alpha_j-\frac{(p_j-\gamma_j)}{q} \right|<\frac{\psi(q)}{q}
	\vspace{1ex}
\end{equation*}
are simultaneously satisfied for all $j\in\{1,\dots,n\}$. As above, when $\psi$ has the form $\psi(q)=q^{-\tau}$ with $\tau>0$, we write $W_n(\tau,\bgamma)$ for $W_n(\psi,\bgamma)$. Of course, in the case where $\bgamma=\0$, we are dealing with the classical homogeneous theory concerning the sets $W_n(\psi)$ and $W_n(\tau)$, respectively. The analogue of Khintchine's Theorem with a fixed inhomogeneous constant $\bgamma\in\I^n$ has been proved by Sz\"usz \cite{Szusz}. From this result we can deduce a Jarn\'ik type statement using the Mass Transference Principle. The proof is identical to the homogeneous case, we simply consider balls with shifted centres. Thus, we obtain the following generalisation of Theorem \ref{KhiJa}. Originally, the part where $s\in(0,n)$ was proved by Schmidt using classical methods \cite{Schmidtjarnik}.

\begin{samepage}
\begin{theorem}[Inhomogeneous Khintchine--Jarn\'ik]\label{Thm:InhomKJ}
	Let $\psi$ be an approximating function, $\bgamma\in\I^n$ and $s\in(0,n]$. Then
	\begin{equation*}
		\cH^s(W_n(\psi,\bgamma))=
		\begin{dcases}
			\cH^s(\I^n)\ &\text{ if }\quad \sum\limits_{q=1}^{\infty}q^{n-s}\psi(q)^s=\infty,\\[2ex]
			0\quad &\text{ if }\quad \sum\limits_{q=1}^{\infty}q^{n-s}\psi(q)^s<\infty.
		\end{dcases}
		\vspace{1ex}
	\end{equation*}
\end{theorem}
\end{samepage}

\begin{remark}\label{Rmk:InhomDir}
	Theorem \ref{Thm:InhomKJ} shows that given any inhomogeneous constant $\bgamma$, we obtain the same measure theoretic statements for $W_n(\psi,\bgamma)$ as in the homogeneous case. However, the inhomogeneous analogue to Dirichlet's Theorem is not true for arbitrary $\bgamma\in\I^n$. Even stronger, Cassels showed the following \cite[Chapter III, Theorem III]{Cassels}.
	\begin{samepage}
	\begin{theorem}\label{Thm:NoInhomDir}
		Let $\psi$ be a positive real-valued function with $\psi(q)\rightarrow 0$ as $q\rightarrow\infty$. Then, there exist $\alpha\notin\QQ$ and $\gamma\in\RR$ such that the system
		\begin{equation*}
			\norm{q\alpha-\gamma}<\psi(Q),\quad 1\leq q\leq Q
		\end{equation*}
		has no integer solution $q$ for infinitely many $Q\in\NN$.
	\end{theorem}
	\end{samepage}
	Obviously, for any $\alpha\in\QQ$ the sequence $(\{q\alpha\})_{q\in\NN}$ only takes finitely many distinct values and so cannot approximate any other points. Hence, the irrationality of $\alpha$ is a central part of this result. Still, Theorem \ref{Thm:NoInhomDir} is not a big impediment. For any irrational $\alpha$ and real $\gamma$ there are infinitely many $q\in\NN$ satisfying
	\begin{equation*}
		\norm{q\alpha-\gamma}<\frac{1+\eps}{\sqrt{5}q},
	\end{equation*}
	where $\eps>0$ can be chosen arbitrarily small \cite{KhintchineFr}. The result in higher dimensions looks slightly different. Cassels proved the existence of $\balpha$ and $\bgamma$ in $\I^n$ such that the inequality
	\begin{equation}\label{Eqn:Cassels}
		\max\limits_{1\leq j\leq n}\norm{q\alpha_j-\gamma_j}<Cq^{-1/n}
	\end{equation}
	has only finitely many solutions $q\in\NN$ for arbitrarily large $C>0$. Namely, such a $\bgamma$ exists for \emph{singular} $\balpha$, a notion which will be introduced in Chapter \ref{dirichlet} and coincides with rational points if and only if $n=1$. For the precise statement involving \eqref{Eqn:Cassels}, see Theorem \ref{Thm:Casselsnonsing}.
\end{remark}

As well as fixing the inhomogeneous constant $\bgamma\in\I^n$ and investigating properties of $W_n(\psi,\gamma)$, we can also consider a converse point of view. Given a fixed $\balpha\in\I^n$, we would like to know the distribution of the sequence $(\{q\balpha\})_{q\in\NN}$ inside $\I^n$. Concretely, given an approximating function $\psi$ and $\balpha\in\I^n$, we are interested in the set
\begin{equation*}
	W_n^{\balpha}(\psi)=\left\{\bgamma\in\I^n:\norm{q\balpha-\bgamma}<\psi(q)\text{ infinitely often}\right\}.
\end{equation*}
As above, when $\psi$ has the form $\psi(q)=q^{-\tau}$ with $\tau>0$, we write $W_n^{\balpha}(\tau)$ for $W_n^{\balpha}(\psi)$. This situation is slightly different from the previously considered theory. Given $q$, we are dealing with a single ball of radius $\psi(q)$ centred at $\{q\balpha\}$ rather than having $q$ balls of size $\psi(q)/q$ around (shifted) rational points. Still, the total measures for each $q$ are identical and so the Borel--Cantelli Lemma tells us that
\begin{equation}\label{Eqn:BCtwist}
	\lambda_n\left(W_n^{\balpha}(\psi)\right)=0 \quad \text{ if } \quad \sum\limits_{q=1}^{\infty}\psi(q)^n<\infty.
\end{equation}
The converse situation is a bit more complicated. A point $\balpha=(\alpha_1,\dots,\alpha_n)\in\I^n$ is called \emph{totally irrational} if the numbers $1,\alpha_1,\dots,\alpha_n$ are linearly independent over $\QQ$. If $\balpha$ is not totally irrational, then the elements of $(\{q\balpha\})_{q\in\NN}$ are contained in a finite union of $\QQ$-linear subspaces of $\RR^n$ intersected with $\I^n$ and so $\lambda_n\left(W_n^{\balpha}(\psi)\right)=0$ for any approximation function $\psi$. On the other hand, it is known that the sequence $(\{q\balpha\})_{q\in\NN}$ is equidistributed in $\I^n$ when $\balpha$ is totally irrational \cite{Weyl}. Roughly speaking, this means that for any ball $B\in\I^n$, the asymptotic proportion of elements of $(\{q\balpha\})_{q\in\NN}$ contained in $B$ equals $\lambda_n(B)$. However, having this property is still not strong enough to guarantee $\lambda_n\left(W_n^{\balpha}(\psi)\right)=1$ for any approximating function $\psi$ with ${\sum_{q\in\NN}\psi(q)^n=\infty.}$ To see this, define the set of \emph{twisted $\psi$-approximable} points as
\begin{equation*}
	W_n^{\times}(\psi)=\left\{\balpha\in\I^n:\lambda_n(W_n^{\balpha}(\psi))=1\right\}.
\end{equation*}
The twisted analogue to Khintchine's Theorem was proved by Kurzweil among other results in his fundamental paper on inhomogeneous Diophantine approximation \cite{Kurzweil}. 

\begin{theorem}[Twisted Khintchine]\label{Thm:TwistedKhin}
	Let $\psi$ be an approximating function. Then we have
	\begin{equation*}
	\lambda_n(W_n^{\times}(\psi))=1\quad \text{ if }\quad \sum\limits_{q=1}^{\infty} \psi(q)^n=\infty
	\end{equation*}
	and
	\begin{equation*}
	W_n^{\times}(\psi)=\emptyset\quad \text{ if }\quad \sum\limits_{q=1}^{\infty}\psi(q)^n<\infty.
	\vspace{2ex}
	\end{equation*}
\end{theorem}

\begin{remark*}
	Of course, the convergence part of Theorem \ref{Thm:TwistedKhin} simply follows from \eqref{Eqn:BCtwist}, which is a consequence of the Borel--Cantelli Lemma, and the divergence part is the main achievement. It is worth noting that for this problem there is no Jarn\'ik type theory since $W_n^{\times}(\psi)$ either has full measure or is the empty set.
\end{remark*}

Remarkably, Kurzweil also proved the following identity, establishing a deep connection between homogeneous and twisted inhomogeneous Diophantine Approximation.

\newpage

\begin{theorem}[Kurzweil]\label{Thm:Kurzweil}
	Denote by $\Psi^{\infty}_n$ the set of all approximating functions $\psi$ such that $\sum\psi(q)^n=\infty$. Then,\vspace{-2ex}
	\begin{equation*}
		\bigcap\limits_{\psi\in\Psi^{\infty}_n}W_n^{\times}(\psi)=\bad_n.
	\end{equation*}
\end{theorem}

Rewritten in a contrapositive way, Theorem \ref{Thm:Kurzweil} tells us that
\begin{equation*}
	\balpha\notin\bad_n\ \Longleftrightarrow\ \exists\ \psi\in\Psi^{\infty}_n\ \text{ such that }\ \lambda_n\left(W_n^{\balpha}(\psi)\right)<1,
\end{equation*}
where the choice of the exceptional $\psi$ might depend on $\balpha$.

\begin{remark*}
	The original version of Theorem \ref{Thm:Kurzweil} was stated in a different and more general fashion than above. This might be a reason why the result and its significance were somewhat overlooked for a long period of time. However, in recent years the work of Kurzweil has become a focal point in Diophantine approximation and advances were made in different directions by Fayad \cite{Fayad}, Kim \cite{Kim}, \cite{Fuchs}, Tseng \cite{Tseng}, Chaika \cite{Chaika} and Simmons \cite{SimmonsKurzweil}. An important extension due to Harrap will be discussed shortly.
\end{remark*}

\section[Weighted Diophantine approximation: standard and twisted case]{Weighted Diophantine approximation:\\ standard and twisted case}\label{sec:weighted}

While so far we have only been concerned with equidistant approximation, i.e. $\limsup$ sets built from balls around rational points (or indeed `shifted' rational points in the case of inhomogeneous approximation) according to the $\max$-norm, we can also vary this approach by applying different weights to different coordinate directions. In the following, an $n$-tuple $\bi=(i_1,\dots,i_n)\in\RR^n$ satisfying
\begin{equation}\label{Eqn:WV1}
	0\leq i_j\leq 1,\quad (1\leq j\leq n)
	\vspace{-2ex}
\end{equation}
and
\begin{equation}\label{Eqn:WV2}
	i_1+i_2+\dots+i_n=1
\end{equation}
will be called a \textit{weight vector}. We will consider simultaneous approximation with respect to rectangles given by a weight vector $\bi$ rather than balls given by $\max$-norm. In this case we have the following generalisation of Dirichlet's Theorem.

\begin{theorem}[Weighted Dirichlet]\label{wDir}
	Let $\bi\in\RR^n$ be a weight vector. Then, for any $\balpha\in\RR^n$ and $Q\in\NN$, there exists $q\in\NN$, $q\leq Q$ such that
	\begin{equation}\label{Eqn:wDir}
		\parallel q\alpha_j\parallel<Q^{-i_j},\quad (1\leq j\leq n).
	\end{equation}
\end{theorem}

Analogously to Corollary \ref{DirCor} in the non-weighted case, Theorem \ref{wDir} immediately implies the following.

\begin{corollary}\label{Cor:wDir}
	Let $\bi\in\RR^n$ be a weight vector. Then, for any $\balpha\in\RR^n$ there exist infinitely many $q\in\NN$ such that
	\begin{equation*}
		\max\limits_{1\leq j\leq n}q^{i_j}\parallel q\alpha_j\parallel<1.
		\vspace{2ex}
	\end{equation*}
\end{corollary}

Theorem \ref{wDir} shows that we have 
\begin{equation*}
	\max\left\{\parallel q\alpha_1\parallel^{1/i_1},\dots,\parallel q\alpha_n\parallel^{1/i_n}\right\}<Q^{-1},
\end{equation*}
which illustrates that instead of a cube given by $\max$-norm, we are dealing with an $n$-dimensional rectangle, the side lengths of which are scaled by the powers $i_j$. Note that while the shape of this rectangle depends on the weight vector $\bi$, the conditions \eqref{Eqn:WV1} and \eqref{Eqn:WV2} ensure that the volume is the same for any choice of $\bi$. The proof of Theorem \ref{wDir} follows from a surprisingly simple geometric observation by Minkowski. 

\begin{samepage}
\begin{theorem}[Minkowski's Convex Body Theorem]\label{Thm:MinkCBT}
	Let $K$ be a bounded convex set in $\RR^m$, symmetric about $\0$, i.e. $\bx\in K\Leftrightarrow -\bx\in K$. Assume that either $\lambda_m(K)>2^m$ or that $K$ is closed and $\lambda_m(K)\geq 2^m$. Then $K$ contains an integer point $\bz\neq\0$. 
\end{theorem}
\end{samepage}

\newpage

\begin{proof}
	We will follow the argument given by Schmidt \cite{schmidt}. Suppose that $K$ satisfies ${\lambda_m(K)>2^m}$. For $q\in\NN$, define
	\begin{equation*}
		Z_q(K)=\left\{\frac{\bp}{q}=\left(\frac{p_1}{q},\dots,\frac{p_m}{q}\right)\in K:\bp\in\ZZ^m\right\}.
	\end{equation*}
	Then it follows that
	\begin{equation*}
		\lim\limits_{q\rightarrow\infty}\frac{\# Z_q(K)}{q^m\lambda_m(K)}=1
	\end{equation*}
	and thus we have
	\begin{equation*}
		\# Z_q(K)>q^m2^m=(2q)^m
	\end{equation*}
	for $q$ large enough. There are $2q$ different residue classes modulo $2q$ and every point in $Z_q(K)$ has $m$ coordinates, so there must be two distinct points $\bp/q$ and $\tilde{\bp}/q$ in $Z_q(K)$ satisfying
	\begin{equation}\label{Eqn:Mod}
		p_j\equiv\tilde{p}_j\mod 2q,\quad \quad (1\leq j\leq m).
	\end{equation}
	By symmetry, $-\tilde{\bp}/q\in K$ and by convexity,
	\begin{equation*}
		\bz=\frac{1}{2}\frac{\bp}{q}+\frac{1}{2}\frac{-\tilde{\bp}}{q}=\frac{\bp-\tilde{\bp}}{2q}\in K.
	\end{equation*}
	Clearly, $\bz\neq\0$ and \eqref{Eqn:Mod} shows that $\bz$ is contained in $\ZZ^n$, which completes the proof.	
	
	The case when $K$ is closed and $\lambda_m(K)= 2^m$ can be reduced to the case when $\lambda_m(K)>2^m$. $K$ and $\ZZ^m\setminus K$ are closed disjoint subsets of $\RR^m$. Since $\RR^m$ is a normal space, $K$ is contained in an open neighbourhood $K'$ satisfying $K'\cap(\ZZ^m\setminus K)=\emptyset$. Clearly, $K'$ can be assumed to be convex and symmetric about zero. It follows that $\lambda_m(K')>\lambda_m(K)=2^m$ and so, by the previous case, $K'$ contains a non-zero integer point $\bz$, which must be contained in $K$. 
\end{proof}

We now use Theorem \ref{Thm:MinkCBT} to deduce the following fundamental theorem in the geometry of numbers.

\newpage

\begin{theorem}[Minkowski's Linear Forms Theorem]\label{Minkowski}
	Given $m\in\NN$, suppose that $a_{j,k}$ with $1\leq j,k\leq m$, are real numbers with determinant $\det (a_{j,k})=\pm 1$ and let $A_1,\dots,A_m$ be positive real numbers with product $A_1 A_2\cdots A_m=1$. Then, there exists an integer point $\bz=(z_1,\dots,z_m)\in\ZZ^m\setminus\{\0\}$ such that the inequalities
	\begin{equation}\label{Eqn:Mink}
		\begin{aligned}
			|a_{j,1}z_1+\dots+a_{j,m}z_m&|<A_1,\quad \quad (1\leq j\leq m-1) \\
			\text{and }\quad \quad |a_{m,1}z_1+\dots+a_{m,m}z_m&|\leq A_m
		\end{aligned}
	\end{equation}
	are satisfied simultaneously.
\end{theorem}

\begin{proof}
	Denote the linear forms in question by
	\begin{equation*}
		L_j(\bx)=a_{j,1}x_1+\dots+a_{j,m}x_m,\quad \quad (1\leq j\leq m)
	\end{equation*}
		and let
	\begin{equation*}
		\tilde{L}_j(\bx)=\frac{1}{A_j}L_j(\bx),\quad \quad (1\leq j\leq m).
		\vspace{2ex}
	\end{equation*}
	Then \eqref{Eqn:Mink} can be rewritten as
	\begin{align*}
		|\tilde{L}_j(\bx)&|<1,\quad \quad (1\leq j\leq m-1)\\
		\text{and }\quad \quad |\tilde{L}_m(\bx)&|\leq 1.
	\end{align*}
	This modified system of linear forms still has determinant $\det(\tilde{a}_{j,k})=\pm 1$, so we can restrict ourselves to the case when $A_1=\dots=A_m=1$. The set $K\subset\RR^n$ of all $\bx\in\RR^n$ satisfying
	\begin{equation*}
		|L_j(\bx)|\leq 1,\quad \quad (1\leq j\leq m)
	\end{equation*}
	is the image of the closed unit cube $\I^n$ under a linear transformation of determinant $\pm 1$ and thus $K$ is a closed parallelepiped, symmetric about $\0$ with volume $\lambda_m(K)=2^m$. By Theorem \ref{Thm:MinkCBT}, there is an integer point $\bz\neq \0$ contained in $K$.
	
	To get inequality for the first $m-1$ linear forms, we need to slightly modify this argument.
	
	For each $\eps>0$, the system of inequalities
	\begin{align*}
		|{L}_j(\bx)&|<1,\quad \quad \quad (1\leq j\leq m-1)\\
		\text{and }\quad \quad |{L}_m(\bx)&|\leq 1+\eps.
	\end{align*}
	defines a symmetric parallelepiped $K_{\eps}$ of volume $\lambda_m(K_{\eps})=2^m(1+\eps)>2^m$. By Theorem \ref{Thm:MinkCBT}, there is an integer point $\bz_{\eps}\neq \0$ contained in $K_{\eps}$. For $\eps<1$, all the sets $K_{\eps}$ will be contained in $K_1$. As a bounded body, $K_1$ only contains finitely many integer points and so there must be a sequence $(\eps_k)_{k\in\NN}$ tending to zero as $k\rightarrow\infty$, such that all the $\bz_{\eps_k}$ are the same, say $\bz$. On letting $k\rightarrow\infty$ we see that $\bz$ must satisfy \eqref{Eqn:Mink}.
\end{proof}

\begin{samepage}
We will often refer to this result simply as Minkowski's Theorem. Theorem \ref{wDir} follows from Theorem \ref{Minkowski} by setting $m=n+1$ and the following choice of coefficients:
\begin{align*}
	a_{j,1}=\alpha_j,\quad a_{j,j+1}=-1,\quad A_j&=Q^{-i_j},\quad (1\leq q\leq n),\\
	a_{n+1,1}=1,\quad A_{n+1}&=Q
\end{align*}
and all the other coefficients $a_{j,k}=0$. The special case where $A_1=\dots=A_n=Q^{1/n}$ proves the original version of Dirichlet's Theorem. 
\end{samepage}

Analogously to the standard theory in Section \ref{Sec:Basic}, we can define the class of \textit{simultaneously $(\bi,\psi)$-approximable} numbers in $\I^n$ as
\begin{equation*}
		W_n(\bi,\psi)=\left\{\balpha\in \I^n: \max\limits_{1\leq j\leq n}\psi(q)^{-i_j}\parallel q\alpha_j\parallel<1\text{ infinitely often}\right\},
\end{equation*}
where $\psi$ is an approximating function and $\bi=(i_1,\dots,i_n)$ is a weight vector. Observe that, by Corollary \ref{Cor:wDir}, $W_n(\bi,\psi)=\I^n$ for $\psi=q^{-1}$ and any weight vector $\bi$. More generally, Khintchine himself extended his classical result from $\bi=(1/n,\dots,1/n)$ to arbitrary weight vectors \cite{Khintchine2}.

\begin{samepage}
\begin{theorem}[Weighted Khintchine]\label{Thm:WeightedKhin}
	Let $\psi$ be an approximating function and $\bi$ an $n$-dimensional weight vector. Then
	\begin{equation*}
		\lambda_n(W_n(\bi,\psi))=
		\begin{dcases}
			1\quad &\text{ if }\quad \sum\limits_{q=1}^{\infty} \psi(q)=\infty,\\[2ex]
			0\quad &\text{ if }\quad \sum\limits_{q=1}^{\infty} \psi(q)<\infty.
		\end{dcases}
		\vspace{1ex}
	\end{equation*}	
\end{theorem}
\end{samepage}

\begin{remark*}
	When $n=1$, then $i=1$ is the only $1$-dimensional weight vector and so this is just Khintchine's Theorem for $\RR$. For $n\geq 2$, we get the classical result by choosing $\bi=(1/n,\dots,1/n)$. Note that in this formulation the exponent within the sum does not depend on $n$. This is simply because the condition $\norm{q\alpha_j}<\psi(q)$ has been replaced by $\norm{q\alpha_j}<\psi(q)^{i_j}$ and these powers satisfy $i_1+\dots+i_n=1$. For $n\geq 2$ we can drop the monotonicity assumption on $\psi$ in accordance with the previously developed theory. 
\end{remark*}

Analogously to the non-weighted case we can introduce the problem of twisted Diophantine Approximation. For any fixed $\balpha\in\RR^n$, let
\begin{equation*}
	W_n^{\balpha}(\bi,\psi)=\left\{\bbeta\in \I^n: \max\limits_{1\leq j\leq n}\psi(q)^{-i_j}\parallel q\alpha_j-\beta_j\parallel<1\text{ infinitely often}\right\}
\end{equation*}
and
\begin{equation*}
W_n^{\times}(\bi,\psi)=\left\{\balpha\in\I^n: \lambda_n\left(W_n^{\balpha}(\bi,\psi)\right)=1\right\}.
\end{equation*}
As for the previous results, the Borel--Cantelli Lemma implies that $\lambda_n(W_n^{\balpha}(\bi,\psi))=0$ if $\sum_{q\in\NN}\psi(q)<\infty$, independent of the choice of $\balpha$. As for non-weighted approximation, it follows that in this case the set $W_n^{\times}(\bi,\psi)$ will be empty. This gives us the convergence part for a Khintchine-type result. The divergence part was proved by Harrap in a recent paper \cite{Harraptwisted}.

\newpage

\begin{samepage}
\begin{theorem}[Weighted Twisted Khintchine]
	Let $\psi$ be an approximating function. Then we have
	\begin{equation*}
	\lambda_n(W_n^{\times}(\bi,\psi))=1\quad \text{ if }\quad \sum\limits_{q=1}^{\infty} \psi(q)=\infty
	\end{equation*}
	and
	\begin{equation*}
	W_n^{\times}(\bi,\psi)=\emptyset\quad \text{ if }\quad \sum\limits_{q=1}^{\infty}\psi(q)<\infty.
	\vspace{2ex}
	\end{equation*}
\end{theorem}
\end{samepage}

In the same paper, Harrap also extended Kurzweil's Theorem to the more general setting of arbitrary weight vectors, giving us the following result.

\begin{samepage}
\begin{theorem}[Weighted Kurzweil]
	Denote by $\Psi^{\infty}$ the set of all approximating functions $\psi$ such that $\sum_{q\in\NN}\psi(q)=\infty$. Then, for any weight vector $\bi\in\I^n$,
	\begin{equation*}
		\bigcap\limits_{\psi\in\Psi^{\infty}}W_n^{\times}(\bi,\psi))=\bad_n(\bi).
	\end{equation*}
\end{theorem}
\end{samepage}

The set $\bad_n(\bi)$ mentioned here is the natural generalisation of $\bad_n$ to the setting of weighted Diophantine Approximation. Corollary \ref{Cor:wDir} states that for any weight vector $\bi$ and $\balpha\in\RR^n$ there are infinitely many integers $q$ satisfying
\begin{equation}\label{Eqn:WBad}
\norm{q\alpha_j}<q^{-i_j},\quad (1\leq j\leq n).
\end{equation}
Analogously to the basic theory in Section \ref{Sec:Basic}, we can ask if we still get infinitely many solutions $q$ for \eqref{Eqn:WBad} if we multiply the right-hand side by a constant factor $c<1$. If this is not true for arbitrarily small constants, we call $\balpha$ \emph{$\bi$-badly approximable} or say $\balpha\in \bad_n(\bi)$. More formally,
\begin{equation*}
	\balpha\in\bad_n(\bi)\ \Leftrightarrow\ \liminf\limits_{q\rightarrow\infty}\max\limits_{1\leq j\leq n}q^{i_j}\norm{q\alpha_j}>0.
\end{equation*}
Of course, the set $\bad_n(1/n,\dots,1/n)$ is simply the previously introduced $\bad_n$. As one would expect, these more general sets $\bad_n(\bi)$ have many of the properties of $\bad_n$. As in the non-weighted case, we can use Theorem \ref{Thm:WeightedKhin} to show that $\lambda_n(\bad_n(\bi))=0$ for any weight vector $\bi\in\RR^n$. The proof is completely analogous to Theorem \ref{Thm:DimBad} and will be skipped here. Regarding the dimension theory, Pollington and Velani showed in \cite{PolVel} that $\dim\bad_n(\bi)=n$ for any weight vector $\bi$. Special interest has been taken in the intersection of sets of the form $\bad_2(\bi)$, in particular the following famous conjecture of Schmidt \cite{SchmidtBad}.

\begin{Conjecture}[Schmidt, 1983]\label{ConSchmidt}
	Let $i,j$ be in $\I$ with $i+j=1$. Then,
	\begin{equation}\label{Eqn:SchmidtCon}
		\bad_2(i,j)\cap\bad_2(j,i)\neq\emptyset.
	\end{equation}
\end{Conjecture}
 
\begin{remark*}
	To be exact, Schmidt formulated Conjecture \ref{ConSchmidt} for the pair $i=1/3$ and $j=2/3$, but, of course, the question is relevant for any choice of weight vector.
\end{remark*}

The proof of Schmidt's Conjecture was part of a fairly recent paper by  Badziahin, Pollington and Velani \cite{Badziahin}. In fact, they established the following much more general result. 

\begin{theorem}\label{Thm:Badziahin}
	Let $(i,j)_{k\in\NN}$ be a sequence of two-dimensional weight vectors in $\I^2$ and assume that
	\begin{equation}\label{Eqn:BPV1}
		\liminf\limits_{k\rightarrow\infty}\min\{i_k,j_k\}>0.
	\end{equation}
	Then,
	\begin{equation}\label{Eqn:Badziahin}
		\bigcap\limits_{k=1}^{\infty}\bad_2(i_k,j_k)\neq\emptyset.
	\end{equation}
\end{theorem}

In particular, condition \eqref{Eqn:BPV1} is trivially satisfied for any finite sequence of weight vectors, which proves \eqref{Eqn:SchmidtCon}. It was later shown by An that \eqref{Eqn:Badziahin} still holds without requiring condition \eqref{Eqn:BPV1} \cite{An}. He actually proved the stronger condition that $\bad_2(i,j)$ is \textit{winning} in the sense of Schmidt games. Theorem \ref{Thm:Badziahin} was subsequently extended to arbitrary dimensions by Beresnevich \cite{Berschmidt}. As in dimension two, the higher dimensional analogue originally required a condition on the weights. This condition was removed by Yang \cite{Yang}. However, the question of winning is still open in higher dimensions. Schmidt's Conjecture is closely related to one of the most famous and deepest open problems in Diophantine approximation.

\begin{Conjecture}[Littlewood, 1930s]
	For any pair $(\alpha,\beta)\in\I^2$,
	\begin{equation}\label{Eqn:LW}
		\liminf\limits_{q\rightarrow\infty}q\norm{q\alpha}\norm{q\beta}=0.
	\end{equation}
\end{Conjecture}

It is easily seen that a counterexample to Schmidt's Conjecture would have implied the truth of Littlewood's Conjecture. Indeed, assume there exists a weight vector $(i,j)$ such that
	\begin{equation*}
		\bad_2(i,j)\cap\bad_2(j,i)=\emptyset.
	\end{equation*}
Then, any element of $\I^2$ is either not contained in $\bad_2(i,j)$ or not in $\bad_2(j,i)$. Fix $(\alpha,\beta)\in\I^2$ and assume without loss of generality that $(\alpha,\beta)\notin\bad_2(i,j)$. This means that for any $\eps>0$, there are infinitely many $q\in\NN$ satisfying
\begin{equation*}
	\max\{\norm{q\alpha}^i,\norm{q\beta}^j\}<\varepsilon q^{-1}.
\end{equation*}
Making use of the fact that $\norm{\cdot}$ is always less than one, we see that there are infinitely many $q\in\NN$, for which
\begin{align*}
	q\norm{q\alpha}\norm{q\beta}&=q\norm{q\alpha}^i\norm{q\alpha}^j\norm{q\beta}^j\norm{q\beta}^i\\&<q\norm{q\alpha}^i\norm{q\beta}^j\\
	&<q\max\{\norm{q\alpha}^i,\norm{q\beta}^j\}\\
	&<\varepsilon
\end{align*}
and since $\eps$ was chosen arbitrarily small, this implies that
\begin{equation*}
	\liminf\limits_{q\rightarrow\infty}q\norm{q\alpha}\norm{q\beta}=0.
\end{equation*}
Even more generally, suppose that
\begin{equation}\label{Eqn:LWtrue}
	\bigcap\limits_{(i,j)}\bad_2(i,j)=\emptyset,
\end{equation}
where the intersection runs over all two-dimensional weight vectors. By the same argument this would immediately imply Littlewood's Conjecture. However, there is no indication why \eqref{Eqn:LWtrue} should be true and proving its negation would not shine any new light on Littlewood's Conjecture.

Despite having been studied for many decades, Littlewood's Conjecture has not been solved yet. However, many partial results have been obtained. It is easily seen that any pair $(\alpha,\beta)$ not satisfying \eqref{Eqn:LW} requires both $\alpha$ and $\beta$ to be badly approximable. Suppose $\alpha\notin\bad$. Then, for any $\eps>0$, there exist infinitely many $q\in\NN$ such that $q\norm{q\alpha}<\eps$, and since $\norm{q\beta}<1$ for any $q\in\NN$, it follows that
\begin{equation*}
	q\norm{q\alpha}\norm{q\beta}<q\norm{q\alpha}<\eps
\end{equation*}
infinitely often, which implies \eqref{Eqn:LW}. Hence, we can restrict our attention to badly approximable pairs, in which case the following statement was proved by Pollington and Velani \cite{Pollington}.

\begin{theorem}\label{Thm:PV}
	Let $\alpha\in\bad$. Then
	\begin{equation*}
		\dim\left(\left\{\beta\in\bad:\liminf\limits_{q\rightarrow\infty}q\log q\norm{q\alpha}\norm{q\beta}=0\right\}\right)=1.
	\end{equation*}
\end{theorem}

Note that the points in Theorem \ref{Thm:PV} satisfy an even stronger property than $\eqref{Eqn:LW}.$ In fact, this is true for any $\alpha$ and almost all $\beta\in\I$. Khintchine's Theorem tell us that
\begin{equation*}
	\liminf\limits_{q\rightarrow\infty}q\log q\norm{q\beta}=0
\end{equation*}
for almost all $\beta\in\I$ and $\norm{q\alpha}$ is bounded from above by $1$ for any $q\in\NN$. 
Regarding potential counterexamples to Littlewood's Conjecture, Einsiedler, Katok and Lindenstrauss proved the following fundamental result \cite{EKL}.

\begin{theorem}
	\begin{equation*}
		\dim\left(\left\{(\alpha,\beta)\in\I^2:\liminf\limits_{q\rightarrow\infty}q\norm{q\alpha}\norm{q\beta}>0\right\}\right)=0.
	\end{equation*}
\end{theorem}

\begin{remark*}
	An analogue to Littlewood's Conjecture can be formulated for arbitrary dimensions. Namely, given a point $\balpha=(\alpha_1,\dots,\alpha_n)\in\I^n$, we can ask whether
	\begin{equation}\label{Eqn:LWhigh}
		\liminf\limits_{q\rightarrow\infty}q\ \prod\limits_{j=1}^n\ \norm{q\alpha_j}=0.
	\end{equation}
	However, this question is of limited interest when $n\neq 2$. In the one-dimensional case, the set of exceptions to \eqref{Eqn:LWhigh} is simply equal to $\bad$. Furthermore, when $n\geq 3$, any counterexample to \eqref{Eqn:LWhigh} must satisfy
	\begin{equation*}
		\liminf\limits_{q\rightarrow\infty}q\norm{q\alpha_j}\norm{q\alpha_k}>0
	\end{equation*}
	whenever $j\neq k$ and thus proving Littlewood's Conjecture would immediately imply the higher-dimensional analogue.
\end{remark*}
	
\section{Outlook}
In Chapter \ref{ubiquity} we will introduce the concept of ubiquity and give proofs for the classical theorems of Khintchine and Jarn\'ik. Moreover, this will provide us with tools to advance these basic results to the more specific setting of affine coordinate subspaces. 

The Khintchine--Jarn\'ik Theorem is very powerful in the sense that it allows us to determine $\cH^s(W_n(\psi))$ for any given approximating function $\psi$ and $s\in(0,n]$. However, it does not reveal which exact points are $\psi$-approximable or not. In particular, if we consider a set $A\subset \I^n$ with $\cH^s(A)=0$, knowing $\cH^s(W_n(\psi))$ will not tell us anything about the intersection $W_n(\psi)\cap A$. This set could be all of $A$, a non-trivial subset of $A$, or even the empty set. In Chapter \ref{fibres} we will develop a theory for the case when the set in question is an affine coordinate subspace, i.e. a subset of $\I^n$ for which one or more coordinates are fixed. Theorems \ref{thm:subspaces} and \ref{thm:lines} are the main results for the Khintchine-type theory and Theorem \ref{HDfibres} is a partial analogue regarding the Hausdorff theory. These results strengthen the classical theory in a new and natural manner.

\newpage

In Chapter \ref{dirichlet} we will define $\bi$-Dirichlet improvable vectors as the points in $\RR^n$ for which Theorem \ref{wDir} still holds if the right-hand side of \eqref{Eqn:wDir} is multiplied by a constant $c<1$. If this is possible for $c$ arbitrarily small, we call a point $\bi$-singular. Much research has been done when $\bi=(1/n,\dots,1/n)$, but less so in the weighted case. We will show that $\bi$-badly approximable vectors are $\bi$-Dirichlet improvable, thus extending a result by Davenport and Schmidt to the weighted case. The second main result of Chapter \ref{dirichlet} shows that non-$\bi$-singular vectors $\balpha$ are well suited for weighted twisted approximation in the following sense: For any $\eps>0$ and $\psi(q)=\eps q^{-1}$, the set $W_n^{\balpha}(\bi,\psi)$ has full Lebesgue measure $\lambda_n$. This generalises a theorem of Shapira.

\chapter{Ubiquity Theory}\label{ubiquity}

This chapter serves to introduce the concept of ubiquity and how it can be used to prove results similar to the theorems of Khintchine and Jarn\'ik. In fact, we are able to derive those classical statements directly from results in ubiquity theory, see Section \ref{Sec:ClassicalResults}. The main theorems in Section \ref{Sec:MainThms} are formulated for the general setting of a compact metric space equipped with a probability measure. This includes the typical case of the unit interval $\I^n$ equipped with the Lebesgue measure $\lambda_n$, but it also shows how Diophantine approximation can be interpreted in a wider sense. On the other hand, we can use ubiquity theory to gain new insight exactly for this very specific classical setting, as it allows us to work around the limitations of the Khintchine--Jarn\'ik Theorem, see Chapter \ref{fibres}. 

Throughout this chapter we will completely follow the theory presented in \cite{limsup}, which means that all the definitions and results including some proofs are taken from \cite{limsup} without any major modifications. However, we omit most of the proofs and we only adopt those parts which are relevant to our problems. While developing the theory we show how it connects to the main concepts and results of Chapter \ref{Ch:Introduction}. For simplicity we limit the strict derivation of these illustrations to the one-dimensional case. Still, the multi-dimensional case can usually be done in an analogous fashion through minor adjustments, as we remark throughout the text. The content of this chapter is based on \cite[Chapter 3]{master} but has been fully revised and heavily modified.

\section{The basic problem}

Let $(\Omega,d)$ be a compact metric space equipped with a non-atomic probability measure $m$. An \textit{atom} is a measurable set of positive measure which contains no subset of smaller but positive measure and \textit{non-atomic} means that there exist no atoms in $\Omega$ with respect to the measure $m$. Let 
\begin{equation*}
\mathcal{R}:=\{R_{\alpha}\subseteq\Omega:\alpha\in J\}
\end{equation*}
be a collection of subsets of $\Omega$ indexed by an infinite but countable set $J$, called the \textit{resonant sets}. Furthermore, let 
\begin{equation*}
\beta:J\rightarrow\mathbb{R}^+,\quad \alpha\mapsto\beta_{\alpha}
\end{equation*}
be a positive function on $J$, which we will refer to as the \textit{weight function}. We also endow this function with the condition that the set $\{\alpha\in J:\beta_{\alpha}<k\}$ has finite cardinality for any positive $k$. For a subset $A$ of $\Omega$, we define
\begin{equation*}
	\Delta(A,\delta):=\{x\in\Omega:\ d(x,A)<\delta\}
\end{equation*}
where $d(x,A)=\inf\{d(x,a):a\in A\}$. Hence, $\Delta(A,\delta)$ is the $\delta$-neighbourhood of $A$. Given a decreasing, positive valued function $\varphi:\mathbb{R}^+\rightarrow\mathbb{R}^+$, let
\begin{equation*}
	\Lambda(\varphi):=\{x\in\Omega:x\in\Delta(R_{\alpha},\varphi(\beta_{\alpha}))\text{ for infinitely many }\alpha\in J\}.
\end{equation*}

The definition of $\Lambda(\varphi)$ reveals its nature as a $\limsup$ set. This is more formally shown when we use a different construction. For $n\in\mathbb{N}$ and a fixed $k>1$, we define
\begin{equation*}
	\Delta(\varphi,n):=\bigcup\limits_{\alpha\in J_k(n)}\Delta(R_{\alpha},\varphi(\beta_{\alpha})),
\end{equation*}
where
\begin{equation*}
	J_k(n):=\{\alpha\in J:k^{n-1}<\beta_{\alpha}\leq k^n\}.
\end{equation*}
By the condition on the weight function $\beta$, the set $J_k(n)$ is finite for any given values of $k$ and $n$. Thus, $\Lambda(\varphi)$ is the set of points in $\Omega$ lying in infinitely many of the sets $\Delta(\varphi,n)$ and we get the identity
\begin{equation*}
	\Lambda(\varphi)=\limsup\limits_{n\rightarrow\infty}\Delta(\varphi,n)=\bigcap\limits_{m=1}^{\infty}\bigcup\limits_{n=m}^{\infty}\Delta(\varphi,n).
\end{equation*}
As in Chapter \ref{Ch:Introduction}, we are now interested in determining the measure theoretic properties of the set $\Lambda(\varphi)$. Since we are dealing with a $\limsup$ set in a probability space, the Borel--Cantelli Lemma (see Lemma \ref{BoCa}) directly implies that $m\left(\Lambda(\varphi)\right)=0$ if
\begin{equation}\label{Eqn:UbiConv}
	\sum\limits_{n=1}^{\infty}m\left(\Delta(\varphi,n)\right)<\infty.
\end{equation}
Obtaining a converse statement is much more intricate. This is done by the first main theorem in Section \ref{Sec:MainThms} under mild conditions on the measure $m$. Assuming a diverging sum condition as well as a `global ubiquity' hypothesis, Theorem \ref{mT1} shows that $\Lambda(\varphi)$ has strictly positive $m$ measure. Moreover, replacing `global ubiquity' by the stronger `local ubiquity' condition implies that $\Lambda(\varphi)$ has full measure, which gives us a Khintchine-type statement for this more general setting.

Regarding the case when \eqref{Eqn:UbiConv} is satisfied, the $\limsup$ set $\Lambda(\varphi)$ is a null-set with respect to the ambient measure $m$. However, as in Section \ref{Sec:Hausdorff}, we can rely on the Hausdorff measures $\cH^f$ to obtain a finer means for investigating the size of 
$\Lambda(\varphi)$. The problem of determining $\cH^f(\Lambda(\varphi))$ is much more subtle than the one regarding $m$-measure and imposes stronger conditions on the measure $m$ as well as mild conditions on the dimension function $f$.  Assuming an `$f$-volume' divergent sum condition and a `local ubiquity' hypothesis, Theorem \ref{hT1} implies that $\cH^f(\Lambda(\varphi))=\infty$. It is often the case that one can obtain a converse statement where convergence of the `$f$-volume' sum implies that $\cH^f(\Lambda(\varphi))=0$. Then, $\cH^f(\Lambda(\varphi))$ satisfies a `zero-infinity' law. In particular, this is satisfied for the Hausdorff $s$-measures $\cH^s$, allowing us to deduce the Hausdorff dimension of $\Lambda(\varphi)$.

\newpage

As a particular example, our set of interest, the set $W_n(\psi)$ of $\psi$-approximable points in $\I^n$ can be expressed in the form $\Lambda(\varphi)$ with $\varphi(q)=\psi(q)/q$ by choosing
\begin{equation}\label{Eqn:list}
	\begin{aligned}
		\Omega&=[0,1]^n,\quad J=\{(\bp,q)\in\mathbb{Z}^n\times\NN:0\leq |\bp|\leq q\},\\
		\alpha&=(\bp,q)\in J,\quad \beta_{\alpha}=q,\quad\text{and}\quad R_{\alpha}=\frac{\bp}{q}.
	\end{aligned}
\end{equation}
As usual, $d$ is the metric induced by the $\max$-norm, and we get
\begin{equation*}	
	 \Delta(R_{\alpha},\varphi(\beta_{\alpha}))=B\left(\frac{\bp}{q},\varphi(q)\right).
\end{equation*}
Hence, in this case the resonant sets are rational points $\bp/q\in\QQ^n$ and the associated sets $\Delta(R_{\alpha},\varphi(\beta_{\alpha}))$ are balls centred at those points. It follows that
\begin{equation*}
	\Delta(\varphi,m)=\bigcup\limits_{k^{m-1}<q\leq k^m}\bigcup\limits_{0\leq |\bp|\leq q}B\left(\frac{\bp}{q},\varphi(q)\right)
\end{equation*}
and 
\begin{equation*}
	W_n(\psi)=\limsup\limits_{m\rightarrow\infty}\Delta(\varphi,m).
\end{equation*}
This is basically the same characterisation as obtained in Chapter \ref{Ch:Introduction}. Here we just take the union over all $q$ in the range $(k^{m-1},k^m)$ in one step instead of doing it for each $q$ separately. The slight inconvenience of having to deal with the the function $\psi(q)/q$ instead of $\psi$ itself is due to our definition of $W_n(\psi)$. The statements in ubiquity theory are more easily formulated using conditions of the form $|\balpha-\bp/q|$ while we generally prefer the notation $\norm{q\balpha-\bp}$. However, this is easily adjusted.

Thus, the basic problems in simultaneous Diophantine approximation of determining $\lambda_n(W_n(\psi))$ or $\cH^s(W_n(\psi))$ are covered by this more general problem of investigating the measure theoretic properties of a $\limsup$ set $\Lambda(\varphi)$.

\section{Ubiquitous systems}

Let $l=(l_n)_{n\in N}$ and $u=(u_n)_{n\in N}$ be positive increasing sequences such that eventually
\begin{equation*}
	l_n<l_{n+1}\leq u_n\quad \text{ and }\quad\lim\limits_{n\rightarrow\infty}l_n=\infty.
\end{equation*} 
For obvious reasons, $l$ and $u$ will be referred to as the \textit{lower sequence} and \textit{upper sequence}, respectively. Now, given a positive, decreasing function $\varphi$, we define
\begin{equation*}
	\Delta_l^u(\varphi,n)=\bigcup\limits_{\alpha\in J_l^u(n)}\Delta(R_{\alpha},\varphi(\beta_{\alpha})),
	\vspace{-2ex}
\end{equation*}
where
\begin{equation*}
	J_l^u(n)=\{\alpha\in J:l_n<\beta_{\alpha}\leq u_n\}.
\end{equation*}
Again, the condition bestowed upon the weight function $\beta_{\alpha}$ provides finiteness of any fixed set $J_l^u(n)$ and, since $l_n$ tends to infinity, we get
\begin{equation*}
	\Lambda(\varphi)=\limsup\limits_{n\rightarrow\infty}\Delta_l^u(\varphi,n)=\bigcap\limits_{m=1}^{\infty}\bigcup\limits_{n=m}^{\infty}\Delta_l^u(\varphi,n),
\end{equation*}
independent of the choice of sequences $l$ and $u$.

Recall that our space $\Omega$ is equipped with a probability measure $m$. Throughout the whole discussion we need to assume some conditions on the measure $m$ . Firstly, any open ball centred at any arbitrary point in $\Omega$ has strictly positive measure and secondly, the measure $m$ is \textit{doubling}. That means there exists a constant $C\geq 1$ such that
\begin{equation*}
	m(B(x,2r))\leq Cm(B(x,r))
\end{equation*}
for any point $x\in\Omega$. This allows us to maintain control over the measure while shrinking or blowing up balls in $\Omega$. Furthermore, it implies that
\begin{equation*}
	m(B(x,tr))\leq C(t)m(B(x,r))
\end{equation*}
for any $t>1$, where $C(t)$ is an increasing function, which does not depend on $x$ or $r$ and satisfies $C(2^k)\leq C^k$. In the case that $m$ is doubling we will also refer to the measure space $(\Omega,d,m)$ as doubling.

We need to introduce two more properties of the measure $m$, which are not very restrictive but will be needed in the problems of determining $m(\Lambda(\varphi))$ and $\mathcal{H}^f(\Lambda(\varphi))$, respectively. For the first problem, we want to assure that balls of the same radius have roughly the same measure when they are centred at points contained in resonant sets $R_{\alpha}$ with all $\alpha$ belonging to the same set $J_l^u(n)$ for some $n$.

\textbf{(M1) } There exist constants $a, b, r_0>0$, which only depend on the sequences $l$ and $u$, such that for any $c\in R_{\alpha}, c'\in R_{\alpha'}$ with $\alpha, \alpha'\in J_l^u(n)$ and $r\leq r_0$ 
\begin{equation*}
	a\leq\frac{m(B(c,r))}{m(B(c',r))}\leq b.
\end{equation*}
When considering the Hausdorff measure problem we need a stronger condition. Namely, we want that the measure of any ball centred at a point in $\Omega$ is proportional to a fixed power of its radius.

\textbf{(M2) } There exist constants $a, b, r_0>0$ and $\delta\geq 0$ such that for any $x\in\Omega$ and $r\leq r_0$
\begin{equation*}
	ar^{\delta}\leq m(B(x,r))\leq br^{\delta}.
\end{equation*}
Such measures are often referred to as \textit{Ahlfors regular} measures. Without loss of generality we can choose $0<a<1<b$. Obviously, the condition (M2) implies (M1) with constants $a/b$, $b/a$ and $r_0$, independent of the choice of $l$ and $u$. Moreover, the condition (M2) implies that $\dim\Omega=\delta$.

\begin{remark*}
	In addition to (M1) and (M2), the original paper \cite{limsup} also introduces the so-called intersection conditions (IC). These conditions control the intersection of resonant sets $R_{\alpha}$ with balls centred at points contained in resonant sets. The intersection conditions are trivially satisfied when the resonant sets are points. We are only interested in this case and thus will not state the conditions (IC). It is worth noting that this also simplifies the conditions within the main theorems and their corollaries in Section \ref{Sec:MainThms}.
\end{remark*}

Now all the preliminaries are given and we are able to define the notion of a ubiquitous system. For this, let $\rho$ be a function with $\lim\limits_{r\rightarrow\infty}\rho(r)=0$ and let
\begin{equation*}
	\Delta_l^u(\rho,n)=\bigcup\limits_{\alpha\in J_l^u(n)}\Delta(R_{\alpha},\rho(u_n)).
\end{equation*}
In accordance with the following definitions, $\rho$ will be called the \textit{ubiquitous function}. Let $B=B(x,r)$ be an arbitrary ball with centre $x$ in $\Omega$ and radius $r\leq r_0$, where $r_0$ is given by either (M1) or (M2). Suppose there exist a function $\rho$, sequences $l$ and $u$ and an absolute constant $\kappa>0$ such that
	\begin{equation}\label{ubiq.cond}
		m(B\cap\Delta_l^u(\rho,n))>\kappa m(B)\quad \text{ for } n\geq n_0(B).
	\end{equation}
Then the pair $(\mathcal{R},\beta)$ is called a \textit{local $m$-ubiquitous system relative to $(\rho,l,u)$.} Suppose there exist a function $\rho$, sequences $l$ and $u$ and an absolute constant $\kappa>0$ such that for $n\geq n_0$, \eqref{ubiq.cond} is satisfied for $B=\Omega$. Then the pair $(\mathcal{R},\beta)$ is called a \textit{global $m$-ubiquitous system relative to $(\rho,l,u)$.}

Since $m$ is a probability measure, in the global case the condition \eqref{ubiq.cond} simply reduces to $m(\Delta_l^u(\rho,n))\geq\kappa$. Here, all we require is that the sets $\Delta_l^u(\rho,n), n\geq n_0$ cover a certain ratio of the whole space $\Omega$ with respect to the measure $m$. In the local case the same property is required to hold for any small enough ball. Clearly this condition is much stronger and it can be easily seen that local ubiquity implies global ubiquity. Simply take an arbitrary ball $B$ centered at $x\in\Omega$ with radius $\leq r_0$. Then for $n$ sufficiently large
\begin{equation*}
	m(\Delta_l^u(\rho,n))\geq m(B\cap \Delta_l^u(\rho,n))\geq \kappa m(B)\eqqcolon\kappa_1>0.
\end{equation*}
Hence, local ubiquity with constant $\kappa$ implies global ubiquity with a constant $\kappa_1$, $0<\kappa_1\leq \kappa$. The converse is not true in general. However, there is a simple and very useful condition under which global ubiquity implies local ubiquity. Namely, if
\begin{equation*}
	\lim\limits_{n\rightarrow\infty}m(\Delta_l^u(\rho,n))=1=m(\Omega).
\end{equation*}
This can be seen as follows. Suppose we have a global $m$-ubiquitous system and let $B\subseteq\Omega$ be an arbitrary ball with $m(B)=\eps>0$ (the statement holds trivially for any null set). For $n$ sufficiently large, we get
\begin{equation*}
	m(\Delta_l^u(\rho,n))>m(\Omega)-\frac{\eps}{2},\vspace{-2ex}
\end{equation*}
and thus
\begin{equation*}
	m(B\cap\Delta_l^u(\rho,n))>\frac{\eps}{2},
\end{equation*}
which shows local $m$-ubiquity.

To establish the inequality \eqref{ubiq.cond} in either case of ubiquity, we do not need the presence of the lower sequence $l$. To show this, assume for $n\geq n_0$ the modified inequality
\begin{equation}\label{mod.ubiq.cond}
	m\left(B\cap\bigcup\limits_{\alpha\in J:\beta_{\alpha}\leq u_n}\Delta(R_{\alpha},\rho(u_n))\right)\geq\kappa m(B)
	\vspace{2ex}
\end{equation}
is satisfied. Now let $t\in\mathbb{N}$. Since $\lim\limits_{r\rightarrow\infty}\rho(r)=0$ there exists $n_t\in\mathbb{N}$ such that for $n\geq n_t$ we have
\begin{equation*}
	m\left(B\cap\bigcup\limits_{\alpha\in J:\beta_{\alpha}\leq t}\Delta(R_{\alpha},\rho(u_n))\right)<\frac{1}{2}\kappa m(B).
	\vspace{2ex}
\end{equation*}
Without loss of generality we have $n_{t+1}\geq n_t+1$ and hence for every $n\in\mathbb{N}$ there is exactly one $t=t(n)$ such that $n$ lies in the interval $[n_{t(n)},n_{t(n)+1})$. Thus the lower sequence $l$ defined by $l_n= t(n)$ is increasing and diverging and for $n\geq n_0$ we have
\begin{equation*}
	m(B\cap \Delta_l^u(\rho,n))=m\left(B\cap\bigcup\limits_{\alpha\in J:l_n<\beta_{\alpha}\leq u_n}\Delta(R_{\alpha},\rho(u_n))\right)\geq\frac{1}{2}\kappa m(B)
\end{equation*}
which shows $(\mathcal{R},\beta)$ is a local $m$-ubiquitous system relative to $(\rho,l,u)$. Hence whenever \eqref{mod.ubiq.cond} is satisfied we know there exists a sequence $l$ such that we get ubiquity relative to $(\rho,l,u)$. It is also worth noting and easy to see that ubiquity relative to $(\rho,l,u)$ implies ubiquity relative to $(\rho,l,s)$ for any subsequence $s$ of $u$.
\newpage
In the case where we deal with the set $W(\psi)$ of one-dimensional $\psi$-approximable points, the considered measure $m$ is simply the one-dimensional Lebesgue measure $\lambda$, which satisfies condition (M2) with $\delta=1$. A direct application of Dirichlet's Theorem (see Theorem \ref{Dir}) yields the following statement.

\begin{lemma}\label{K-J-lemma}
	There exists a constant $k>1$ such that the pair $(\mathcal{R},\beta)$ defined in \eqref{Eqn:list} is a local $m$-ubiquitous system relative to $(\rho,l,u)$ for $l_{n+1}=u_n=k^n$ and $\rho:t\rightarrow kt^{-2}$.
\end{lemma}

\begin{proof}
	Let $A=[a,b]\subset \I=[0,1]$. By Dirichlet's Theorem, for any $x\in A$ and for any $k^n>1$ there are coprime integers $p$ and $q$ with $1\leq q\leq k^n$ such that 
\begin{equation*}
	\left\lvert x-\frac{p}{q}\right\rvert <\frac{1}{qk^n}.
\end{equation*}	
Clearly, $p/q$ has to lie in the interval $\left[a-\frac{1}{q},b+\frac{1}{q}\right]$ which implies 
\begin{equation*}
	aq-1\leq p\leq bq+1.
\end{equation*}
Hence, for a fixed $q$ there exist at most $\lambda(A)q+3$ possible values of $p$ satisfying the above inequality. For $n$ large enough it follows that
	\begin{align*}
		\lambda\left(A\cap\bigcup\limits_{q\leq k^{n-1}}\bigcup\limits_{0\leq p\leq q} B\left(\frac{p}{q},\frac{1}{qk^n}\right)\right)&\leq\sum\limits_{q\leq k^{n-1}}\frac{2}{qk^n}\left(\lambda(A)q+3\right)\\[1ex]
		&=2\sum\limits_{q\leq k^{n-1}}\left(\frac{\lambda(A)}{k^n}+\frac{3}{qk^n}\right)\\[1ex]
		&=\frac{2}{k}\lambda(A)+6\sum\limits_{q\leq k^{n-1}}\frac{1}{qk^n}\\[1ex]
		&\leq\frac{3}{k}\lambda(A)
	\end{align*}
	since the last sum tends to zero for $n\rightarrow\infty$. If we take $k\geq6$, we get
	\begin{equation*}
		\lambda\left(A\cap\bigcup\limits_{k^{n-1}<q\leq k^n}\bigcup\limits_{0\leq p\leq q} B\left(\frac{p}{q},\frac{k}{k^{2n}}\right)\right)\geq \lambda(A)-\frac{3}{k}\lambda(A)\geq\frac{1}{2}\lambda(A).
		\qedhere
	\end{equation*}
\end{proof}

The divergence parts of the theorems of Jarn\'ik and Khintchine will be an immediate consequence of Lemma \ref{K-J-lemma} and the main theorems of the ubiquity theory, which we state in the next section.

\begin{remark*}
	In the $n$-dimensional case, when dealing with the $n$-dimensional Lebesgue measure $\lambda_n$ of $W_n(\psi)$, the condition $\mathrm{(M2)}$ is satisfied with $\delta=n$. The same proof as above with adjusted constants tells us that $(\mathcal{R},\beta)$ is a local $\lambda_n$-ubiquitous system relative to $(\rho,l,u)$ for $l_{m+1}=u_m=k^m$ and $\rho:t\rightarrow kt^{-\left(1+\frac{1}{n}\right)}$, where the difference in exponent is due to Dirichlet's Theorem.
\end{remark*}

\section{The main theorems}\label{Sec:MainThms}

Next we state the main theorems of ubiquity theory. While we do not present any proofs here, it is worth mentioning that the main parts of \cite{limsup} serve to prove those results. Once the theorems are established, both the theorems of Khintchine and Jarn\'ik are rather simple consequences, which illustrates the power of this theory.

We start with some notation. Let $m$ be a measure satisfying the condition (M1) with respect to the sequences $l$ and $u$. By $B_n(r)$ we denote a generic ball of radius $r$ centred at a point of a resonant set $R_{\alpha}$ with $\alpha\in J_l^u(n)$. The condition (M1) ensures that for any ball $B(c,r)$ with $c\in R_{\alpha}$ and $\alpha\in J_l^u(n)$ we have $m(B(c,r))\asymp m(B_n(r))$.

\begin{theorem}\label{mT1}
	Let $(\Omega,d)$ be a compact metric space equipped with a probability measure $m$ satisfying condition $\mathrm{(M1)}$ with respect to lower and upper sequences $l$ and $u$. Suppose that $(\mathcal{R},\beta)$ is a global $m$-ubiquitous system relative to $(\rho,l,u)$ and that $\varphi$ is an approximating function. Furthermore, either assume that
	\begin{equation*} 
		\limsup\limits_{n\rightarrow\infty}\frac{\varphi(u_n)}{\rho(u_n)}>0
	\end{equation*}
	or assume that both
	\begin{equation*}
		\sum\limits_{n=1}^{\infty}\frac{m(B_n(\varphi(u_n)))}{m(B_n(\rho(u_n)))}=\infty
	\end{equation*}
	and for $Q$ sufficiently large
	\begin{equation*}
		\sum\limits_{s=1}^{Q-1}\frac{1}{m(B_s(\rho(u_s)))}\sum\limits_{\substack{s+1\leq t\leq Q:\\ \varphi(u_s)<\rho(u_t)}}m(B_t(\varphi(u_t)))
		\ll\left(\sum\limits_{n=1}^{Q}\frac{m(B_n(\varphi(u_n)))}{m(B_n(\rho(u_n)))}\right)^2.
	\end{equation*}
	Then, $m(\Lambda(\varphi))>0$. In addition, if any open subset of $\Omega$ is $m$-measurable and $(\mathcal{R},\beta)$ is locally $m$-ubiquitous relative to $(\rho,l,u)$, then $m(\Lambda(\varphi))=1$.
\end{theorem}

Before we are able to state the Hausdorff measure analogue of Theorem \ref{mT1} we need to introduce one more definition. Given a sequence $u$, a positive real-valued function $h$ is said to be \textit{$u$-regular} if there exists a constant $\lambda\in(0,1)$ such that for $n$ large enough the inequality
\begin{equation*}
	h(u_{n+1})\leq\lambda h(u_n)
\end{equation*}
is satisfied where $\lambda$ is independent of $n$ but may depend on $u$. If $h$ is $u$-regular then the function $h$ is eventually strictly decreasing along the sequence $u$. Furthermore, $u$-regularity implies $s$-regularity for any subsequence $s$ of $u$.

\begin{theorem} \label{hT1}
	Let $(\Omega,d)$ be a compact metric space equipped with a probability measure $m$ satisfying condition $\mathrm{(M2)}$. Suppose that $(\mathcal{R},\beta)$ is a locally $m$-ubiquitous system relative to $(\rho,l,u)$ and that $\varphi$ is an approximating function. Let $f$ be a dimension function such that $r^{-\delta}f(r)\rightarrow\infty$ as $r\rightarrow 0$ and such that $r^{-\delta}f(r)$ is decreasing. Let $g$ be the real, positive function given by
	\begin{equation*}
		g(r)= f(\varphi(r))\rho(r)^{-\delta},\quad \text{ and }\quad G=\limsup\limits_{n\rightarrow\infty} g(u_n).
	\end{equation*}
	\begin{itemize}
		\item 
			Suppose that $G=0$ and that $\rho$ is $u$-regular. Then,
			\begin{equation} \label{ineq10}
				\mathcal{H}^f(\Lambda(\varphi))=\infty\quad \text{ if }\quad \sum\limits_{n=1}^{\infty}g(u_n)=\infty.
			\end{equation}
		\item
			Suppose that $0<G\leq\infty$. Then, $\mathcal{H}^f(\Lambda(\varphi))=\infty$.
	\end{itemize}
\end{theorem}

\begin{remark*}
	Since the condition $\mathrm{(M2)}$ is independent of both the sequences $l$ and $u$ and, as shown above, the sequence $l$ is irrelevant for establishing ubiquity, results like Theorem \ref{hT1} do not require mention of the sequence $l$. However, the condition $\mathrm{(M1)}$ clearly depends on $l$ and $u$ and hence statements like Theorem \ref{mT1} rely on the used sequences and $l$ cannot be dropped from the preconditions.
\end{remark*}

\subsection{Corollaries}

In \cite{limsup}, multiple corollaries are stated for both Theorem \ref{mT1} and Theorem \ref{hT1}. These corollaries mainly illustrate how we get stronger results when certain restrictions hold. Most of the restrictions are naturally satisfied for typical applications, which makes the corollaries especially useful.

For Theorem \ref{mT1} we state a corollary which applies if the considered measure does not only satisfy condition (M1), but also condition (M2).

\begin{corollary} \label{cor2}
	Let $(\Omega,d)$ be a compact metric space equipped with a probability measure $m$ satisfying condition $\mathrm{(M2)}$. Suppose that $(\mathcal{R},\beta)$ is a global $m$-ubiquitous system relative to $(\rho,l,u)$ and that $\varphi$ is an approximating function. Moreover, assume that either $\varphi$ or $\rho$ is $u$-regular and that
	\begin{equation*}
		\sum\limits_{n=1}^{\infty}\left(\frac{\varphi(u_n)}{\rho(u_n)}\right)^{\delta}=\infty.
	\end{equation*}
	Then $m(\Lambda(\varphi))>0$. If in addition any open subset of $\Omega$ is $m$-measurable and $(\mathcal{R},\beta)$ is locally $m$-ubiquitous relative to $(\rho,l,u)$, then $m(\Lambda(\varphi))=1$.
\end{corollary}

\begin{remark*}
	 By choosing the right sequence $u$, the additional requirement that the function $\varphi$ is $u$-regular is easily satisfied in the classical example of $\psi$-approximability with $\varphi(q)=\psi(q)/q$ . Hence, this corollary is particularly useful as it allows us to prove the divergence case of Khintchine's Theorem. Furthermore, Corollary \ref{cor2} will be vital in the proof of Theorem \ref{thm:lines}, a Khintchine-type result for affine coordinate subspaces (see Section \ref{sec:ubiquity}).
\end{remark*}

Next we turn to subsequent results of Theorem \ref{hT1}. While Theorem \ref{hT1} itself is all we need to complete the proof of Jarn\'ik's Theorem, the following statement is formulated in a slightly simpler way and will be referred to in the proof of Theorem \ref{HDfibres}, a partial Jarn\'ik-type analogue to Theorems \ref{thm:subspaces} and \ref{thm:lines}.

\begin{corollary} \label{cor4}
	Let $(\Omega,d)$ be a compact metric space equipped with a probability measure $m$ satisfying condition $\mathrm{(M2)}$. Suppose that $(\mathcal{R},\beta)$ is a locally $m$-ubiquitous system relative to $(\rho,l,u)$ and that $\varphi$ is an approximating function. For $0<s<\delta$, define
	\begin{equation*}
		g(r)=\varphi(r)^{s}\rho(r)^{-\delta},\quad \text{ and }\quad G=\limsup\limits_{n\rightarrow\infty}g(u_n).
	\end{equation*}
	\begin{itemize}
		\item
			Suppose that $G=0$ and that either $\varphi$ or $\rho$ is $u$-regular. Then
			\begin{equation*}
				\mathcal{H}^s(\Lambda(\varphi))=\infty\quad \text{ if }\quad \sum\limits_{n=1}^{\infty}g(u_n)=\infty.
			\end{equation*}
		\item
			Suppose that $0<G\leq\infty$. Then $\mathcal{H}^s(\Lambda(\varphi))=\infty.$
	\end{itemize}
\end{corollary}

Note that Corollary \ref{cor4} only applies to Hausdorff $s$-measures rather than the more general class of measures $\cH^f$. This means that the growth conditions on the dimension function in Theorem \ref{hT1} are trivially satisfied and allows us to weaken the regularity condition on $\rho$. As a consequence of the second part of Corollary \ref{cor4} we obtain the following dimension formulae for $\Lambda(\varphi)$.

\begin{corollary}\label{cor5}
	Let $(\Omega,d)$ be a compact metric space equipped with a probability measure $m$ satisfying condition $\mathrm{(M2)}$. Suppose that $(\mathcal{R},\beta)$ is a locally $m$-ubiquitous system relative to $(\rho,l,u)$ and that $\varphi$ is an approximating function.
	\begin{itemize}
		\item
			If $\ \lim\limits_{n\rightarrow\infty}\varphi(u_n)/\rho(u_n)=0$, then
			\vspace{-2ex} 
			\begin{equation*}
				\dim\Lambda(\varphi)\geq \sigma\delta,\quad \text{ where }\quad \sigma:=\limsup\limits_{n\rightarrow\infty}\frac{\log\rho(u_n)}{\log\varphi(u_n)}.
			\end{equation*}			 
			Furthermore, if $\ \liminf\limits_{n\rightarrow\infty}\rho(u_n)/\varphi(u_n)^{\sigma}<\infty$, then $\mathcal{H}^{\sigma\delta}(\Lambda(\varphi))=\infty$.
		 \item
		 	If $\ \limsup\limits_{n\rightarrow\infty}\varphi(u_n)/\rho(u_n)>0,$ then $0<\mathcal{H}^{\delta}(\Lambda(\varphi))<\infty$ and so $\dim\Lambda(\varphi)=\delta$.
	\end{itemize}
\end{corollary}

\section{The classical results}\label{Sec:ClassicalResults}

Next we show how we can derive the divergence parts of the theorems of Khintchine and Jarn\'ik as consequences of Lemma \ref{K-J-lemma} and the statements above. We will also make use of the following observation.

\begin{lemma}[Cauchy condensation test]\label{sum-lemma}
	Let $\phi:\mathbb{R}^+\rightarrow\mathbb{R}^+$ be a positive decreasing function and let $k>1$. Then
	\begin{equation}\label{Eqn:Cauchy}
		\sum\limits_{q=1}^{\infty}\phi(q)=\infty\quad \Longleftrightarrow\quad \sum\limits_{n=1}^{\infty} k^n\phi(k^n)=\infty.
	\end{equation}
\end{lemma}

\begin{proof}
	Fix an integer $k>1$. Any $q\in\NN$ is contained in an interval of the form $[k^n,k^{n-1})$ with $0\leq n\in\ZZ$. The function $\phi$ is decreasing, hence
	\begin{equation*}
		\phi(k^n)\geq\phi(q)\geq\phi(k^{n+1}).
	\end{equation*}
	The interval $[k^n,k^{n-1})$ contains
	\begin{equation*}
		k^{n+1}-k^n=k^n(k-1)=k^{n+1}\left(1-\frac{1}{k}\right)
	\end{equation*}
	integer points. Thus, we get upper and lower bounds for $\sum_{q\in\NN}\phi(q)$ by
	\begin{align*}
		(k-1)\sum\limits_{n=0}^{\infty}k^n\phi(k^n)&=\sum\limits_{n=0}^{\infty}(k^{n+1}-k^n)\phi(k^n)\\[1ex]
		&\geq\sum\limits_{q=1}^{\infty}\phi(q)\\[1ex]
		&\geq\sum\limits_{n=0}^{\infty}(k^{n+1}-k^n)\phi(k^{n+1})=\left(1-\frac{1}{k}\right)\sum\limits_{n=1}^{\infty}k^n\phi(k^n).
	\end{align*}
	The addition of the term $\phi(1)$ to the first series does not affect whether the sum converges or diverges. Hence, the sum $\sum_{q\in\NN}\phi(q)$ is essentially bounded from above and below by constant multiples of $\sum_{k\in\NN}k^n\phi(k^n)$, which implies \eqref{Eqn:Cauchy}.
\end{proof}

\newpage

\begin{remark*}
	Note that the argument above only proves \eqref{Eqn:Cauchy} for integers $k$. However, it can easily be extended to any real number $k>1$ by making use of the fact that
	\begin{equation*}
		\sum\limits_{q=1}^{\infty}\phi(q)=\infty\quad \Longleftrightarrow\quad \int\limits_{t=1}^{\infty}\phi(t)=\infty
	\end{equation*}
	for any positive decreasing function $\phi$ and by splitting the integral domain into intervals of the form $[k^n,k^{n-1})$ as above.
\end{remark*}

Now we can turn to the proof of the divergence case of Khintchine's Theorem in the one-dimensional case, see Theorem \ref{Khintchine}.

\begin{corollary}
	Let $W(\psi)$ be the set of $\psi$-approximable points in $[0,1]$. Then $\lambda(W(\psi))=1$ if $\ \sum\limits_{q=1}^{\infty}\psi(q)=\infty$.
\end{corollary}

\begin{proof}
	Remember that in Lemma \ref{K-J-lemma} we have established that the pair $(\mathcal{R},\beta)$ defined in \eqref{Eqn:list} is locally ubiquitous with respect to $(\rho,l,u)$, where $l_{n+1}=u_n=k^n$ and $\rho:r\rightarrow k r^{-2}$. Clearly, $\rho$ is $u$-regular, so we can apply Corollary \ref{cor2} to $\varphi(q)=\psi(q)/q$. This tells us that $m(W(\psi))=1$, if
	\begin{equation*}
		k\sum\limits_{n=1}^{\infty}\frac{\psi(k^n)k^{2n}}{k^n}=k\sum\limits_{n=1}^{\infty}\psi(k^n)k^n=\infty.
		\vspace{2ex}
	\end{equation*}
	Now the statement directly follows by applying Lemma \ref{sum-lemma}.
\end{proof}

\begin{remark*}
	In the $n$-dimensional case, by adjusting $\rho$ and $\delta$ appropriately, we get in an analogous fashion that $\lambda_n(W(\psi))=1$ if 
	\begin{equation*}
		k^n\sum\limits_{l=1}^{\infty}\frac{\psi(k^l)^n \left(\left(k^l\right)^{1+\frac{1}{n}}\right)^n}{k^{ln}}=k^n\sum\limits_{l=1}^{\infty}\psi(k^l)^n k^{l} =\infty,
	\end{equation*}
	which by Lemma \ref{sum-lemma} is equivalent to $\sum_{q\in\NN}\varphi(q)^n=\infty$.
\end{remark*}

This proves the divergence part of Khintchine's Theorem for arbitrary dimensions and hence completes the proof of the theorem since we already obtained the convergence part in Chapter \ref{Ch:Introduction}.

The divergence part of Jarn\'ik's Theorem follows directly from Theorem \ref{hT1}. Note that we will not differ between the two cases $G=0$ and $G>0$ since the function $\rho$ is $u$-regular and $G>0$ obviously implies divergence in \eqref{ineq10}.

\begin{samepage}
\begin{corollary}
	Let $W(\psi)$ be the set of $\psi$-approximable points in $[0,1]$ and $s\in(0,1)$. Then 
	\begin{equation*}
		\mathcal{H}^s(W(\psi))=\infty\quad \text{ if }\quad \sum\limits_{q=1}^{\infty}q^{1-s}\psi(q)^s=\infty.
	\end{equation*}		
\end{corollary}
\end{samepage}

\begin{proof}
	We apply Theorem \ref{hT1} to the situation where $\varphi(q)=\psi(q)/q$, $\delta=1$, $u_n=k^n$, the $u$-regular ubiquitous function is $\rho:r\rightarrow kr^{-2}$ and the dimension function is given by $r\rightarrow r^s$. Since $s<1$, the function $r\rightarrow r^{s-1}$ is decreasing and diverges for $r\rightarrow 0$. We get that
	\begin{equation*}
		g(r)=k\psi(r)^s r^2/r^s
	\end{equation*}		
	and hence, $\mathcal{H}^s(W(\psi))=\infty$, if
	\begin{equation*}
		k\sum\limits_{n=1}^{\infty}\frac{\psi(k^n)^s k^{2n}}{k^{ns}}=k\sum\limits_{n=1}^{\infty},\psi(k^n)^s(k^n)^{1-s}k^n=\infty
	\end{equation*}
	which again due to Lemma \ref{sum-lemma} is exactly the case when $\sum_{q\in\NN}q^{1-s}\psi(q)^s=\infty$.
\end{proof}

This completes the proof of Jarn\'ik's Theorem in dimension $1$. In the $n$-dimensional case it can be established analogously that $\mathcal{H}^s(W(\psi))=\infty$ for $s\in(0,n)$ if
	\begin{equation*}
		k\sum\limits_{l=1}^{\infty}\frac{\psi(k^l)^s k^{(n+1)l}}{k^{ls}}=k\sum\limits_{l=1}^{\infty}\psi(k^l)^s(k^l)^{n-s}k^l=\infty,
	\end{equation*}
which happens precisely when $\sum_{q\in\NN}q^{n-s}\psi(q)^s$ diverges.

\begin{remark*}
	These proofs illustrate how easily both the theorems of Khintchine and Jarn\'ik follow once the major theorems of ubiquity theory are established. Note that Jarn\'ik's Theorem can be directly deduced from Khintchine's Theorem using the Mass Transference Principle as shown in Section \ref{secMTP}. Indeed, given the Mass Transference Principle, the real power of ubiquity is that it enables us to establish Khintchine type theorems with respect to the ambient measure.
\end{remark*}

\chapter{Rational approximation of affine coordinate subspaces of Euclidean space}\label{fibres}
\chaptermark{Approximation on fibres}

Khintchine's Theorem is sufficient to determine the Lebesgue measure $\lambda_n(W_n(\psi))$ for any given approximating function $\psi$, but it fails to answer more specific questions arising in a natural manner. For instance, if we consider an approximating function $\psi$ such that $\sum_{q\in\NN}\psi(q)^2=\infty$, we know that almost every pair $(\alpha,\beta)\in\I^2$ is $\psi$-approximable. Now we would like to know what happens if we fix the coordinate $\alpha$. Is it still true that almost every pair $(\alpha,\beta)\in\{\alpha\}\times\I$ is $\psi$-approximable? Essentially this means we want to obtain a one-dimensional Lebesgue measure statement from a two-dimensional setting. Khintchine's Theorem implies that
\begin{equation*}
	\lambda\left(\{\alpha\}\times\I\cap W_2(\psi)\right)=1\quad \text{ for }\lambda\text{-almost all }\alpha\in\I.
\end{equation*}
However, given any fixed $\alpha\in\I$, the set $\{\alpha\}\times\I$ is a null-set with respect to two-dimensional Lebesgue measure and so, a priori, we cannot say anything about the intersection $\{\alpha\}\times\I\cap W_2(\psi)$. This consideration is easily extended to higher dimensions.

\section{Preliminaries and the main results}

We want to investigate the following question. Let $\ell$ and $m$ be positive integers with $\ell+m=n$, and let $\psi$ be an approximating function. Fix $\balpha\in\I^{\ell}$ and define the \textit{fibre above $\balpha$} by
\begin{equation*}
	\F_n^{\balpha}:=\{\balpha\}\times\I^m\subset \I^n.
\end{equation*}
Then, is it true that
\begin{equation}\label{Eqn:Khin?}
	\lambda_{m}(\F_n^{\balpha}\cap W_n(\psi))=
	\begin{dcases}
		1,\quad &\text{ if }\quad\sum\limits_{q=1}^{\infty}\psi(q)^n=\infty \\[2ex]
		0,\quad &\text{ if }\quad\sum\limits_{q=1}^{\infty}\psi(q)^n<\infty
	\end{dcases}
	\quad\quad\textbf{?}
\end{equation}
Upon choosing a rational vector $\balpha$, it is easily established that the convergence part of \eqref{Eqn:Khin?} cannot be true. To see this, fix $\alpha=a/b\in\QQ$ with $b\in\NN$. Dirichlet's Theorem implies that for any $\beta\in\I$ there exist infinitely many $q\in\NN$ such that $\norm{q\beta}<1/q$. Hence,
\begin{equation*}
	\norm{bq\alpha}=0<\frac{b^2}{bq}=\psi(bq)
	\vspace{-2ex}
\end{equation*}
and
\begin{equation*}	
	\norm{bq\beta}<\frac{b}{q}=\frac{b^2}{bq}=\psi(bq),
\end{equation*}
where $\psi:\NN\rightarrow\RR^+$ is the approximating function given by $\psi(q)=b^2/q$. This shows that every point $(\alpha,\beta)$ in the set $\F_2^{\alpha}=\{\alpha\}\times\I\subset\RR^2$ is $\psi$-approximable and thus
\begin{equation*}
	\lambda(\F_2^{\alpha}\cap W_2(\psi))=1.
	\vspace{-2ex}
\end{equation*}
On the other hand, $\psi$ satisfies
\begin{equation*}
	\sum\limits_{q=1}^{\infty}\psi(q)^2=\sum\limits_{q=1}^{\infty}b^4q^{-2}<\infty.
\end{equation*}
Analogous examples work for any choice of $n$ and $\ell$. This shows it is worth treating the two sides of the problem separately. We will concentrate on the divergence side. A similar argument to above shows that any rational vector $\balpha$ satisfies the divergence part of \eqref{Eqn:Khin?}. Again, for simplicity we will just consider the case when $n=2$ and $\alpha=a/b\in\QQ$. Assume that $\psi$ is an approximating function satisfying $\sum_{q\in\NN}\psi(q)^2=\infty$ and define the function $\bar{\psi}$ by
\begin{equation*}
	\bar{\psi}(q)=\frac{\psi(bq)}{b}.
\end{equation*}
Clearly, monotonicity of $\psi$ implies monotonicity of $\bar{\psi}$ and it follows that
\begin{align*}
	\sum\limits_{q=1}^{\infty}\bar{\psi}(q)&=\sum\limits_{q=1}^{\infty}\frac{\psi(bq)}{b}=\frac{1}{b^2}\sum\limits_{q=1}^{\infty}b\psi(bq)\\[1ex]
	&\geq\frac{1}{b^2}\sum\limits_{q=1}^{\infty}\psi(bq)+\psi(bq+1)+\dots+\psi(bq+b-1)\\[1ex]
	&=\frac{1}{b^2}\sum\limits_{q=b}^{\infty}\psi(q)=\infty.
\end{align*}
Hence, by Khintchine's Theorem, almost every $\beta\in\I$ is $\bar{\psi}$ approximable. This implies that there are infinitely many $q\in\NN$ satisfying $\norm{q\beta}<\bar{\psi}(q)$ and thus
\begin{equation*}
	\norm{bq\beta}\leq b\norm{q\beta}<b\bar{\psi}(q)=\frac{b\psi(q)}{b}=\psi(bq)
\end{equation*}
for infinitely many $q\in\NN$. On the other hand, $\norm{bq\alpha}=0<\psi(bq)$ for any $q\in\NN$ and so for almost every $\beta\in\I$ the pair $(\alpha,\beta)$ is $\psi$-approximable.

\begin{remark*}
	Note that the above argument only make use of the property that $\sum_{q\in\NN}\psi(q)=\infty$ rather than the stronger assumption that $\sum_{q\in\NN}\psi(q)^2=\infty$. This argument extends to arbitrary dimensions and illustrates that picking a rational vector $\balpha=\boldsymbol{a}/b\in\QQ^{\ell}$ essentially reduces the problem of Diophantine approximation within $\F_n^{\balpha}$ to the $m$-dimensional case of Khintchine's Theorem. Throughout the rest of this chapter we will assume that $\balpha\notin\QQ^{\ell}$.
\end{remark*}

An affine coordinate subspace $\{\balpha\}\times\RR^m \subseteq\RR^n$ is said to be of \emph{Khintchine type for divergence} if $\F_n^{\balpha}$ satisfies the divergence case of \eqref{Eqn:Khin?}, i.e. if for any approximating function $\psi:\RR\to\RR^+$ such that $\sum_{q\in\NN}\psi(q)^n$ diverges, almost every point on $\{\balpha\}\times\RR^m$ is $\psi$-approximable. Intuitively, $\{\balpha\}\times\RR^m$ is of Khintchine type for divergence if its typical points behave like the typical points of Lebesgue measure with respect to the divergence case of Khintchine's theorem. The recent article~\cite{hyperplanes} addresses the issue for certain affine coordinate hyperplanes in $\RR^n$, where $n\geq 3$. There, sufficient conditions are given for a hyperplane to be of Khintchine type for divergence. However, the techniques of \cite{hyperplanes} are not capable of handling subspaces of codimension greater than one, nor those of large Diophantine type. Here, we overcome these difficulties by taking a different approach. We show that affine coordinate subspaces of dimension at least two are of Khintchine type for divergence, and we make substantial progress on the one-dimensional case. All of the following, unless otherwise noted, is joint work with Ram\'irez and Simmons \cite{mine}.

\begin{remark*}
	The question \eqref{Eqn:Khin?} can also be extended to other types of manifolds. A manifold $\mathcal{M}\subset\RR^n$ is called \textit{non-degenerate} if it is sufficiently curved to deviate from any hyperplane. Clearly, this differs from the case of affine (coordinate) subspaces, which are often referred to as \textit{degenerate} manifolds. It is widely believed that non-degeneracy is the right criterion for a manifold to be endowed with to allow a Khintchine-type theorem for $\mathcal{M}\cap W_n(\psi)$ in both the convergence and divergence case (see \cite{DAaspects} for more background information).
  
	\begin{Conjecture}\label{Con:DreamThm}
  		Let $\mathcal{M}$ be a $d$-dimensional non-degenerate submanifold of $\RR^n$, let $\mu_d$ be the normalised $d$-dimensional Lebesgue measure induced on $\mathcal{M}$ and let $\psi$ be an approximating function. Then
  		\begin{equation}\label{Eqn:DreamThm}
			\mu_d(\mathcal{M}\cap W_n(\psi))=
			\begin{dcases}
				1,\quad &\text{ if }\quad\sum\limits_{q=1}^{\infty}\psi(q)^n=\infty, \\[2ex]
				0,\quad &\text{ if }\quad\sum\limits_{q=1}^{\infty}\psi(q)^n<\infty.
			\end{dcases}
			\vspace{1ex}
		\end{equation}
	\end{Conjecture}
	
	The following list shows the various contributions that have been made towards Conjecture \ref{Con:DreamThm}.
	
	\begin{itemize}
		\item 
			\textit{Extremal manifolds.} A submanifold $\mathcal{M}$ of $\RR^n$ is called \textit{extremal} if
			\begin{equation*}
				\mu_d(\mathcal{M}\cap W_n(\tau))=0\quad \text{for all }\tau>\frac{1}{n}.
			\end{equation*}
			Note that $\mathcal{M}\cap W_n(\tau)=\mathcal{M}$ for $\tau\leq 1/n$ by Dirichlet's Theorem. Hence, a manifold is extremal if and only if \eqref{Eqn:DreamThm} holds for any approximation function $\psi$ of the form $\psi:q\mapsto q^{-\tau}$. Kleinbock and Margulis proved that any non-degenerate submanifold $\mathcal{M}$ of $\RR^n$ is extremal \cite{KleinbockMargulis}.
			
		\item
			\textit{Planar curves.} Conjecture \ref{Con:DreamThm} is true when $\mathcal{M}$ is a non-degenerate planar curve, i.e. when $n=2$ and $d=1$. The convergence part of \eqref{Eqn:DreamThm} was established in \cite{VaughanVelani} and strengthened in \cite{Zorin}. The divergence part was proved in \cite{BDVplanarcurves}.

		\item
			\textit{Beyond planar curves} The divergence case of Conjecture \ref{Con:DreamThm} is true for analytic non-degenerate submanifolds of $\RR^n$ \cite{Bermanifolds} as well as non-degenerate curves and manifolds that can be `fibred' into such curves \cite{BVVZ2}. This category includes non-degenerate manifolds which are smooth but not necessarily analytic. The convergence case has been shown to be true for non-degenerate manifolds of high enough dimension $d$ relative to $n$ \cite{BVVZ}, \cite{Simmons-convergence-case}. Earlier work proved the convergence part of Conjecture \ref{Con:DreamThm} for manifolds satisfying a geometric curvature condition \cite{Dodson}.	
	\end{itemize}	
\end{remark*}
		
Coming back to the present problem, we prove the following:

\begin{theorem}\label{thm:subspaces}
  Every affine coordinate subspace of Euclidean space of dimension at
  least two is of Khintchine type for divergence.
\end{theorem}

\begin{remark*}
  Combining Theorem \ref{thm:subspaces} with Fubini's theorem shows
  that every submanifold of Euclidean space which is foliated by
  affine coordinate subspaces of dimension at least two is of
  Khintchine type for divergence. For example, given $a,b,c\in\RR$
  with $(a,b)\neq (0,0)$, the three-dimensional affine subspace
  \[
    \{(x,y,z,w) : a x + b y = c\} \subseteq \RR^4
  \]
  is of Khintchine type for divergence, being foliated by the
  two-dimensional affine coordinate subspaces $\left\{(x,y)\times\RR^2:x,y\in\RR,\ a x + b y = c\right\}$.
\end{remark*}

The reason for the restriction to subspaces of dimension at least two is that Gallagher's Theorem is used in the proof. Recall that this removes the monotonicity condition from Khintchine's Theorem, but only in dimension two and higher. Regarding one-dimensional affine coordinate subspaces, we have the following theorem.

\begin{theorem}\label{thm:lines}
  Consider a one-dimensional affine coordinate subspace
  $\{\balpha\}\times\RR \subseteq \RR^n$, where $\balpha\in\RR^{n - 1}$.
  \begin{itemize}
  \item[(i)] If the dual Diophantine type of $\balpha$ is strictly greater
    than $n$, then $\{\balpha\}\times \RR$ is contained in the set of very
    well approximable vectors
    \[
      \mathrm{VWA}_n = \{\bx : \exists \eps > 0 \; \exists^\infty
      q\in\NN \;\; \norm{q\bx} < q^{-1/n - \eps}\}.
    \]
  \item[(ii)] If the dual Diophantine type of $\balpha$ is strictly less
    than $n$, then $\{\balpha\}\times\RR$ is of Khintchine type for
    divergence.
  \end{itemize}
\end{theorem}

\noindent Here the \emph{dual Diophantine type} of a point
$\balpha\in\RR^\ell$ is the number
\begin{equation}
  \label{taudef}
  \tau_D(\balpha)=\sup\left\{\tau\in\mathbb{R}^+: \norm{\bq\cdot\balpha} < |\bq|^{-\tau} \textrm{ for i.m. } \bq\in\ZZ^{\ell}\backslash\{\boldsymbol{0}\} \right\}.
\end{equation}

\begin{remark*}
  The inclusion $\{\balpha\}\times\RR \subseteq \mathrm{VWA}_d$ in part
  (i) is philosophically ``almost as good'' as being of Khintchine
  type for divergence, since it implies that for sufficiently ``nice''
  functions $\psi:\NN\to\RR^+$ such that $\sum_{q\in\NN} \psi(q)^n$
  diverges, almost every point on $\{\balpha\}\times\RR$ is
  $\psi$-approximable. For example, call a function $\psi$ \emph{good}
  if for each $c > 0$, we have either $\psi(q)\geq q^{-c}$ for all $q$
  sufficiently large or $\psi(q) \leq q^{-c}$ for all $q$ sufficiently
  large. Then by the comparison test, if $\psi$ is a good function
  such that $\sum_{q\in\NN} \psi(q)^n$ diverges, then for all
  $\eps > 0$, we have $\psi(q) \geq q^{-1/n - \eps}$ for all $q$
  sufficiently large and thus, by Theorem \ref{thm:lines}(i), every
  point of $\{\balpha\}\times\RR$ is $\psi$-approximable. The class of
  good functions includes the class of \emph{Hardy $L$-functions}
  (those that can be written using the symbols $+,-,\times,\div,\exp$,
  and $\log$ together with constants and the identity function), see
  \cite[Chapter III]{Hardy} or \cite{AvdD} for further
  discussion and examples.
\end{remark*}

Taken together, parts (i) and (ii) of Theorem \ref{thm:lines} imply
that if $\psi$ is a Hardy $L$-function such that
$\sum_{q\in\NN} \psi(q)^n$ diverges, and if $\balpha\in\I^{n-1}$ is a vector
whose dual Diophantine type is not exactly equal to $d$, then almost
every point of $\{\balpha\}\times\RR \subseteq \RR^d$ is
$\psi$-approximable. This situation is somewhat frustrating, since it
seems strange that points in $\I^{n - 1}$ with dual Diophantine type
exactly equal to $n$ should have any special properties (as opposed to
those with dual Diophantine type $(n - 1)$, which are the ``not very
well approximable'' points). However, it seems to be impossible to
handle these points using our techniques.

Even if $\sum_{q\in\NN} \psi(q)^n$ converges, we might be interested in the set of $\psi$-approximable points. Jarn\'ik's Theorem gives us a means to determine the Hausdorff $s$-measure of $W_n(\psi)$ for any given $s<n$ as well as its Hausdorff dimension. Still, as in the above case, given a base point $\balpha$ in $\I^{\ell}$, we cannot say much about the intersection $\F_n^{\balpha}\cap W_n(\psi)$. Clearly, we are only interested in base points $\balpha\in W_{\ell}(\psi)$, as otherwise no point in $\balpha\times \I^m$ can be $\psi$-approximable. Focussing on the case where $\psi$ is a monomial of the form $\psi(q)=q^{-\tau}$, we have proved the following result:

\begin{theorem}\label{HDfibres}
	Let $\ell,m\in\NN$ with $\ell+m=n$ and $\balpha\in W_{\ell}(\tau)\subset\I^{\ell}$ and define
	\begin{equation*}
		s^{\balpha}_n(\tau):=\dim(\F^{\balpha}_n\cap W_n(\tau)).
		\vspace{-2ex}
	\end{equation*}		
	Then
	\begin{equation*}
		s^{\balpha}_n(\tau)\geq s_n^{\ell}(\tau):=
		\begin{dcases}
			m\ &\text{ if }\quad \tau\leq\frac{1}{n},\\[2ex]
			\frac{n+1}{\tau+1}-\ell\ &\text{ if }\quad \frac{1}{n}<\tau\leq\frac{1}{\ell},\\[2ex]
			\frac{m}{\tau+1}\ &\text{ if }\quad \tau>\frac{1}{\ell}.
		\end{dcases}
		\vspace{2ex}
	\end{equation*}
	Furthermore, $\cH^{s_n^{\ell}(\tau)}(\F_n^{\balpha}\cap W_n(\tau))=\cH^{s_n^{\ell}(\tau)}(\I^m)$.
\end{theorem}

We will show that for most base points $\balpha\in W_{\ell}(\tau)$ we get the exact dimension result
\begin{equation*}
	s^{\balpha}_n(\tau)= s_n^{\ell}(\tau).
\end{equation*}
In the first case, where $\tau\leq 1/n$, Theorem \ref{HDfibres} is trivially true with $s_n^{\balpha}(\tau)=m$ for all $\balpha\in\I^{\ell}$. This follows directly from Dirichlet's Theorem. In the second case, where $1/n<\tau\leq 1/\ell$, we still have $W_{\ell}(\tau)=\I^{\ell}$, so we can consider any $\balpha\in\I^{\ell}$. In the third case, where $\tau> 1/\ell$, the set of suitable base points $W_{\ell}(\tau)$ is a proper subset of $\I^{\ell}$ of dimension $\frac{\ell+1}{\tau+1}$, satisfying
	\begin{equation*}
		\cH^{\frac{\ell+1}{\tau+1}}(W_{\ell}(\tau))=\infty,
	\end{equation*}		
	by Jarn\'ik's Theorem. In both of the latter cases there are base points $\balpha\in W_{\ell}(\tau)$, for which $s_n^{\balpha}(\tau)>s_n^{\ell}(\tau)$. However, this set of exceptions is ``small" as shown by the following result.
	
	\begin{corollary}\label{Cor:HD}
	If $1/n<\tau\leq 1/\ell$, then the collection of points $\balpha\in \I^{\ell}$ such that
	\begin{equation}\label{Eqn:HDslice}
		s_n^{\balpha}=\dim(\F_n^{\balpha}\cap W_n(\tau))>\frac{n+1}{\tau+1}-\ell
	\end{equation}		
	is a null-set with respect to the Lebesgue measure $\lambda_{\ell}$. If $\tau> 1/\ell$, then the collection of points $\balpha\in \I^{\ell}$ such that $s^{\balpha}_n(\tau)>\frac{m}{\tau+1}$ is a zero set with respect to the measure $\mathcal{H}^{\frac{\ell+1}{\tau+1}}$.
	\end{corollary}
	
	We will prove Corollary \ref{Cor:HD} in Section \ref{sec:HDrmk}. We have not investigated the set of exceptions any further, but it trivially includes rational points and, depending on $\tau$, points with rational dependencies between different coordinates. 

\section[Proof of Theorem~\ref{thm:subspaces}: Subspaces of dimension at least two]{Proof of Theorem~\ref{thm:subspaces}:\\ Subspaces of dimension at least two}\label{Sec:Subspaces2}

Consider an affine coordinate subspace $\{\balpha\}\times\RR^m$, where
$\balpha\in\I^\ell$ and $\ell + m = n$. Given a non-increasing function
$\psi:\NN\to\RR^+$, for each $M,N$ with $M < N$ let
\begin{equation*}
  Q_{\psi}(M,N):= \Abs{\left\{M < q\leq N : \norm{q\balpha} < \psi(N) \right\}},
\end{equation*}
and write $Q_{\psi}(N) : = Q_{\psi}(0, N)$. Since any real number $\delta > 0$ may
be thought of as a constant function, the expression $Q_\delta(M,N)$
makes sense. 

\begin{lemma}\label{lem:count}
  For all $N\in\NN$,\vspace{-1ex}
  \begin{equation*}
    Q_\delta(N) = \Abs{\left\{q\in\NN : \norm{q\balpha} < \delta, q\leq N\right\}}\geq N\delta^\ell - 1.
  \end{equation*}
\end{lemma}

\begin{proof}
  Let
  \begin{equation*}
  \cQ_\delta(N) = \left\{q\in\NN : \norm{q\balpha} < \delta, q\leq
    N\right\},
  \end{equation*}
  so that $Q_\delta(N) = \Abs{\cQ_\delta(N)}$. We first claim that
  $Q_\delta(N)\geq Q_{\frac{\delta}{2},\bgamma}(N) - 1$ for any
  $\bgamma\in\RR^\ell$ and $N\in\NN$, where
  \begin{equation*}
    \cQ_{\delta,\bgamma}(N) : = \left\{q\in\NN : \norm{q\balpha + \bgamma} < \delta, q\leq N\right\}
  \end{equation*}
  and $Q_{\delta,\bgamma}(N) = \Abs{\cQ_{\delta,\bgamma}(N)}$. Simply
  notice that if $q_1 < q_2\in\cQ_{\frac{\delta}{2},\bgamma}(N)$, then,
  by the triangle inequality, $q_2 -
  q_1\in\cQ_{\delta}(N)$.
  Therefore, letting $q_0 = \min\cQ_{\frac{\delta}{2},\bgamma}(N)$, we
  have that
  \begin{equation*}
  	\cQ_{\frac{\delta}{2},\bgamma}(N)-q_0:=\left\{q-q_0:q\in\cQ_{\frac{\delta}{2},\bgamma}(N)\right\}\subseteq\cQ_\delta(N)\cup\{0\},
  \end{equation*}
  which implies that 
  \begin{equation*}
  	Q_\delta(N)\geq Q_{\delta/2,\bgamma}(N) - 1.
  \end{equation*} 
  
  Now we show that for any $N\in\NN$ there is some $\bgamma$ such that
  $Q_{\frac{\delta}{2},\bgamma}(N)\geq N\delta^\ell$. Notice that
  \begin{align*}
    \int_{\TT^\ell} Q_{\frac{\delta}{2},\bgamma}(N)\,d\bgamma & = \int_{\TT^\ell} \sum_{q = 1}^N \mathbf{1}_{\left( - \frac{\delta}{2}, \frac{\delta}{2}\right)^\ell}(q\bx + \bgamma)\,d\bgamma = N\delta^\ell,
  \end{align*}
  where $\TT^{\ell}=\RR^{\ell}/\ZZ^{\ell}$ is the $\ell$-dimensional torus and $\mathbf{1}$ is the characteristic function.
  Therefore, $Q_{\frac{\delta}{2},\bgamma}(N)$ must take some value
  $\geq N \delta^\ell$ at some $\bgamma$. Combining this with the
  previous paragraph proves the lemma.
\end{proof}

\begin{samepage}
\begin{lemma}
  \label{lem:series}
  Let $\balpha\in\I^{\ell}$ and $m\in\NN$ with $\ell+m=n$. Suppose that $\psi:\NN\to\RR^+$ is an approximating function such that $\sum_{q\in\NN}\psi(q)^n$ diverges. Then,
  \begin{equation}
    \label{psikseries}
    \sum_{\norm{q\balpha} < \psi(q)}\psi(q)^{m} = \infty.
    \vspace{2ex}
  \end{equation}
\end{lemma}
\end{samepage}

\begin{remark*}
	The index $\norm{q\balpha}<\psi(q)$ in the above sum is short for ``$q\in\NN,$ $\norm{q\balpha}<\psi(q)$'' and will be used throughout this chapter.
\end{remark*}

\begin{proof}
  We may assume without loss of generality that $\psi$ is a step function of the form $\psi(q) = 2^{-k_q}$
  where $k_q\in\NN$. Indeed, given any $\psi$ as in the theorem
  statement, we can let $k_q = \ceil{-\log_2\psi(q)}$ and replace
  $\psi(q)$ with $2^{-k_q}$. For any $q\in\NN$, we will have reduced $\psi(q)$ by no more
  than a factor of $\frac{1}{2}$, hence preserving the divergence of the
  series $\sum_{q\in\NN}\psi(q)^n$. On the other hand, since this modified function is less
  than the old $\psi$, divergence of \eqref{psikseries} for the new
  function implies divergence of \eqref{psikseries} for the old $\psi$.
  Now,
  \begin{align*}
    \sum_{\norm{q\balpha} < \psi(q)}\psi(q)^{m} &\geq \sum_{k\in\NN}\psi(2^k)^{m}\Abs{\left\{2^{k - 1} < q\leq 2^k : \norm{q\balpha} < \psi(2^k)\right\}}\\[1ex]
                                            & = \sum_{k\in\NN}\psi(2^k)^{m}Q(2^{k - 1}, 2^k) \\[1ex]
                                            & = \sum_{k\in\NN}\sum_{j\geq k}\left(\psi(2^j)^{m} - \psi(2^{j + 1})^{m}\right)Q(2^{k - 1}, 2^k)\\[1ex]
                                            & = \sum_{j\in\NN}\left(\psi(2^j)^{m} - \psi(2^{j + 1})^{m}\right)\sum_{k = 1}^j Q(2^{k - 1}, 2^k) \\[1ex]
                                            &\geq \sum_{j\in\NN}\left(\psi(2^j)^{m} - \psi(2^{j + 1})^{m}\right) Q(2^j) \\[1ex]
                                            &\geq \sum_{j\in\NN}\left(\psi(2^j)^{m} - \psi(2^{j + 1})^{m}\right)
                                              [2^j\psi(2^j)^\ell - 1] \quad (\text{by Lemma \ref{lem:count}})\\[1ex]
                                            &=  -\psi(2)^m + \sum_{j\in\NN}\left(\psi(2^j)^{m} - \psi(2^{j + 1})^{m}\right)
                                              2^j\psi(2^j)^\ell.
  \end{align*}
  Let $(j_d)_{d = 1}^\infty$ be the sequence indexing the set
  $\{j\in\NN : k_{2^{j}}\neq k_{2^{j + 1}}\}$ in increasing
  order. Then we have
  \begin{equation*}
  	\psi(2^{j_d})^m - \psi(2^{j_d + 1})^m \gg \psi(2^{j_d})^m,
  \end{equation*}
  and hence,
  \begin{align*}
    \sum_{j\in\NN}\left(\psi(2^j)^{m} - \psi(2^{j + 1})^{m}\right)
    2^j\psi(2^j)^\ell
    &\gg \sum_{d\in\NN}2^{j_d} \psi(2^{j_d})^{m + \ell} \\
    &\gg \sum_{d\in\NN}\left(\sum_{k = j_{d - 1} + 1}^{j_d}2^k\right)\psi(2^{j_d})^n \\[1ex]
    & = \sum_{k\in\NN}2^k\psi(2^k)^{n},
  \end{align*}
  which diverges by Cauchy's condensation test (see Lemma \ref{sum-lemma}).
\end{proof}

\begin{proof}[Proof of Theorem~\ref{thm:subspaces}]
  Suppose that $m\geq 2$. Then, by Lemma \ref{lem:series}, we can apply
  Gallagher's extension of Khintchine's theorem \cite{Gallagherkt} to
  the function
  \begin{equation}
    \label{psix}
    \psi_{\balpha}(q) = \begin{cases}
      \psi(q) & \text{if }\ \norm{q\balpha} < \psi(q), \\
      0 & \text{otherwise},
    \end{cases}
  \end{equation}
  and get that $\{\balpha\}\times\RR^m$ is of Khintchine type for
  divergence. But $\balpha\in\I^\ell$ was chosen arbitrarily, and applying
  permutation matrices does not affect whether a manifold is of
  Khintchine type for divergence. This completes the proof.
\end{proof}

\section[Proof of Theorem~\ref{thm:lines}(i): Base points of high Diophantine type]{Proof of Theorem~\ref{thm:lines}(i):\\ Base points of high Diophantine type}\label{Sec:KTP}

The proof of Theorem \ref{thm:lines}(i) is based on the following
standard fact, which can be found for example in~\cite[Theorem
V.IV]{Cassels}:

\theoremstyle{plain} \newtheorem*{ktp}{Khintchine's transference
  principle}
\begin{ktp}
  Let $\balpha\in\I^d$ and define the numbers
  \begin{equation*}
    \omega_D=\omega_D(\balpha) = \sup\left\{\omega\in\mathbb{R}^+: \norm{\inner{\bq}{\balpha}} \leq \abs{\bq}^{-(n + \omega)} \textrm{ for i.m. } \bq\in\ZZ^n\backslash\{\boldsymbol{0}\} \right\}
  \end{equation*}
  and
  \begin{equation*}
    \omega_S=\omega_S(\balpha) = \sup\left\{\omega\in\mathbb{R}^+: \norm{q\balpha} \leq q^{-(1 + \omega)/n} \textrm{ for i.m. } q\in\NN \right\}.
  \end{equation*}
  Then
  \begin{equation*}
    \frac{\omega_D}{n^2 + (n - 1)\omega_D} \leq \omega_S \leq \omega_D,
  \end{equation*}
  where the cases $\omega_D = \infty$ and $\omega_S = \infty$ should
  be interpreted in the obvious way.
\end{ktp}

Note that $\omega_D$ is related to the dual Diophantine type $\tau_D$
defined in \eqref{taudef} via the formula
$\tau_D(\bx) = \omega_D(\bx) + n$.

\begin{proof}
[Proof of Theorem~\ref{thm:lines}(i)]
We fix $\balpha = (\alpha_1,\dots, \alpha_{n - 1})\in\I^{n - 1}$ such that
$\tau_D(\balpha) > n$, and we consider a point
$(\balpha,\beta) \in \{\balpha\}\times\RR$. It is clear from \eqref{taudef} that
$\tau_D(\balpha,\beta) \geq \tau_D(\balpha)$, so $\tau_D(\balpha,\beta) > n$ and thus
$\omega_D(\balpha,\beta) > 0$. Then, by Khintchine's transference principle,
$\omega_S(\balpha,\beta) > 0$, {\it i.e.} $(\balpha,\beta) \in \mathrm{VWA}_n$.
\end{proof}

\section[Proof of Theorem~\ref{thm:lines}(ii): Base points of low Diophantine type]{Proof of Theorem~\ref{thm:lines}(ii):\\ Base points of low Diophantine type}\label{sec:ubiquity}

We start by stating a result of Cassels' \cite{Cassels-01law}, which will be used in the proof.

\begin{theorem}[Cassels]\label{Thm:Cassels01}
	Let $(\phi(i))_{i\in\NN}$ be any sequence of non-negative numbers and let $(q_i)_{i\in\NN}$ be any sequence of integers. Then $\parallel q_i\alpha \parallel < \phi(i)$ has infinitely many solutions either for almost all or for almost no $\alpha\in\RR$.
\end{theorem}

\begin{remark*}
	Theorem \ref{Thm:Cassels01} is also known as Cassels' ``0-1 law''. Gallagher's ``0-1 law'', which was referred to in Chapter \ref{Ch:Introduction}, is an extension of Theorem \ref{Thm:Cassels01} to the coprime setting of the Duffin--Schaeffer Conjecture, see Conjecture \ref{Con:DS}.
\end{remark*}

We will also need to make use of a property of lattices. Let $\Lambda=\Lambda(A)=A\ZZ^m\subset\RR^m$ be a full-rank lattice generated by $A\in\RR^{m\times m}$ satisfying $\det A\neq 0$. For $j\in\{1,\dots,m\},$ we define the \textit{$j$-th successive minimum} of $\Lambda$ as
\begin{equation*}
	\mu_j=\mu_j(\Lambda):=\inf\left\{r>0:\Lambda\cap \bar{B}(\0,r)\text{ contains }j\text{ linearly independent vectors}\right\},
\end{equation*}
where $\bar{B}$ denotes a closed ball. Clearly, $0<\mu_1\leq\mu_2\leq\dots\leq\mu_m<\infty$. Furthermore, let the \textit{Dirichlet fundamental domain} of $\Lambda$ centred at $\0$ be defined as
  \[
    \cD = \{\br\in\RR^{m} : \mathrm{dist}(\br,\Lambda) =
    \mathrm{dist}(\br,\0) = |\br|\}.
  \]
Then the following holds.

\begin{lemma}\label{Lma:Lattice}
	Suppose that $\cD \not\subseteq B_{m}(\0,R)$. Then the last
  	successive minimum of $\Lambda$ satisfies $\mu_m\geq R/m$.
\end{lemma}

\begin{proof}
	Assume that $\mu_m< R/m$. As $\cD \not\subseteq B_{m}(\0,R)$, there exists $\bx\in\cD$ with $|\bx|>R$. Since $\bx\in\cD$, it follows that $\bar{B}(\bx,R)\cap\Lambda=\emptyset$. There are $m$ linearly independent vectors $\bv_1,\dots,\bv_m\in\Lambda$ satisfying $|\bv_1|=\mu_1,\dots,|\bv_m|=\mu_m$. These vectors span $\RR^m$ and so we can write
	\begin{equation*}
		\bx=s_1\bv_1+\dots+s_m\bv_m,\quad \text{ with } s_j\in\RR,\quad (1\leq j\leq m).
	\end{equation*}
	Let
	\begin{equation*}
		\bz=\lfloor s_1\rfloor \bv_1+\dots+\lfloor s_m\rfloor \bv_m\in\Lambda.
	\end{equation*}
	It follows that
	\begin{align*}
		|\bx-\bz|=\left|\sum\limits_{j=1}^{m}(s_j-\lfloor s_j\rfloor)\bv_j\right|\leq \sum\limits_{j=1}^{m}|\bv_j|\leq R
	\end{align*}
	and so $\bz\in\bar{B}(\bx,R)\cap\Lambda$, which is a contradiction.
\end{proof}

\begin{remark}
	The lower bound $\mu_m\geq R/m$ is not optimal, but for our purposes we only need the fact that $\mu_m\geq cR$ for some constant $c>0$ which does not depend on the lattice $\Lambda$.
\end{remark}

Now to the proof of Theorem~\ref{thm:lines}(ii).
Let $\psi:\RR\to\RR^+$ be non-increasing and such that
$\sum_{q\in\NN}\psi(q)^n$ diverges. Our goal here is to use the ideas
of ubiquity theory introduced in Chapter \ref{ubiquity} to show that almost every point on
$\{\balpha\}\times\RR\subseteq\RR^n$ is $\psi$-approximable, where
$\balpha\in\RR^{n - 1}$ has been fixed with dual Diophantine type strictly
less than $n$. The ubiquity approach begins with the fact that for any
$N\in\NN$ such that
\begin{equation}
  \label{Nreq}
  N^{-1/(n - 1)} < \psi(N) < 1,
  \vspace{-2ex}
\end{equation}
we have
\begin{equation}
  \label{eq:mink}
  [0,1] \subseteq \bigcup_{\substack{q\leq N \\ \norm{q \balpha} < \psi(N)}}\bigcup_{p = 0}^q B\left(\frac{p}{q},\frac{2}{q N\psi(N)^{n - 1}}\right),
  \vspace{2ex}
\end{equation}
which is a simple consequence of Minkowski's theorem. The basic aim is
to show that a significant proportion of the measure of the above
double-union set is represented by integers $q$ that are closer to $N$ than to
$0$. Specifically, we must show that for some $k\geq 2$, the following
three conditions hold:
\begin{enumerate}
\item[\textbf{(U)}] In accordance with the theory presented in Chapter \ref{ubiquity}, we define the following objects:
  \begin{align*}
    J &= \{(p,q)\in\ZZ\times\NN : \norm{q\balpha} < \psi(q)\},&
                                                   R_{(p,q)} &= \left\{p/q\right\} \;\;\; ((p,q) \in J),\\
    \cR &= \{R_{(p,q)} : (p,q)\in J\},&
                                    \beta_{(p,q)} &= q \;\;\; ((p,q)\in J),\\
    l_j &= k^{j - 1} \;\;\; (j\in\NN),&
                                        u_j &= k^j \;\;\; (j\in\NN).
  \end{align*}
  Furthermore, we define the function $\rho:\NN\rightarrow\RR^{+}$ by \vspace{-1ex}
  \begin{equation*}
    \rho(q) = \frac{c}{q^2 \psi(q)^{n - 1}},
  \end{equation*}
  where $c > 0$ will be chosen later. Then the pair $(\cR,\beta)$ forms a global
  ubiquitous system with respect to the triple $(\rho,l,u)$. This means that there is some
  $\kappa > 0$ such that
  \begin{align*}
    \lambda\left([0,1]\cap\bigcup_{\substack{k^{j - 1} < q \leq k^j \\ \norm{q\balpha} < \psi(k^j)}}\bigcup_{p = 0}^q
      B\left(\frac{p}{q},\frac{c}{k^{2j}\psi(k^j)^{n - 1}}\right)\right) \geq \kappa
  \end{align*}
  for all $j$ sufficiently large. \item[\textbf{(R)}] The function
  $\varphi(q)= \psi(q)/q$ is $u$-regular, meaning that there
  is some constant $c < 1$ such that
  $\varphi(k^{j + 1})\leq c \varphi(k^j)$ for all $j$ sufficiently
  large.
\item[\textbf{(D)}] The sum
  $\sum_{j\in\NN}\frac{\varphi(k^j)}{\rho(k^j)}$ diverges.
\end{enumerate}
Then Corollary \ref{cor2} will imply that the set of
$\psi_{\balpha}$-approximable numbers (see \eqref{psix}) in $\RR$ has
positive measure, and Theorem \ref{Thm:Cassels01} will
imply that it has full measure. Since the set of
$\psi_{\balpha}$-approximable numbers is just the set of $y\in\RR$ for which
$(\balpha,y)$ is $\psi$-approximable, this will show that the set of
$\psi$-approximable points on the line
$\{\balpha\}\times\RR\subseteq\RR^n$ has full one-dimensional Lebesgue
measure. The following lemma establishes \textbf{(R)} and~\textbf{(D)}.

\begin{lemma}
  If $\psi:\RR\to\RR^+$ is non-increasing, then~\textup{\textbf{(R)}}
  holds. Furthermore, if $\sum_{q\in\NN}\psi(q)^n$ diverges,
  then~\textup{\textbf{(D)}} holds.
\end{lemma}

\begin{proof}
  In the first place, we have
  \begin{equation*}
    \frac{\varphi(k^{j + 1})}{\varphi(k^j)} = \frac{\psi(k^{j + 1})}{k\psi(k^j)}
    \leq \frac{1}{k},
  \end{equation*}
  which proves~\textbf{(R)}. For~\textbf{(D)},
  \begin{equation*}
    \sum_{j\in\NN}\frac{\varphi(k^j)}{\rho(k^j)} = \sum_{j\in\NN}k^j\psi(k^j)^n,
  \end{equation*}
  which diverges by Cauchy's condensation test.
\end{proof}

The challenge then is to establish~\textbf{(U)}.

\begin{lemma}\label{ubiqlemma}
  Let $\psi:\RR\to\RR^+$ be non-increasing such that \eqref{Nreq} holds
  for all sufficiently large $N$. Assume that for all $k\geq 2$
  there exists $j_k\geq 1$ such that, for all $j\geq j_k$,
  \begin{equation*} \Abs{\left\{0 < q\leq k^{j - 1} : \norm{q\balpha} <
        \psi(k^j)\right\}} \ll k^{j - 1}\psi(k^j)^{d - 1},
        \vspace{1ex}
  \end{equation*}
  where the implied constant in $\ll$ is
  assumed to be independent of $k$. Then~\textup{\textbf{(U)}} holds
  for some $k\geq 2$.
\end{lemma}

\begin{proof}
  For all $k\geq 2$ and $j\geq j_k$, we have
  \begin{equation*}
    \lambda\left([0,1]\cap\bigcup_{\substack{q\leq k^{j - 1} \\ \norm{q\balpha} < \psi(k^j)}}\bigcup_{p = 0}^q B\left(\frac{p}{q},\frac{2}{q k^j \psi(k^j)^{d - 1}}\right)\right)
    \leq\sum_{\substack{q\leq k^{j - 1} \\ \norm{q\balpha} < \psi(k^j)}}\frac{4}{k^j\psi(k^j)^{d - 1}}
    \ll\frac{1}{k}\cdot
    \vspace{1ex}
  \end{equation*}
  After choosing $k$ to be larger than the implied constant in the
  ``$\ll$'' comparison, we see that the left hand side is
  $\leq 1 - \kappa < 1$ for some $\kappa > 0$.

  Combining this with \eqref{eq:mink}, we see that for all $j\geq j_k$
  large enough so that \eqref{Nreq} holds for $N = k^j$, we have
  \begin{equation*}
    \lambda\left([0,1]\cap\bigcup_{\substack{k^{j - 1} < q\leq k^j \\ \norm{q\balpha} < \psi(k^j)}}\bigcup_{p = 0}^q
      B\left(\frac{p}{q},\frac{2}{q k^j \psi(k^j)^{d - 1}}\right)\right)
    \geq \kappa > 0,
    \vspace{1ex}
  \end{equation*}
  and this implies \textbf{(U)} with $c = 2k$.
\end{proof}

Thus the goal is to show that the conditions of Lemma \ref{ubiqlemma} are satisfied.
The one-dimensional case of the following lemma was originally proven
by Beresnevich, Haynes, and Velani using a continued fraction
argument \cite{nalpha}.

\begin{lemma}
  \label{nalphalemma}
  Fix $\balpha\in\RR^\ell$ and $\tau > \tau_D(\balpha)$. Then for all $N$
  sufficiently large and for all $\delta \geq N^{-1/\tau}$, we have\vspace{-2ex}
  \begin{equation}
    \label{BHVformula}
    \Abs{\{q\in\mathbb{N}:\norm{q\balpha} < \delta,q< N\}}\leq 4^{\ell + 1} N \delta^\ell.
  \end{equation}
\end{lemma}
\begin{proof}
  Consider the lattice $\Lambda = g_t u_{\balpha} \ZZ^{\ell + 1}$, where
  \begin{align*}
    g_t & = \left[\begin{array}{ll}
                    e^{t/\ell}I_\ell &\\
                                     & e^{-t}
                  \end{array}\right],\\[2ex]
    u_{\balpha} & = \left[\begin{array}{ll}
                      I_\ell & - \balpha\\
                             &\ \ \ 1
                    \end{array}\right],
  \end{align*}
  and where $t$ is chosen so that $R : = e^{t/\ell} \delta = e^{-t} N$,
  i.e.
  \begin{equation*}
  	t=\frac{\log(N/\delta)}{1+1/\ell}.
  \end{equation*}
  Let $\br=(p_1,\dots,p_{\ell},q)\in\ZZ^{\ell+1}$. Then
  \begin{equation*}
  	g_t u_{\balpha}\br=(e^{t/\ell}(p_1+q\alpha_1),\dots,e^{t/\ell}(p_{\ell}+q\alpha_{\ell}),e^{-t}q),
  \end{equation*}
  and so \eqref{BHVformula} can
  be rewritten as
  \[
    \Abs{\{\br\in \Lambda : |\br| < R\}} \leq (4R)^{\ell + 1}.
  \]
  Let $\cD$ be the Dirichlet fundamental domain for $\Lambda$
  centred at $\0$, i.e.
  \[
    \cD = \{\br\in\RR^{\ell + 1} : \mathrm{dist}(\br,\Lambda) =
    \mathrm{dist}(\br,\0) = |\br|\}.
  \]
  Since $\Lambda$ is unimodular, $\cD$ is of volume 1, so
  \begin{align*}
    \Abs{\{\br\in \Lambda : |\br| < R\}} &=
    \lambda_{\ell+1}\left(\bigcup_{\substack{\br\in\Lambda \\ |\br|<
          R}} (\br + \cD)\right)\\[1ex] &\leq^* \lambda_{\ell+1}(B_{\ell + 1}(\0,2R)) =
    (4R)^{\ell + 1},
  \end{align*}
  where the starred inequality is true as long as
  $\cD \subseteq B_{\ell + 1}(\0,R)$. So we need to show that
  $\cD\subseteq B_{\ell + 1}(\0,R)$ assuming that $N$ is large enough.

  Suppose that $\cD \not\subseteq B_{\ell + 1}(\0,R)$. Then, by Lemma \ref{Lma:Lattice}, the last
  successive minimum of $\Lambda$ is $\gg R$, so by \cite[Theorem
  VIII.5.VI]{Cassels-geometry}, some point $\bs$ in the dual lattice
  \begin{equation*}
  	\Lambda^* = \{\bs\in \RR^{\ell + 1} : \br\cdot\bs\in\ZZ \text{ for all }\br\in\Lambda\}
  \end{equation*}
  satisfies $0 < |\bs|\ll R^{-1}$. Given a lattice $\Lambda(A)$, its dual lattice is generated by the inverse transpose of $A$ and so we can write
  $\bs = g_t' u_{\balpha}' (\bq,p)$ for some $p\in\ZZ$, $\bq\in\ZZ^\ell$,
  where $g_t'$ and $u_{\balpha}'$ denote the inverse transposes of $g_t$ and
  $u_{\balpha}$, respectively. Then the inequality $|\bs|\ll R^{-1}$
  becomes
  \begin{align*}
    \begin{split}
      e^{-t/\ell} |\bq| &\ll R^{-1}\\
      e^t |\inner\bq\balpha + p| &\ll R^{-1}
    \end{split}
                               \hspace{.6in}\text{ i.e. }
                               \begin{split}
                                 |\bq| &\ll \delta^{-1}\\
                                 |\inner\bq\balpha + p| &\ll N^{-1}.
                               \end{split}
  \end{align*}
  Since $\delta \geq N^{-1/\tau},$ we get
  \begin{equation}
    \label{final}
    |\inner\bq\balpha + p| \ll \delta^\tau \ll |\bq|^{-\tau}.
  \end{equation}
  Because $\tau > \tau_D(\balpha)$, there are only finitely many pairs
  $(p,\bq)$ satisfying \eqref{final}. Hence, for all sufficiently large
  $N$, we have $\cD \subseteq B_{\ell + 1}(\0,R)$ and thus
  \eqref{BHVformula} holds.
\end{proof}

From this we can deduce the following consequence.

\begin{corollary}\label{cor:nalpha}
  Let $\balpha\in\I^{n - 1}$ be of dual Diophantine type
  $\tau_D(\balpha) < n$ and suppose that, for any $\eps > 0$, we have
  $\psi(q) \geq q^{-1/n - \eps}$ for all $q$ sufficiently large. Then,
  for any $k\geq 2$ and $\ell\in\ZZ$, we have
  \begin{equation*}
    \Abs{\left\{ 0 < q \leq k^{j + \ell} : \norm{q\balpha} < \psi(k^j) \right\}}\ll k^{j + \ell}\psi(k^j)^{n - 1}
  \end{equation*}
  for $j$ large enough.
\end{corollary}

\begin{proof}
  We show that for large enough $j$ we are in a situation where we can
  apply Lemma~\ref{nalphalemma} with $N = k^{j + \ell}$ and
  $\delta = \psi(k^j)$. Since $\tau_D < n$ we can choose
  $\tau\in (\tau_D,n)$ and then, for all large enough $j$,
  \begin{equation*}
    N^{-1/\tau} = k^{-(j + \ell)/\tau} < \psi(k^j);
  \end{equation*}
  hence Lemma~\ref{nalphalemma} applies.
\end{proof}

Armed with Corollary \ref{cor:nalpha}, we are now ready to finish the proof.

\begin{proof}
[Proof of Theorem~\ref{thm:lines}(ii)]
Let $\balpha\in\I^{n - 1}$ be a point whose dual Diophantine type is
strictly less than $n$, and let $\psi:\NN\to\RR^+$ be a non-increasing
function such that $\sum_{q\in\NN}\psi(q)^n$ diverges. Furthermore,
assume that, for every $\eps > 0$, the inequality
\begin{equation}\label{Eqn:psibounds}
	1 > \psi(q) \geq q^{-1/n - \eps}
\end{equation}
holds for all sufficiently large
$q$. Then, by Corollary~\ref{cor:nalpha}, we satisfy all the parts of
Lemma~\ref{ubiqlemma}, so there exists $k\geq 2$ such that
\textbf{(U)} holds. Thus, by the argument given earlier, we can use
Corollary \ref{cor2} to conclude that almost every point on the
line $\{\balpha\}\times\RR\subseteq\RR^n$ is $\psi$-approximable.

We now show that assumption \eqref{Eqn:psibounds} can be made
without loss of generality. If $\psi(q) \geq 1$ for all $q$, then all
points are $\psi$-approximable and the theorem is trivial. If
$\psi(q) < 1$ for some $q$, then, by monotonicity, $\psi(q) < 1$ for all
$q$ sufficiently large. So we just need to show that the assumption
$\psi(q) \geq q^{-1/n - \eps}$ can be made without loss of
generality. Let 
\begin{equation*}
	\phi(q) = (q(\log q)^2)^{-1/n}
\end{equation*}
and define the function
\begin{equation*}
	\overline\psi(q) = \max\{\psi(q),\phi(q)\}.
\end{equation*}
Then $\overline\psi$ satisfies our assumptions and, therefore, almost every
point on $\{\balpha\}\times\RR$ is $\overline\psi$-approximable.
Corollary \ref{cor:nalpha} implies that
\begin{align*}
  \sum_{\norm{q\balpha} < \phi(q)}\phi(q) &\leq \sum_{j\in\NN}\phi(2^j)\Abs{\left\{0 < q\leq 2^{j + 1} : \norm{q\balpha} < \phi(2^j)\right\}}\\[1ex]
  &\ll \sum_{j\in\NN}2^{j + 1} \phi(2^j)^n,
\end{align*}
which converges because $\sum_{q\in\NN}\phi(q)^n$ does. Hence,
almost every point on $\{\balpha\}\times\RR$ is not
$\phi$-approximable. But every $\overline\psi$-approximable point
which is not $\phi$-approximable is $\psi$-approximable. Therefore,
the set of $\psi$-approximable points on the line
$\{\balpha\}\times\RR\subseteq\RR^n$ is of full measure, and the theorem
is proved.
\end{proof}

\section{Proof of Theorem \ref{HDfibres}}\label{Sec:HDfibres}
As mentioned previously, the first case follows directly from Dirichlet's Theorem, so the two latter cases are left to prove. Let $\balpha\in W_{\ell}(\tau)$ and
\begin{equation*}
	\cQ(\balpha,\tau)=\left\{q\in\NN:\parallel q\balpha\parallel<q^{-\tau}\right\}.
\end{equation*}
This is an infinite set, so it can be written as an increasing sequence $(q_i)_{i\in\NN}$.	The collection $\cQ=\bigcup_{i\in\NN}Q_i$ of sets
\begin{equation*}
    	Q_i:=\left\{\frac{\bp}{q}\in\mathbb{Z}^{m}\times\mathbb{N} : 0< q\leq q_i, \parallel q\balpha\parallel<q_i^{-\tau}\right\}\subseteq [0,1]^{m},\quad i\in\NN,
\end{equation*}
forms a locally $\lambda_{m}$-ubiquitous system with respect to the infinite increasing sequence $(q_i)_{i\in\NN}$, the weight function $\beta(\bp/q)=q$ and the function $\rho(q)=q^{-1}$. This follows directly from the fact that the intervals of the form
\begin{equation*}
	\left(\frac{k}{q_i},\frac{k+1}{q_i}\right),\quad k\in\{0,1,\dots,k-1\},
\end{equation*}
or their multi-dimensional analogues, respectively, cover $\I^{m}$. We do not know the growth rate of the sequence $(q_i)_{i\in\mathbb{N}}$ for an arbitrary $\balpha$. However, as long as we can guarantee 
\begin{equation*}
  	G=\limsup\limits_{i\rightarrow\infty}g(u_i)>0,\quad \text{ where }\quad g(r)=\varphi(r)^{s}\rho(r)^{-\delta},
  	\vspace{2ex}
\end{equation*}  	
we still get full $\mathcal{H}^s$-measure for $\F^{\balpha}_n\cap W_n(\psi)$ by Corollary \ref{cor4}. In our case, we have $\varphi(r)=\psi(r)/r$ and $\delta=m$. Thus, we get
\begin{equation*}
 	g(q_i)=\left(\frac{\psi(q_i)}{q_i}\right)^s\rho(q_i)^{-m}=q_i^{(-\tau-1)s+m},
\end{equation*}
and if $s\leq\frac{m}{\tau+1}=s_n^{\ell}(\tau)$, then $G\geq 1$. Hence, we have shown that
\begin{equation*}
	\cH^s(\F^{\balpha}_n\cap W_n(\psi))=\infty
	\vspace{-2ex}
\end{equation*}
for $s\leq s_n^{\ell}(\tau)$ and thus
\begin{equation*}
	s_n^{\alpha}(\tau)\geq s_n^{\ell}(\tau)= \frac{m}{\tau+1}
	\vspace{1ex}
\end{equation*}
for $\balpha\in W_{\ell}(\tau)$, which proves the third case of Theorem \ref{HDfibres}. 

\begin{remark*}
	This dimension result was already proved by the author in \cite{master}. Corollary \ref{cor5} was used for the conclusion instead of Corollary \ref{cor4} and so the statement regarding Hausdorff $s_n^{\ell}(\tau)$-measure was missing. The remaining case was mentioned as a conjecture in \cite{master}, but no progress towards a proof had been made.
\end{remark*}

It is worth noting that the above argument does not depend on the choice of $\tau$. However, this fact is not sufficient to prove the second case as 
\begin{equation*}
	\frac{n+1}{\tau+1}-\ell>\frac{m}{\tau+1}\quad \text{ for }\quad \tau\in\left(\frac{1}{n},\frac{1}{\ell}\right).
\end{equation*}
For this case we will need to use the Mass Transference Principle (see Theorem \ref{MTP}). A similar argument can be found in \cite{Lee}. By Minkowski's Theorem (see Theorem~\ref{Minkowski}), for any $\bbeta\in\I^m$ there are infinitely many numbers $q\in\NN$ simultaneously satisfying \vspace{-2ex}
\begin{equation*}
	\parallel q\alpha_i \parallel < q^{-\tau},\quad (1\leq i \leq \ell)
	\vspace{-2ex}
\end{equation*}
and
\begin{equation*}
	\parallel q\beta_j \parallel <q^{-\left(\frac{1-\ell\tau}{m}\right)},\quad (1\leq j \leq m). 
\end{equation*}
In other words, for any $\bbeta\in\I^m$ there are infinitely many numbers $q$ in the intersection
\begin{equation*}
	\cQ(\balpha,\tau)\cap\cQ\left(\bbeta,\frac{1-\ell\tau}{m}\right).
\end{equation*}
Switching to the language of Theorem \ref{MTP}, this tells us that for any ball $B\subset \I^m$, we get
\begin{equation*}
	\cH^{m}\left(B\cap\limsup\limits_{k\rightarrow\infty}B_k^{\left(\frac{m+1-\ell\tau}{\tau+1}\right)}\right)=\cH^{m}(B)
	\vspace{2ex}
\end{equation*}
where $B_k$ runs over all balls of radius $q^{-(\tau+1)}$ centered at points $\frac{\bp}{q}\in\ZZ^m\times \cQ(\balpha,\tau)$. Applying the Mass Transference Principle shows that
\begin{equation*}
	\cH^s\left(B\cap\limsup\limits_{k\rightarrow\infty}B_k\right)=\cH^s(B)=\infty,
	\vspace{-2ex}
\end{equation*}
where
\begin{equation*}
	s=\frac{m+1-\ell\tau}{\tau+1}=\frac{n-\ell+1-\ell\tau}{\tau+1}=\frac{n+1-\ell(\tau+1)}{\tau+1}=\frac{n+1}{\tau+1}-\ell.
\end{equation*}
Thus, the set of $\bbeta\in\I^m$, for which there are infinitely many numbers $q\in\cQ(\balpha,\tau)$ with $\parallel q\bbeta\parallel < q^{-\tau}$, has full $\cH^s$-measure. Equivalently, the set of $(\balpha,\bbeta)\in \F_n^{\balpha}$ which are $\tau$-approximable has full $\cH^s$-measure for $s=\frac{n+1}{\tau+1}-\ell$, which finishes the proof of Theorem \ref{HDfibres}.

\subsection{Proof of Corollary \ref{Cor:HD}}\label{sec:HDrmk}

Corollary \ref{Cor:HD} is a consequence of the Jarn\'ik-Besicovitch Theorem applied to both $W_n(\tau)$ and $W_m(\tau)$ and the following Theorem (7.11 in \cite{Falconer}), which provides a measure formula for fibres above a base set.
	
\begin{theorem}\label{slicing}
	Let $F$ be a subset of $\I^n=\I^{\ell}\times\I^m$ and let $E$ be any subset of $\I^{\ell}$. Let $s,t\geq 0$ and suppose there is a constant $c$ such that 
	\begin{equation*}
		\mathcal{H}^t(F\cap \F^n_{\balpha})\geq c
		\vspace{-2ex}
	\end{equation*}
	for all $\balpha\in E$. Then
	\begin{equation*}
		\mathcal{H}^{s+t}(F)\geq bc\mathcal{H}^s(E),
	\end{equation*}
	where $b>0$ only depends on $s$ and $t$. 
\end{theorem}

Reformulated as a statement about Hausdorff dimension, Theorem \ref{slicing} can be interpreted as follows.

\newpage
\begin{corollary}\label{slicingcor}
	Let $F$ be a subset of $\I^n=\I^{\ell}\times\I^m$ and let $E$ be any subset of $\I^{\ell}$. Suppose that
	\begin{equation*}
		\dim(F\cap \F^n_{\balpha})\geq s
		\vspace{-2ex}
	\end{equation*}		
	for all $\balpha\in E$. Then
	\begin{equation*}
		\dim(F)\geq \dim(E)+s.
		\vspace{1ex}
	\end{equation*}
\end{corollary}

Equipped with Corollary \ref{slicingcor} we are ready to prove Corollary \ref{Cor:HD}.

\begin{proof}[Proof of Corollary \ref{Cor:HD}]
	Let $\frac{1}{n}<\tau\leq\frac{1}{\ell}$ and $\eps=\frac{1}{k}$ for $k\in\mathbb{Z}^+$. Then the set $\I^{\ell}_k(\tau)$ of $\balpha\in \I^{\ell}$ such that \vspace{1ex}
	\begin{equation*}
		s^{\balpha}_n(\tau)\geq \frac{n+1}{\tau+1}-\ell+\eps
		\vspace{2ex}
	\end{equation*}
	is a set of dimension less than $\ell$. Indeed, if it had dimension at least $\ell$, then the union over all the sets $\F_n^{\balpha}\cap W_n(\tau)$ would be a set of dimension strictly greater than \vspace{1ex}
	\begin{equation*}
		\ell+\frac{n+1}{\tau+1}-\ell+\eps=\frac{n+1}{\tau+1}+\eps>\dim W_n(\tau),
	\end{equation*}	 by the slicing formula for product-like sets given in Corollary \ref{slicingcor}. However, this is not possible due to the monotonicity property of Hausdorff dimension. Hence, $\dim \I^{\ell}_k(\tau)<\ell$, and, in particular, 
	\begin{equation*}
		\lambda_{\ell}(\I^{\ell}_k(\tau))=0.
	\end{equation*}
	Now, the set of $\balpha\in \I^{\ell}$ satisfying \eqref{Eqn:HDslice} is the countable union over all sets $\I^{\ell}_k(\tau)$ with $k\in\mathbb{Z}^+$ and hence also a zero set with respect to $\lambda_{\ell}$. The second case works completely analogously.
\end{proof}

\chapter{Dirichlet improvability and singular vectors}\label{dirichlet}
\chaptermark{Dirichlet improvability}

We recall that the notion of a weight vector was introduced in Section \ref{sec:weighted} and that given any weight vector $\bi\in\RR^n$ and $Q\in\NN$, Theorem \ref{wDir} states that the system of inequalities given by
\begin{equation*}
	\norm{q\alpha_j}<Q^{-i_j},\quad j\in \{1,\dots,n\},
\end{equation*} 
has a non-zero integer solution $q\leq Q$. We will say that $\balpha\in\RR^n$ is \emph{$\bi$-Dirichlet improvable} or $\balpha\in \cD_n(\bi)$ if there exists a positive constant $c<1$ such that the system of inequalities
\begin{equation}\label{DI}
		\norm{q\alpha_j}<cQ^{-i_j},\quad j\in \{1,\dots,n\},
\end{equation} 
has a non-zero integer solution $q\leq Q$ for all $Q$ large enough. If this is true for $c$ arbitrarily small, then we call $\balpha$ \emph{$\bi$-singular} or we say $\balpha\in\cS_n(\bi)$. We omit the $\bi$ in both notations in the case where $\bi=(1/n,\dots,1/n)$, i.e. when we are dealing with an improved version of the classical non-weighted theorem by Dirichlet. In dimension one, we simply denote the sets in question by $\cD$ and $\cS$, respectively. Of course, in this case, the only choice of weight vector is $i=1$.

\section{Background and our results}

Khintchine introduced the notion of singular vectors in the 1920s \cite{KhintchineSing}. He showed in the one-dimensional case that the rationals are the only singular numbers. This is a direct consequence of the following statement.

\begin{lemma}\label{Lma:Wald}
	Let $\alpha$ be a real number. Assume there exists $Q_0=Q_0(\alpha)$ such that for each integer $Q\geq Q_0$ there exist $p=p(Q)\in\ZZ$ and $q=q(Q)\in\NN$ satisfying $q\leq Q$ and \vspace{-1ex}
	\begin{equation}\label{Eqn:Wald}
		|q\alpha-p|<\frac{1}{3Q}.
		\vspace{2ex}
	\end{equation}
	Then, $\alpha$ is rational and $p/q=\alpha$ for each $Q\geq Q_0$.
\end{lemma}

\begin{proof}
	We use a slightly modified version of an argument from \cite{Waldschmidt}. Assume that $Q\geq Q_0$ and that $p$ and $q$ are the integers satisfying \eqref{Eqn:Wald}. Moreover, denote by $p'$ and $q'$ the integers such that\vspace{-1ex}
	\begin{equation*}
		|q'\alpha-p'|<\frac{1}{3(Q+1)}\quad\text{ and }\quad 1\leq q'\leq Q+1.
	\end{equation*}
	We want to show that $p/q=p'/q'$. The integer $qp'-q'p$ satisfies
	\begin{align*}
		|qp'-q'p|&=|q(p'-q'\alpha)+q'(q\alpha-p)|\\[1ex]
		&\leq |q(p'-q'\alpha)|+|q'(q\alpha-p)|\\[1ex]
		&<\frac{Q}{3(Q+1)}+\frac{Q+1}{3Q}\\[1ex]
		&<\frac{1}{3}+\frac{2}{3}=1,
	\end{align*}
	hence it vanishes. This implies that $qp'=q'p$ and thus the rational number $p/q=p'/q'$ does not depend on $Q\geq Q_0$. On the other hand,
	\begin{equation*}
		\lim\limits_{Q\rightarrow\infty}\frac{p(Q)}{q(Q)}=\alpha,
	\end{equation*}
	which shows that $\alpha=p/q$.
\end{proof}

\begin{remark*}
	The constant $1/3$ in \eqref{Eqn:Wald} is not optimal. In fact, Lemma \ref{Lma:Wald} is proved with the constant $1/2$ in both \cite{KhintchineSing} and \cite{Waldschmidt}. However, Khintchine's argument uses results about the convergents of continued fractions and Waldschmidt defines a slightly stronger form of Dirichlet improvability, requiring that the solution in \eqref{DI} satisfies the strict inequality $q<Q$.
\end{remark*}

On the other hand, it is easily seen that any rational $a/b$ is singular. Indeed, for any $c>0$ and for $Q\geq b$, the choice of $q=b$ trivially satisfies \eqref{DI}. 

Davenport and Schmidt were the first to introduce the set $\cD$ and proved that $\alpha\in\RR\setminus\QQ$ is Dirichlet improvable if and only if $\alpha$ is badly approximable \cite{Davenport}. Their argument relies on the theory of continued fractions. This completes the one-dimensional theory.

\begin{theorem}
	$\cS=\QQ$ and $\cD\setminus\cS=\bad$.
\end{theorem}

The situation in higher dimensions is more intricate. It is straightforward to show that points on any rational hyperplane are singular. Hence, ${\dim\cS_n\in[n-1,n]}$. Also, on utilising the Borel--Cantelli Lemma it can be verified that the set $\cS_n$ has zero $n$-dimensional Lebesgue measure \cite[Chapter V, \S 7]{Cassels}. However, by means of a geometric argument, Khintchine proved the existence of totally irrational vectors contained in $\cS_n$ for $n\geq 2$ \cite{KhintchineSing}. Furthermore, in a recent groundbreaking paper, Cheung showed that the Hausdorff dimension of $\cS_2$ is equal to $4/3$ \cite{Cheung} and later Cheung and Chevallier extended this result to arbitrary dimensions \cite{Cheung2}.
\begin{theorem}[Cheung--Chevallier]
	For $n\geq 2$,
	\begin{equation*}
	\dim\cS_n=\frac{n^2}{n+1}.
	\end{equation*}
\end{theorem}
Significantly, this shows that $\cS_n$ is much bigger than the set of rationally dependent vectors, which is only of dimension $n-1$. These articles make use of dynamical methods. Previously, and through a classical approach, partial results towards establishing $\dim\cS_n$ had been made by Baker \cite{Baker1}, \cite{Baker2} and Rynne \cite{Rynne}. Regarding Dirichlet improvable vectors, Davenport and Schmidt showed that $\lambda_n(\cD_n)=0$ \cite{Davenport2}. They also proved the following result \cite{Davenport}.

\begin{theorem}[Davenport--Schmidt]\label{Thm:DSoriginial}
	$\bad_n\subset\cD_n.$
\end{theorem}

\begin{remark*}
	As an immediate consequence of Theorem \ref{Thm:DSoriginial} we see that $\dim\cD_n=n$.
\end{remark*}

All of these results concern the standard set-up and until recently very little research had been done for the weighted case. However, in a very recent article \cite{Tamam}, Liao, Shi, Solan and Tamam have managed to extend Cheung's two-dimensional result to general weight vectors.
\begin{theorem}
	Let $\bi=(i_1,i_2)$ be a weight vector. Then
	\begin{equation*}
		\dim\cS_2(\bi)=2-\frac{1}{1+\max\{i_1,i_2\}}.
	\end{equation*}
\end{theorem}

Before stating our results, we recall the definition of the set $\bad_n(\bi)$ for arbitrary weight vectors $\bi$:\vspace{-2ex}
\begin{equation*}
	\balpha\in\bad_n(\bi)\ \Longleftrightarrow\ \liminf\limits_{q\rightarrow\infty}\max\limits_{1\leq j\leq n}q^{i_j}\norm{q\alpha_j}>0.
\end{equation*}
We start by extending the above theorem of Davenport and Schmidt to the weighted case:

\begin{theorem}\label{ThmDS}
$\bad_n(\bi)\subseteq\cD_n(\bi).$
\end{theorem}

\begin{remark*}
	Clearly, no badly approximable vector can be singular, so we know that $\bad_n(\bi)$ and $\cS_n(\bi)$ are disjoint. Interestingly, apart from the trivial case $n=1$, it is not currently known if the sets $\cD_n\setminus (\bad_n\cup\cS_n)$ or their weighted analogues are empty.
\end{remark*}

Our second main result will show that non-singular vectors are well-suited for twisted inhomogeneous approximation. The following non-weighted result has been proved by Shapira using dynamical methods \cite{Shapira}.

\begin{theorem}[Shapira]
	Let $\balpha\notin\cS_n$. Then for almost all $\bbeta \in \I^n$,\vspace{-1ex}
	\begin{equation*}
		\liminf\limits_{q\rightarrow\infty}q^{1/n}\max\limits_{1\leq j\leq n}\norm{q\alpha_j-\beta_j}=0.
	\end{equation*}
\end{theorem}

We extend Shapira's result to general weight vectors. Our approach is classical.

\begin{theorem}\label{Con1}
	Let $\balpha\notin\cS_n(\bi)$. Then, for almost all $\bbeta\in \I^n$,\vspace{-1ex}
	\begin{equation*}	
		\liminf\limits_{q\rightarrow\infty}\max\limits_{1\leq j\leq n}q^{i_j}\norm{q\alpha_j-\beta_j}=0.
		\vspace{1ex}
	\end{equation*}
\end{theorem}

\begin{remark*}
	Another way to state Theorem \ref{Con1} is by using the notation introduced in Section \ref{sec:weighted}. Given $\eps>0$, let $\psi_{\eps}(q)=\eps q^{-1}$. Then, for any $\balpha\notin\cS_n(\bi)$ and for any $\eps>0$, the set $W_n^{\balpha}(\bi,\psi_{\eps})$ has full Lebesgue measure $\lambda_n$. Hence,
	\begin{equation*}
		\I^n\setminus\cS_n(\bi)\subset\bigcap\limits_{\eps>0} W_n^{\times}(\bi,\psi_{\eps}).
	\end{equation*}
\end{remark*}

\section{Preliminary results}

In this section we introduce various auxiliary statements which will be used to prove Theorems \ref{ThmDS} and \ref{Con1}.

\subsection{Results needed to prove Theorem \ref{ThmDS}.}

Haj{\'o}s proved the following statement for systems of linear forms \cite{Hajos}:

\begin{theorem}[Haj\'os]\label{HajThm2}
	Let $L_1,\dots,L_m$ be $m$-dimensional linear forms given by
	\begin{equation*}
		L_j(\bx)=a_{j,1}x_1+\dots+a_{j,m}x_m,\quad \quad (1\leq j \leq m),
	\end{equation*}
	with real coefficients $a_{j,k}$, $1\leq j,k \leq m$. Assume the system of linear forms has determinant $\det(a_{j,k})=\pm 1$ and they satisfy
	\begin{equation}\label{maxeq2}
	\max\left\{|L_1(\bx)|,\dots,|L_m(\bx)|\right\}\geq 1
\end{equation}
	for all integer vectors $\bx=(x_1,\dots,x_m)\neq \0\in\ZZ^m$. Then, one of the linear forms has only integer coefficients.
\end{theorem}

Any matrix $A=(a_{j,k})\in\RR^{m\times m}$ satisfying $\det A\neq 0$ gives rise to a full-rank lattice
\begin{equation*}
	\Lambda(A)=\left\{A\bz:\bz\in\ZZ^m\right\}.
\end{equation*}
Two lattices $\Lambda(A_1)$ and $\Lambda(A_2)$ are identical if and only if there exists a unimodular matrix $B\in\ZZ^{m\times m}$ such that $A_1=A_2 B$. This implies that
\begin{equation*}
	\Lambda(A)=\Lambda(AB)
\end{equation*}
for any matrix $B\in\ZZ^{m\times m}$ with $\det B=\pm 1$. Clearly, $A\bz=AB(B^{-1}\bz)$ and thus, turning back to the linear forms defined through $A$, we see that
\begin{equation*}
	L_j(\bx)=\sum\limits_{k=1}^{\infty}a_{j,k}x_k=\sum\limits_{k=1}^{\infty}(ab)_{j,k}y_k=L_j'(\by),\quad (1\leq j \leq m),
\end{equation*}
where we perform a change of variables, substituting $\by$ for $B^{-1}\bx$, and where $(ab)_{j,k}$ denotes the entries of $AB$. Importantly, any non-zero integer vector $\by$ is the image of a non-zero integer vector $\bx$ under the multiplication by $B^{-1}$. Hence, rather than investigating $L_1(\bx),\dots,L_m(\bx)$ for $\bx\in\ZZ^{m\times m}\setminus\{\0\}$, we can consider a modified collection of linear forms $L'_1(\by),\dots,L'_m(\by)$ at integer vectors $\by\neq\0$. This will prove to be very useful.

For our purpose, we will need the following corollary of Theorem \ref{HajThm2}.

\begin{corollary}\label{HajThm1}
	Let $L_1,\dots,L_m$ be given as in Theorem \ref{HajThm2}. After a possible permutation $\pi$ of the linear forms $L_{j}$, there is an integral linear transformation of determinant $\pm 1$ from the variables $x_1,\dots,x_m$ to $y_1,\dots,y_m$, such that in the new variables $y_{\ell}$ the linear forms have lower triangular form and all diagonal elements are equal to $1$.
\end{corollary}

In other words, we get 
\begin{equation*}
	L_{\pi^{-1}(j)}(\bx)=\phi_{j,1}y_1+\dots+\phi_{j,j-1}y_{j-1}+y_{j},\quad (1\leq j\leq m),
\end{equation*}
with all further coefficients being $0$. According to Haj{\'o}s, Theorem \ref{HajThm2} and Corollary~\ref{HajThm1} are equivalent. However, he does not provide a proof for this equivalence, so we will give the details.

\begin{proof}
The main difficulty is to show that Theorem \ref{HajThm1} follows from Theorem \ref{HajThm2}.
We will assume that Theorem \ref{HajThm2} implies that the first linear form has only integer entries. Then, we need to show the following: Given a unimodular matrix $A\in\mathbb{R}^{m\times m}$ with $a_{1,1},\dots,a_{1,m}\in\ZZ$ satisfying $|A\bx|\geq 1$ for all $\bx\in \ZZ^m\backslash\{\0\}$, there exists a matrix $B\in\ZZ^{m\times m}$ with determinant $\pm 1$ such that $AB$ has lower triangular form with all diagonal entries equal to $1$. Corollary \ref{HajThm1} is trivially true for $m=1$, so, by induction, it is sufficient to show that we can choose $B$ in such a way that $AB$ has first row entries 
\begin{equation*}
	ab_{1,1}=1,\ ab_{1,2}=\dots=ab_{1,m}=0,
\end{equation*}
since then we can apply the statement to the matrix 
\begin{equation*}
	AB^*=(ab_{j,k})_{j,k=2}^{m}\in \RR^{(m-1)\times(m-1)},
\end{equation*}
which satisfies $\det AB^{*}=\det AB=\pm 1$.
We do this in two steps. First we show that we can get $AB$ with first row entries 
\begin{equation*}
	ab_{1,1}=\gcd(a_{1,1},\dots,a_{1,m}),ab_{1,2}=\dots=ab_{1,m}=0
\end{equation*}
and then we conclude that $\gcd(a_{1,1},\dots,a_{1,m})=1$.

Multiplying $A$ by the diagonal matrix $B_0$ with $b_{i,i}=\sgn(a_{1,i})$ gives us a matrix $AB_0$ with only positive entries in the first row and then our plan is to emulate the Euclidean algorithm. We take the last two entries of the first row of $AB_0$ (if non-zero) and keep on subtracting the lesser entry from the greater one until the two of them are equal to $\gcd(a_{1,m-1},a_{1,m})$ and zero. If needed we then swap those entries such that the latter one equals zero. 

\noindent The subtractions correspond to multiplication on the right by matrices of the form
  \begin{align*}
	\begin{bmatrix}
    	I_{m-2} & &  \\
    		& 1 & 0  \\
    		& -1 & 1 & 
	\end{bmatrix}\quad \text{ or }\quad
		\begin{bmatrix}
    	I_{m-2} & &  \\
    		& 1 & -1  \\
    		& 0 & 1 & 
	\end{bmatrix},
  \end{align*}
respectively. Swapping is done via multiplication by a matrix of the form
  \begin{align*}
	\begin{bmatrix}
    	I_{m-2} & &  \\
    		& 0 & 1  \\
    		& 1 & 0 & 
	\end{bmatrix}.
  \end{align*}
We are only interested in the entries of the first row and thus are not concerned how these multiplications affect the other rows. We then continue analogously by doing the same procedure for the next two non-zero entries of the first row, and so on, using adjusted multiplication matrices. Any such operation is done via multiplication by a matrix of determinant $\pm 1$ and in the end we get a matrix $AB$ with first row entries as desired.

Finally, it is easy to see that $\gcd(a_{1,1},\dots,a_{1,m})=1$. Assume this is not the case. Then, it follows that\vspace{-2ex}
\begin{equation*}
	|\det AB^*|=1/\gcd(a_{1,1},\dots,a_{1,m})<1.
\end{equation*}
Hence, by Minkowski's Theorem (see Theorem \ref{Minkowski}), there exists a non-zero vector $\bz^*\in\ZZ^{m-1}$ such that $0<|AB^*\bz^*|<1$. Let $\bz=(0,\bz^*)$. It follows that $0<|AB\bz|<1$, contradicting the assumption \eqref{maxeq2}.

Now we show that Corollary \ref{HajThm1} implies Theorem \ref{HajThm2}. Let $AB$ be as given above. The matrix $B^{-1}$ has only integer entries and so the same is true for the first row of $A=ABB^{-1}$.
\end{proof}

\newpage
We also need to make use of another property of lattices. Given a full-rank lattice $\Lambda=\Lambda(A)\subset\RR^m$, for $j\in\{1,\dots,m\}$ we define its \textit{$j$-th successive minimum} as
\begin{equation*}
	\mu_j=\mu_j(\Lambda):=\inf\left\{r>0:\Lambda\cap \bar{B}(\0,r)\text{ contains }j\text{ linearly independent vectors}\right\},
\end{equation*}
where $\bar{B}$ denotes a closed ball. Clearly, $0<\mu_1\leq\mu_2\leq\dots\leq\mu_m<\infty$. Furthermore, Minkowski proved the following fundamental theorem.

\begin{theorem}[Minkowski's Second Theorem]\label{Thm:Mink2nd}
	\begin{equation*}
		\frac{2^n}{n!}\det A\leq\prod\limits_{j=1}^{m}\mu_j\leq 2^n\det A.
	\end{equation*}
\end{theorem}

In particular, Theorem \ref{Thm:Mink2nd} implies that if we have a lower bound $\mu_1>\eps>0$, then all the successive minima are uniformly bounded by $\mu_j<2^n\eps^{-(m-1)}\det A$.

\subsection{Results needed to prove Theorem \ref{Con1}}

We will be using the following standard result in measure theory.

\begin{theorem}[Lebesgue Density Theorem]\label{LDT}
	Let $\mu$ be a Radon measure on $\RR^n$. If $A\subset\RR^n$ is $\mu$-measurable, then the limit
	\begin{equation}\label{Eqn:LDT}
		\lim\limits_{r\rightarrow\infty}\frac{\mu\left(A\cap B_r(\bx)\right)}{\mu\left( B_r(\bx)\right)}
	\end{equation}
	exists and equals $1$ for $\mu$-almost all $\bx\in A$ and equals $0$ for $\mu$-almost all $\bx\in\RR^m\backslash\{A\}$.
\end{theorem}

Theorem \ref{LDT} can be found in \cite{Harman}. Importantly, it does not depend on the choice of metric, so we can apply it to the distance $d$ as defined below. In our case, the measure $\mu$ is simply the Lebesgue measure $\lambda_n$.

Given a vector $\bi=(i_1,\dots,i_n)\in\I^n$, let $i_{min}\coloneqq\min\limits_{1\leq j\leq n}i_j$. Then $\bar{i}_j\coloneqq\frac{i_{min}}{i_j}\leq 1$ for $1\leq j\leq n$ and so $d_j(u,v)\coloneqq |u-v|^{\bar{i}_j}$ is a metric on $\RR$ for $1\leq j\leq n$. Thus,
\begin{equation}\label{Eqn:DefMetric}
	d(\bx,\by)\coloneqq\max\limits_{1\leq j\leq n}d_j(x_j,y_j)
\end{equation}
is a metric on $\RR^n$. We will refer to balls with respect to the metric $d$ as $d$-balls and denote by $B_r^d(\bx)$ a $d$-ball centred at $\bx\in\RR^n$ of radius $r$. The following generalisation of a result in \cite{VBSV} from balls with respect to $\max$-norm to $d$-balls will be used to prove Theorem \ref{Con1}. For better readability within the proof, we will use the notation~$|\cdot |$ to refer to the $n$-dimensional Lebesgue measure $\lambda_n$ of a set.

\begin{theorem}\label{ConBV}
	Let $(A_k)_{k\in\NN}$ be a sequence of $d$-balls in $\RR^n$ with $|A_k|\rightarrow 0$ as $k\rightarrow\infty$. Let $(U_k)_{k\in\NN}$ be a sequence of Lebesgue measurable sets such that $U_k\subset A_k$ for all $k$. Assume that, for some $c>0$, $|U_k|>c|A_k|$ for all $k$. Then the sets
	\begin{equation*}
		\mathcal{U}=\limsup\limits_{k\rightarrow\infty}U_k\coloneqq \bigcap\limits_{j=1}^{\infty}\bigcup\limits_{k=j}^{\infty}U_k\quad \text{ and }\quad \mathcal{A}=\limsup\limits_{k\rightarrow\infty}A_k\coloneqq \bigcap\limits_{j=1}^{\infty}\bigcup\limits_{k=j}^{\infty}A_k
	\end{equation*}
	have the same Lebesgue measure.
\end{theorem}

\begin{remark*}
	Given a constant $c>0$, we already know from the weighted version of Khintchine's Theorem that the sets $W_n(\bi,\psi)$ and $W_n(\bi,c\psi)$ have the same Lebesgue measure (see Theorem \ref{Thm:WeightedKhin}). Theorem \ref{ConBV} implies that this property is shared by more general $\limsup$ sets. 
\end{remark*}

\begin{proof}
	Let $\mathcal{U}_j\coloneqq \bigcup_{k\geq j}U_k$ and $\mathcal{C}_j\coloneqq \mathcal{A}\backslash\mathcal{U}_j$. Then $\mathcal{U}_j\supset\mathcal{U}_{j+1}$ and $\mathcal{C}_j\subset\mathcal{C}_{j+1}$. Define
	\begin{equation*}
		\mathcal{C}\coloneqq\mathcal{A}\backslash\mathcal{U}=\mathcal{A}\backslash\bigcap\limits_{j=1}^{\infty}\mathcal{U}_j=\bigcup\limits_{j=1}^{\infty}\left(\mathcal{A}\backslash\mathcal{U}_j\right)=\bigcup_{j=1}^{\infty}\mathcal{C}_j.
	\end{equation*}
	We are to show that $\mathcal{C}$ has measure zero or, equivalently, that every $\mathcal{C}_j$ has measure zero.
	
	Assume the contrary. Then there is an $\ell\in\NN$ such that $|\mathcal{C}_\ell|>0$ and therefore there is a density point $\bx_0$ of $\mathcal{C}_\ell$, i.e. a point $\bx_0$ for which the limit in \eqref{Eqn:LDT} is equal to $1$. Since $\bx_0\in\mathcal{A}$, we know that $\bx_0\in A_{j_k}$ for a sequence $(j_k)_{k\in\NN}$. As $|A_{j_k}|$ tends to zero, we can conclude that
	\begin{equation}\label{Eqn:Claim}
		|\mathcal{C}_{\ell}\cap A_{j_k}|\sim|A_{j_k}|,\quad\text{as }k\rightarrow\infty,
	\end{equation}
which is shown by the following considerations. 

If $A_{j_k}$ is a $d$-ball of radius $r_{j_k}$ containing $\bx_0$, then $A_{j_k}$  will be contained in a $d$-ball $B^d_{2r_{j_k}}(\bx_0)$. Indeed, doubling the radius corresponds to extending the $j$-th side length by the factor
\begin{equation*}
	2^{1/\bar{i}_j}=2^\frac{i_j}{i_{min}}\geq 2,\quad j\in\{1,\dots,n\}.
\end{equation*}
Comparing Lesbegue measures, it follows that 
\begin{equation}\label{Eqn:LDT1}
	\frac{|B^d_{2r_{j_k}}(\bx_0)|}{|A_{j_k}|}=2^s, \text{\qquad where \qquad} s\coloneqq\sum\limits_{j=1}^n\frac{i_j}{i_{min}},
\end{equation}
and thus
\begin{equation}\label{Eqn:LDT2}
	\frac{|B^d_{2r_{j_k}}(\bx_0)\backslash A_{j_k}|}{|B^d_{2r_{j_k}}(\bx_0)|}=1-\frac{1}{2^s},
	\vspace{2ex}
\end{equation}
since $A_{j_k}$ is fully contained in $B^d_{2r_{j_k}}(\bx_0)$. The Lebesgue density theorem tells us that for any $\varepsilon>0$ and $\delta$ small enough,
\begin{equation}\label{Eqn:LDT3}
	\frac{|\mathcal{C}_{\ell}\cap B^d_{\delta}(\bx_0)|}{|B^d_{\delta}(\bx_0)|}>1-\varepsilon.
\end{equation}
Combining the identities \eqref{Eqn:LDT1}, \eqref{Eqn:LDT2} and \eqref{Eqn:LDT3}, we see that
\begin{align*}
	\frac{|\mathcal{C}_{\ell}\cap A_{j_k}|}{|A_{j_k}|}&=\frac{|\mathcal{C}_{\ell}\cap A_{j_k}|}{|B^d_{2r_{j_k}}(\bx_0)|}\frac{|B^d_{2r_{j_k}}(\bx_0)|}{|A_{j_k}|}\\[1ex]
	&\geq\frac{|\mathcal{C}_{\ell}\cap B^d_{2r_{j_k}}(\bx_0)|-|B^d_{2r_{j_k}}(\bx_0)\backslash A_{j_k}|}{|B^d_{2r_{j_k}}(\bx_0)|}\frac{|B^d_{2r_{j_k}}(\bx_0)|}{|A_{j_k}|}\\[1ex]
	&=\left(\frac{|\mathcal{C}_{\ell}\cap B^d_{2r_{j_k}}(\bx_0)|}{|B^d_{2r_{j_k}}(\bx_0)|}-\frac{|B^d_{2r_{j_k}}(\bx_0)\backslash A_{j_k}|}{|B^d_{2r_{j_k}}(\bx_0)|}\right)\frac{|B^d_{2r_{j_k}}(\bx_0)|}{|A_{j_k}|}\\[1ex]
	&>\left(1-\varepsilon-\left( 1-\frac{1}{2^s}\right)\right)2^s\\[1ex]
	&=1-2^s\varepsilon,\quad \text{ for } r_{j_k} \text{ small enough}.
\end{align*}
The value of $\eps$ can be chosen to be arbitrarily small by \eqref{Eqn:LDT3}. Hence, this quotient tends to~$1$ as $k\rightarrow\infty$, which proves \eqref{Eqn:Claim}.

Since $\mathcal{C}_j\supset\mathcal{C}_{\ell}$ for all $j\geq \ell$, it follows that
\begin{equation}\label{cont}
	|\mathcal{C}_{j_k}\cap A_{j_k}|\sim|A_{j_k}|\quad \text{ as }k\rightarrow\infty.
\end{equation}
On the other hand, by definition, $\mathcal{C}_{j_k}\cap U_{j_k}=\emptyset$. Using that $|U_k|>c|A_k|$ for all $k$, we get that
\begin{equation*}
	|A_{j_k}|\geq |U_{j_k}|+|\mathcal{C}_{j_k}\cap A_{j_k}|\geq c|A_{j_k}|+|\mathcal{C}_{j_k}\cap A_{j_k}|,
\end{equation*}
and thus 
\begin{equation*}
	|\mathcal{C}_{j_k}\cap A_{j_k}|<(1-c)|A_{j_k}|
\end{equation*}
for $k$ sufficiently large. This is a contradiction to \eqref{cont}. Thus, every set $\mathcal{C}_j$ has zero Lebesgue measure, which completes the proof.
\end{proof}

The final auxiliary result is due to Cassels \cite{Cassels}. It relates homogeneous and inhomogeneous approximation properties of linear forms.

\begin{theorem}\label{Cassels}
	Let $L_1,\dots,L_{\ell}$ be linear forms in the $\ell$ variables $\bz=(z_1,\dots,z_\ell)$ given by
	\begin{equation*}
		L_k(\bz)=a_{k,1}z_1+\dots+a_{k,l}z_l,\quad (1\leq k\leq \ell),
	\end{equation*}		
	with real coefficients $a_{k,j}$, $1\leq k,j\leq \ell$. Assume the system of linear forms has determinant $\Delta=\det(a_{k,j}) \neq 0$ and suppose that the only integer solution of
	\begin{equation}\label{maxeq4}
		\max\limits_{1\leq k\leq \ell}|L_k(\bz)|<1
	\end{equation}
	is $\bz=\0$. Then, for all real vectors $\boldsymbol{\gamma}=(\gamma_1,\dots,\gamma_\ell)\in\I^{\ell}$, there are integer solutions of 
	\begin{equation*}
		\max\limits_{1\leq k\leq \ell}|L_k(\bz)-\gamma_k|<\frac{1}{2}(\lfloor|\Delta|\rfloor+1),
	\end{equation*}
	where $\lfloor\cdot\rfloor$ denotes the integer part of a real number.
\end{theorem}

Essentially, Theorem \ref{Cassels} tells us the following. If the values of a collection of linear forms taken at integer points are bounded away from $\0$, then these linear forms can uniformly approximate any inhomogeneous constant $\bgamma$. This will be used to deduce twisted approximation properties of non-singular vectors, see Theorem \ref{Thm1}.

\section{Proof of Theorem \ref{ThmDS}}

We will show that any vector not contained in $\cD_n(\bi)$ cannot be badly approximable. If $\balpha=(\alpha_1,\dots,\alpha_n)$ is not in $\cD_n(\bi)$, then for any $c<1$ there exists an infinite sequence of integers $Q$ such that \eqref{DI} has no solution $q<Q$. This implies that there is a sequence $(Q_k)_{k\in\NN}$ such that 
\begin{equation*}
	\norm{q\alpha_j}>(1-\frac{1}{2^k})Q_k^{-i_j}
\end{equation*}
for all $q\leq Q_k$ and some $j\in \{1,\dots,n\}$. In other words, we have
\begin{equation}\label{maxeq1}
	\max\left\{Q_k^{-1}q,Q_k^{i_1}|q\alpha_1+p_1|,\dots,Q_k^{i_n}|q\alpha_n+p_n|\right\}>1-\frac{1}{2^k}
\end{equation}
for all integers $q,p_1,\dots,p_n$ not all equal to $0$. The $n+1$ linear forms 
\begin{equation*}
	Q_k^{-1}q,\ Q_k^{i_1}(q\alpha_1+p_1),\dots,\ Q_k^{i_n}(q\alpha_n+p_n)
\end{equation*}
define a lattice of determinant $1$ in $(n+1)$-dimensional space. By \eqref{maxeq1}, this lattice has no non-zero point within distance $1/2$ of the origin. By Theorem \ref{Thm:Mink2nd}, such a lattice has a basis of $n+1$ points, the coordinates of which are all bounded from above by a numerical constant. This implies that there exists a linear transformation with integral coefficients and determinant $1$ from $q,p_1,\dots,p_n$ to $x_0,\dots,x_n$ such that
\begin{align}\label{trafo1}
	\begin{cases}
		Q_k^{-1}q&=\ \theta_{0,0}^{(k)}x_0+\theta_{0,1}^{(k)}x_1+\dots+\theta_{0,n}^{(k)}x_n,\\
		Q_k^{i_1}(q\alpha_1+p_1)&=\ \theta_{1,0}^{(k)}x_0+\theta_{1,1}^{(k)}x_1+\dots+\theta_{1,n}^{(k)}x_n,\\
		\vdots\\
		Q_k^{i_n}(q\alpha_n+p_n)&=\ \theta_{n,0}^{(k)}x_0+\theta_{n,1}^{(k)}x_1+\dots+\theta_{n,n}^{(k)}x_n,
	\end{cases}
\end{align}
where the absolute values of the $\theta_{\ell,m}^{(k)}$ are bounded by a uniform constant $C$ for all $\ell,m,k$. The transformations depend on $k$, but, for each $k$, the determinant of $(\theta_{\ell,m}^{(k)})$ is equal to $1$. Define the matrices $\Theta_k=(\theta_{\ell,m}^{(k)})$ for $k\in\NN$. The sequence $(\Theta_k)_{k\in\NN}$ is contained in a compact subset of $\SL_n(\RR)$. This implies there is a subsequence~$(\kappa)$ of values of $(k)_{k\in\NN}$ such that $(\Theta_{\kappa})_{\kappa}$ converges to an element of $\SL_n(\RR)$. We denote this limit by $\Theta=(\theta_{\ell,m})$ and get the linear forms
\begin{align}\label{trafo2}
	\begin{cases}
		X_0&=\ \theta_{0,0}x_0+\theta_{0,1}x_1+\dots+\theta_{0,n}x_n\\
		X_1&=\ \theta_{1,0}x_0+\theta_{1,1}x_1+\dots+\theta_{1,n}x_n\\
		\vdots\\
		X_n&=\ \theta_{n,0}x_0+\theta_{n,1}x_1+\dots+\theta_{n,n}x_n
	\end{cases}
\end{align}
of determinant $1$ with the property that 
\begin{equation*}
	\max\left\{|X_0|,|X_1|,\dots,|X_n|\right\}\geq 1
\end{equation*}
for all integer vectors $(x_0,\dots,x_n)\neq (0,\dots,0)$. Indeed, if there was a non-zero tuple $(x_0^*,\dots,x_n^*)$ satisfying 
\begin{equation*}
	\max\left\{|X_0|,|X_1|,\dots,|X_n|\right\}<1,
	\vspace{1ex}
\end{equation*}
then putting $(x_0^*,\dots,x_n^*)$ in \eqref{trafo1} would violate \eqref{maxeq1} for large enough values of $k$. By Corollary \ref{HajThm1}, after a possible integral transformation of determinant $\pm 1$, we get 
\begin{equation*}
	X_{\pi^{-1}(\ell)}=\phi_{\ell 0}y_0+\phi_{\ell 1}y_1+\dots+y_{\ell},\quad (0\leq \ell\leq n),
\end{equation*}
with all other coefficients being equal to zero. Independent of the permutation $\pi$, it is always possible to satisfy either 
\begin{equation*}
	X_0=0\quad\quad \text{ or }\quad\quad X_1=\dots=X_n=0
\end{equation*}
with the non-zero integer vector 
\begin{equation*}
	(y_0,\dots,y_{n-1},y_n)=(0,\dots,0,1).
\end{equation*}
Hence, the same is true of the linear forms in $\eqref{trafo2}$ with integers $x_0,\dots,x_n$ not all equal to $0$. On substituting these into \eqref{trafo1}, we obtain, for any $\eps>0$ on taking $k$ sufficiently large, either a solution of
\begin{equation*}
	Q_k^{-1}q<C_{1},\ Q_k^{i_1}|q\alpha_1+p_1|<\eps,\dots,\ Q_k^{i_n}|q\alpha_n+p_n|<\eps,
\end{equation*}
with $C_{1}$ independent of $\eps$, or a solution to
\begin{equation*}
	Q_k^{-1}q<\eps,\ Q_k^{i_1}|q\alpha_1+p_1|<C_{1},\dots,\ Q_k^{i_n}|q\alpha_n+p_n|<C_{1}.
\end{equation*}
In the first case, substituting $N_k$ for $C_1Q_k$ shows the existence of $q\in\NN$ such that
\begin{equation*}
	q^{i_j}\norm{q\alpha_j}\leq N_k^{i_j}\norm{q\alpha_j}<\eps C_1^{i_j}\leq \eps C_1,\quad (1\leq j\leq n).
\end{equation*}
In the second case, substituting $N_k$ for $\eps Q_k$ gives us a $q\in\NN$ satisfying
\begin{equation*}
	q^{i_j}\norm{q\alpha_j}\leq N_k^{i_j}\norm{q\alpha_j}<\eps^{i_j} C_1\leq \eps^{i_{min}} C_1,\quad (1\leq j\leq n).
\end{equation*}
The constant $\eps>0$ was chosen to be arbitrarily small. Hence, in both cases, $\balpha$ cannot be $\bi$-badly approximable. This completes the proof of Theorem \ref{ThmDS}.

\newpage

\section{Proof of Theorem \ref{Con1}}

Theorem \ref{ConBV} is the main ingredient used in the proof of Theorem \ref{Con1}. Hence, we start by preparing the ground for applying \ref{ConBV}.

If $\balpha=(\alpha_1,\dots,\alpha_n)$ is not in $\cS_n(\bi)$, then there exists an $\varepsilon(\balpha)\in(0,1)$ such that the equation
\begin{equation*}
	\norm{q\alpha_j}<\eps(\balpha)Q_k^{-i_j},\quad j\in \{1,\dots,n\}
	\vspace{1ex}
\end{equation*}
has no integer solution $q\leq Q_k$ for an infinite increasing sequence $(Q_k)_{k\in\NN}$. In other words, there is a sequence $(Q_k)_{k\in\NN}$ such that $\norm{q\alpha_j}>\varepsilon(\balpha)Q_k^{-i_j}$ for some $j\in \{1,\dots,n\}$ for each $q\leq Q_k$. This implies that
\begin{equation}\label{maxeq3}
	\max\left\{Q_k^{-1}q,\varepsilon(\balpha)^{-1} Q_k^{i_1}|q\alpha_1+p_1|,\dots,\varepsilon(\balpha)^{-1} Q_k^{i_n}|q\alpha_n+p_n|\right\}\geq 1
\end{equation}
for all integers $q,p_1,\dots,p_n$ not all equal to $0$.

For any fixed $Q_k$, the $n+1$ linear forms appearing in \eqref{maxeq3} in the $n+1$ variables $q,p_1,\dots,p_n$ form a system of linear forms of determinant $\varepsilon(\balpha)^{-n}$ satisfying condition \eqref{maxeq4}. Hence, by Theorem \ref{Cassels}, for any $Q_k$ and any $\bgamma=(\gamma_1,\dots,\gamma_n)\in \I^n$, there exists a non-zero integer vector $(q,p_1,\dots,p_n)$ satisfying
\begin{equation*}
	\max\left\{Q_k^{-1}q,\varepsilon(\balpha) Q_k^{i_1}\left|q\alpha_1+p_1-\frac{\gamma_1}{\varepsilon(\balpha)}\right|,\dots,\varepsilon(\balpha) Q_k^{i_n}\left|q\alpha_n+p_n-\frac{\gamma_n}{\varepsilon(\balpha)}\right|\right\}< \delta,
	\vspace{1ex}
\end{equation*}
where $\delta=\frac{1}{2}(\lfloor|\varepsilon(\balpha)^{-n}|\rfloor+1)$. Substituting $\tilde{Q}_k$ for $\delta Q_k$ and $\tilde{\gamma}_j$ for $\gamma_j/\varepsilon(\balpha)$ implies that there exists a positive integer solution $q$ to the system of inequalities
\begin{equation*}
	\begin{aligned}
		\norm{ q\alpha_j-\tilde{\gamma}_j} &<\delta^{1+i_j}\varepsilon(\balpha) \tilde{Q}_k^{-i_j},\quad (1\leq j\leq n),\\[1ex]
		q &< \tilde{Q}_k.
	\end{aligned}
	\vspace{-2ex}
\end{equation*}
By letting
\begin{equation*}
	C(\balpha)=\max\limits_{1\leq j\leq n}\delta^{1+i_j}\varepsilon(\balpha),
	\vspace{1ex}
\end{equation*}
we have proved the following statement.
\newpage

\begin{samepage}
\begin{theorem}\label{Thm1}
	Let $\balpha\notin\cS_n(\bi)$. Then, all $\bbeta \in \I^n$ are uniformly $\bi$-approximable by the sequence $(q\balpha)_{q\in\NN}$, i.e. there exists a constant $C(\balpha)>0$ such that for all $\bbeta \in \I^n$ there are infinitely many $q\in\NN$ satisfying 
	\begin{equation*}
		\max\limits_{1\leq j\leq n}q^{i_j}\norm{q\alpha_j-\beta_j}<C(\balpha).
	\end{equation*}		
\end{theorem}
\end{samepage}

In other words, every point $\bbeta\in \I^n$ is contained in infinitely many sets of the form
\begin{equation*}
	\prod\limits_{j=1}^n[q\alpha_j-p_j-C(\balpha)q^{-i_j},q\alpha_j-p_j+C(\balpha)q^{-i_j}].
\end{equation*}
These sets are not $d$-balls, but, if we extend them slightly, we get that
\begin{equation}\label{LSup}
	\bigcap\limits_{k=1}^{\infty}\bigcup\limits_{q=k}^{\infty}\left(\bigcup\limits_{|\bp|<q}A^{\balpha}_{\bp,q}\right)=\I^n,
\end{equation}
where our $\limsup$ set is built from the sets
\begin{equation}\label{Eqn:Sets}
	A^{\balpha}_{\bp,q}:=\prod\limits_{j=1}^n[q\alpha_j-p_j-C(\balpha)^{\frac{i_j}{i_{min}}}q^{-i_j},q\alpha_j-p_j+C(\balpha)^{\frac{i_j}{i_{min}}}q^{-i_j}]
\end{equation}
with $q\in\NN$ and $\bp\in\ZZ^n$. By setting $r(\balpha,q)=C(\balpha)q^{-i_{min}}$ it follows that 
\begin{equation*}
A^{\balpha}_{p,q}=B^d_{r(\balpha,q)}(q\balpha-\bp).
\end{equation*}
Indeed, if we recall the definition of the metric $d$ given in \eqref{Eqn:DefMetric}, we see that
\begin{align*}
	d_j(u,v)\leq C(\balpha)q^{-i_{min}}&\Leftrightarrow |u-v|^{\bar{i}_j}\leq C(\balpha)q^{-i_{min}}\\[1ex]
	&\Leftrightarrow |u-v|\leq (C(\balpha)q^{-i_{min}})^{\frac{i_j}{i_{min}}}\\[1ex]
	&\Leftrightarrow |u-v|\leq C(\balpha)^{\frac{i_j}{i_{min}}}q^{-i_j}.
\end{align*}
As we are dealing with a countable collection of $d$-balls of the form $A^{\balpha}_{\bp,q}$, it is possible to rewrite them as a sequence $(A_k)_{k\in\NN}$ and \eqref{LSup} is equivalent to the statement that \vspace{-1ex}
\begin{equation*}
	\limsup_{k\rightarrow\infty}A_k=\I^n.
\end{equation*}
This means we are in a situation where we can apply Theorem \ref{ConBV}.

\begin{proof}[Proof of Theorem \ref{Con1}]
	We are to show that, for any given $\eps>0$, the $\limsup$ set
	\begin{equation*}
		\bigcap\limits_{k=1}^{\infty}\bigcup\limits_{q=k}^{\infty}\left(\bigcup\limits_{|\bp|<q}U^{\balpha}_{\bp,q}(\eps)\right)
	\end{equation*}
	has full Lebesgue measure, where
	\begin{equation*}
		U^{\balpha}_{\bp,q}(\eps)=\prod\limits_{j=1}^n[q\alpha_j-p_j-\eps q^{-i_j},q\alpha_j-p_j+\eps q^{-i_j}].
	\end{equation*}
	For $\eps$ small enough, any such set $U^{\balpha}_{\bp,q}(\eps)$ is contained in a set $A^{\balpha}_{\bp,q}$ as defined by \eqref{Eqn:Sets}. Furthermore,\vspace{2ex}
	\begin{equation*}
		\frac{|U^{\balpha}_{\bp,q}(\eps)|}{|A^{\balpha}_{\bp,q}|}=\frac{2^n\eps^n q^{-1}}{2^n C(\balpha)^{\sum\limits_{j=1}^n\frac{i_j}{i_{min}}}q^{-1}}=\frac{\eps^n}{C(\balpha)^s},
		\vspace{2ex}
	\end{equation*}
	a ratio which does not depend on $\bp$ or $q$. Thus the conditions for Theorem \ref{ConBV} are satisfied and using \eqref{LSup} we can conclude that
	\begin{equation*}
		\left|\bigcap\limits_{k=1}^{\infty}\bigcup\limits_{q=k}^{\infty}\left(\bigcup\limits_{|\bp|<q}U^{\balpha}_{\bp,q}(\eps)\right)\right|=
		\left|\bigcap\limits_{k=1}^{\infty}\bigcup\limits_{q=k}^{\infty}\left(\bigcup\limits_{|\bp|<q}A^{\balpha}_{\bp,q}\right)\right|=1.
		\vspace{2ex}
		\qedhere
	\end{equation*} 
\end{proof}

\newpage
 
\section{Future development}

	It is worth noting that Theorem \ref{Thm1} has been shown to be a statement of equivalence in the non-weighted case.
	\begin{theorem}[Theorem V.XIII in \cite{Cassels}]\label{Thm:Casselsnonsing}
		Let $\balpha\in \I^n$. Then, $\balpha$ is non-singular if and only if all $\bbeta \in \I^n$ are uniformly approximable by the sequence $(q\balpha)_{q\in\NN}$, i.e. there exists a constant $C(\balpha)>0$ such that for all $\bbeta \in \I^n$ there are infinitely many $q\in\NN$ satisfying 
		\begin{equation*}
			\max\limits_{1\leq j\leq n}\norm{q\alpha_j-\beta_j}<C(\balpha)q^{-1/n}.
			\vspace{2ex}
		\end{equation*}				
	\end{theorem}
	This is shown through the construction of a vector $\bbeta\in\I^n$ for which the inequality \vspace{-2ex}
	\begin{equation*}
		\max\limits_{1\leq j\leq n}\norm{q\alpha_j-\beta_j}<Cq^{-1/n}
	\end{equation*}
	has only finitely many solutions $q\in\NN$ for any given constant $C>0$. This also implies that there is no analogous statement to Dirichlet's Theorem for an arbitrary inhomogeneous constant, as discussed in Remark \ref{Rmk:InhomDir}. We are confident that Theorem~\ref{Thm:Casselsnonsing} can be extended to the weighted case. So far we have not finalised a proof.
	
	Another interesting question in both the standard and weighted case is whether Theorem \ref{Con1} is actually true if and only if $\balpha$ is non-singular. This would be a strengthening of Theorem \ref{Thm:Casselsnonsing} and give us the following Kurzweil-type statement:
	\begin{Conjecture}
		For $\eps>0$, denote by $\psi_{\eps}$ the approximating function given by $\psi_{\eps}(q)=\eps q^{-1}$. Then
		\begin{equation*}
			\bigcap\limits_{\eps>0}W^{\times}_n(\bi,\psi_{\eps})=\I^n\setminus \cS_n(\bi).
		\end{equation*}
	\end{Conjecture}

\end{document}